\documentclass[a4paper,11pt]{article}

\usepackage{hyperref}
\usepackage{appendix}
\usepackage{amsmath}
\usepackage{amsthm}
\usepackage{amssymb}
\usepackage{amsfonts}
\usepackage{bbm}
\usepackage[svgnames]{xcolor}
\usepackage{tikz}
\usepackage{braket}
\usepackage{verbatim}
\usepackage{authblk}
\usepackage{fullpage}
\usepackage{enumitem}
\usepackage{diagrams}
\usepackage{subcaption}
\usepackage{makecell}
\usepackage{subcaption}
\usepackage{mathdots}
\usepackage{fancyhdr}
\usepackage{float}
\floatstyle{boxed}
\restylefloat{figure}

\newarrow{Implies}{=}{=}{=}{=}{=>}
\usetikzlibrary{decorations.markings}
\usetikzlibrary{shapes.geometric}
\pgfdeclarelayer{edgelayer}
\pgfdeclarelayer{nodelayer}
\pgfsetlayers{edgelayer,nodelayer,main}
\tikzstyle{none}=[inner sep=0pt]
\tikzstyle{rn}=[circle,fill=Red,draw=Black,line width=0.8 pt]
\tikzstyle{gn}=[circle,fill=Lime,draw=Black,line width=0.8 pt]
\tikzstyle{yn}=[circle,fill=Yellow,draw=Black,line width=0.8 pt]
\tikzstyle{simple}=[-,draw=Black,line width=2.000]
\tikzstyle{arrow}=[-,draw=Black,postaction={decorate},decoration={markings,mark=at position .5 with {\arrow{>}}},line width=2.000]
\tikzstyle{tick}=[-,draw=Black,postaction={decorate},decoration={markings,mark=at position .5 with {\draw (0,-0.1) -- (0,0.1);}},line width=2.000]

\newcounter{jamiecomment}
\newcounter{dominiccomment}
\theoremstyle{plain}
\newtheorem{theorem}{Theorem}[section]
\newtheorem{proposition}[theorem]{Proposition}
\newtheorem{corollary}[theorem]{Corollary}
\newtheorem{lemma}[theorem]{Lemma}

\theoremstyle{definition}

\newtheorem{definition}[theorem]{Definition}
\newtheorem{remark}[theorem]{Remark}
\newtheorem{example}[theorem]{Example}

\def\Id{\ensuremath{\mathrm{Id}}}
\def\id{\ensuremath{\mathrm{id}}}
\def\Hom{\ensuremath{\mathrm{Hom}}}
\begin{document}

\title{Coherence for braided and symmetric pseudomonoids}
\author{Dominic Verdon}
\affil{School of Mathematics, University of Bristol}
\maketitle

\begin{abstract}
Computads for unbraided, braided, and symmetric pseudomonoids in semistrict monoidal bicategories are defined. Biequivalences characterising the monoidal bicategories generated by these computads are proven. It is shown that these biequivalences categorify results in the theory of monoids and commutative monoids, and generalise the standard coherence theorems for braided and symmetric monoidal categories to braided and symmetric pseudomonoids in any weak monoidal bicategory.
\end{abstract}
\section{Introduction}

\subsection{Overview}

\paragraph{Braided and symmetric pseudomonoids.} Naked, braided and symmetric pseudomonoids are categorifications of noncommutative and commutative monoids, obtained by replacing equality with coherent isomorphism. In the symmetric monoidal bicategory {\bf Cat} of categories, functors and natural transformations, such structures are precisely naked, braided and symmetric monoidal categories. Naked, braided and symmetric pseudomonoids are more general, however, as they can be defined in any monoidal bicategory with the requisite braided structure~\cite{McCrudden2000}. 

Pseudomonoids in braided and symmetric monoidal bicategories arise in a variety of mathematical settings. By categorification of the representation theory of Hopf algebras, bicategories encoding the data of a four-dimensional topological field theory can be obtained as representation categories of certain `Hopf' pseudomonoids~\cite{Neuchl1997,Crane2011}. Pseudomonoids have also appeared recently in the the theory of surface foams, where certain pseudomonoids in a braided monoidal category represent knotted foams in four-dimensional space~\cite{Carter2011}. Many properties of monoidal categories can be formulated externally as structures on pseudomonoids; for instance, Street showed that Frobenius pseudomonoids correspond to star-autonomous categories~\cite{Street2004}, giving rise to a diagrammatic calculus for linear logic~\cite{Dunn2016}. Furthermore, it has been shown that the three-dimensional cobordism category is a symmetric monoidal bicategory generated from the data of a certain Frobenius pseudomonoid~\cite{Bartlett2014a}.

Given these structures' recent appearances in algebra and topology, it is natural to ask whether the well-known coherence theorems for naked, braided and symmetric monoidal categories~\cite{MacLane1978} can be extended in general to naked, braided and symmetric pseudomonoids. For naked pseudomonoids in an naked monoidal bicategory, this question was answered in the affirmative by Lack \cite{Lack2000}. Lack's result, however, does not apply in braided and symmetric monoidal bicategories, or to braided and symmetric pseudomonoids. In this work we solve this problem by proving coherence theorems for naked, braided and symmetric pseudomonoids in fully weak braided and symmetric monoidal bicategories. 

\paragraph{Our approach to coherence.}
In non-higher algebra, algebraic theories are commonly treated using PROs, PROBs and PROPs (collectively, PROs). These are naked, braided and symmetric monoidal categories whose objects are natural numbers, and whose morphisms are specified by generators and relations (a \emph{computad}). For instance, the monoid PRO has two generating 1-cells, $m: 2 \to 1$ and $u: 0 \to 1$, and one of its generating equalities is associativity, $m \circ (m \otimes \id) = m \circ (\id \otimes m)$. Models of the theory in a category of interest are precisely functors from the PRO. One approach to understanding an  algebraic theory is to find an isomorphism between the PRO, defined by a computad, and some simpler combinatorial category. For naked and commutative monoids, these isomorphisms with combinatorial categories have been found by other authors~\cite{Day1997,Davydov2010,Pirashvili2002} and are summarised in Table~\ref{tbl:monoidsresultsinintro}. 

For higher algebraic theories, we take the same approach. Our higher PROs are naked, braided and symmetric monoidal bicategories\footnote{We could not treat the sylleptic case here, due to the lack of a coherence theorem for sylleptic monoidal bicategories.} generated from the computad for a naked, braided or symmetric pseudomonoid. Models in a bicategory of interest correspond to strict naked, braided or symmetric monoidal bifunctors from the higher PRO. Our coherence results are biequivalences between these higher PROs and certain simpler combinatorial 2-categories.

\paragraph{Our results.}
Pseudomonoids are weakenings of monoids, with identical 0- and 1-cell data, and equalities of 1-cells replaced with coherent 2-isomorphisms. Our combinatorial 2-categories are identical at the level of 0- and 1-cells to the combinatorial categories appearing in the theory of monoids.  In all but one of the cases we consider, we show that the categorification adds no additional data: the combinatorial 2-category is \emph{locally discrete}, that is, it has only identity 2-cells. Here we say that `all diagrams commute'. 

The case where not all diagrams commute is that of a braided pseudomonoid in a symmetric monoidal bicategory, where the biequivalent combinatorial category is a categorification ${\bf FS}^{\text{br}}$ of the category of finite sets and functions, whose objects are natural numbers, whose 1-cells $f:m \to n$ are functions $\{1,\dots,m\} \to \{1,\dots,n\}$, and all of whose 2-cells are endomorphisms $\alpha:f \to f$, corresponding to elements of a product of pure braid groups based on $f$.
\begin{table}
\centering

\begin{tabular}{|l|c|c|}
\hline
\parbox{2cm}{{\bf Monoidal category}}  &  \multicolumn{2}{c|}{{\bf Monoid signature}}  \\
	&	\multicolumn{1}{c}{Naked} &  \multicolumn{1}{c|}{Commutative}  \\
\cline{2-3}
\parbox{2cm}{Naked \\(PRO)} & \parbox{5cm}{\center{${\bf \Delta}$} \\ Morphisms are monotone functions $\underline{m} \to \underline{n}$ \vspace{2pt}} & N/A  \\
\cline{2-3}
\parbox{2cm}{Braided \\(PROB)} & \parbox{5cm}{\center{${\bf B\Delta}$} \\ Morphisms are pairs of a monotone function $\underline{m} \to \underline{n}$ and an element of the braid group $B_m$~\cite{Day1997} \vspace{2pt}} & \parbox{5cm}{\center{${\bf B\Delta/{\raise.17ex\hbox{$\scriptstyle{\sim}$}}}$} \\ Morphisms are pairs of a monotone function $\underline{m} \to \underline{n}$ and an element of a quotient of the braid group $B_m$~\cite{Davydov2010} \vspace{2pt}}  \\
\cline{2-3}
\parbox{2cm}{Symmetric \\(PROP)} & \parbox{5cm}{\center{${\bf S\Delta}$} \\ Morphisms are pairs of a monotone function $\underline{m} \to \underline{n}$ and an element of the symmetric group $S_m$~\cite{Day1997} \vspace{2pt}}& \parbox{5cm}{\vspace{-.9cm}\center{${\bf FS}$} \\ Morphisms are functions $\underline{m} \to \underline{n}$~\cite{Davydov2010,Pirashvili2002} \vspace{2pt}} \\
\hline
\end{tabular}
\caption{The table describes the combinatorial category isomorphic to a given PRO. All these categories have natural numbers as objects, so what are described in the table are the morphisms $m \to n$. Here $\underline{n}$ is the set $\{1, \dots, n\}$. Also displayed in the table is our notation for the combinatorial category.}
\label{tbl:monoidsresultsinintro}
\end{table}

These results are summarised in Table~\ref{tbl:pseudomonoidsintro}. We show that, in the special case of naked, braided and symmetric pseudomonoids in the symmetric monoidal bicategory {\bf Cat}, these biequivalences imply the classical coherence results of MacLane.
\begin{table}
\centering
\begin{tabular}{|l|c|c|c|}
\hline
\parbox{2cm}{\vspace{2pt}{\bf Monoidal\\ bicategory}}  &  \multicolumn{3}{c|}{{\bf Pseudomonoid signature}} \\
	&	\multicolumn{1}{c}{Naked} &  \multicolumn{1}{c}{Braided} & \multicolumn{1}{c|}{Symmetric} \\
	\cline{2-4} 
Naked & ${\bf \Delta}$ & N/A & N/A \\
\cline{2-4}
Braided &${\bf B\Delta}$ & ${\bf B\Delta/{\raise.17ex\hbox{$\scriptstyle{\sim}$}}}$ & N/A \\
\cline{2-4}
Symmetric & ${\bf S\Delta}$ & \parbox{5cm}{\center{${\bf FS}^{\text{br}}$} \\[2pt] Locally disconnected categorification of ${\bf FS}$ whose 2-cells are elements of a product of pure braid groups. \vspace{2pt}} & ${\bf FS}$ \\
\hline
\end{tabular}
\caption{This table presents our results regarding the combinatorial bicategory biequivalent to a given higher PRO. Where a combinatorial category from Table~\ref{tbl:monoidsresultsinintro} is given, it is considered as a locally discrete 2-category. We were unable to treat the case of sylleptic monoidal bicategories due to the lack of a known coherence theorem. The result for naked pseudomonoids in a naked monoidal bicategory was already proved by Lack~\cite{Lack2000}. }
\label{tbl:pseudomonoidsintro}
\end{table}

\paragraph{Our techniques.}
We use semistrictness results, allowing us to work with Gray monoids rather than fully weak monoidal bicategories. In Gray monoids, some of the coherent 2-isomorphisms in the definition of a weak monoidal bicategory are taken to be identity 2-cells. This allows a flexible and intuitive `movie calculus'. We develop techniques for working with this calculus which should be applicable to other problems in higher algebra, including the problem of finding similar coherence theorems for Frobenius pseudomonoids~\cite{Street2004} and pseudobialgebras (also known as Hopf categories)~\cite{Neuchl1997}.

\subsection{Related work}

\paragraph{Semistrictness for braided and symmetric monoidal bicategories.}
Gurski proved~\cite{Gurski2011} that every weak braided monoidal bicategory~\cite{Kapranov1994,Baez1996,Crans1998} is biequivalent to a Crans semistrict braided monoidal bicategory~\cite{Crans1998}. The braided monoidal bicategories we consider, derived from the Bar-Vicary definition of semistrict 4-category~\cite{Bar2017}, are slightly stricter than those of Crans, as the hexagonators are trivial; however, in Appendix~\ref{app:semistrictnesshexagonators} we sketch a proof that trivial hexagonators do not affect semistrictness. Our definition also includes the PT-B equality (see Definition~\ref{def:braidedgraymonoid}), which has not appeared in previous definitions of braided monoidal bicategory and apparently cannot be derived from the other axioms; we argue in Section~\ref{sec:newsymmgrayaxioms} that this omission was erroneous, and that PT-B should be included in any definition of a braided monoidal bicategory.

Every weak symmetric  monoidal bicategory is biequivalent to a \emph{quasistrict} symmetric monoidal bicategory\cite{Schommer-Pries2009, Gurski2013}. Our definition of symmetric monoidal bicategories is weaker than the quasistrict definition, as this simplifies our proofs. An alternative formulation in terms of permutative Gray monoids was introduced in recent work~\cite{Gurski2017,Gurski2017a}, but was not required here.

\paragraph{Rewriting theory.}
While there has been much work on higher dimensional rewriting using polygraphs, yielding a powerful theory \cite{Lafont1997,Mimram2014,Guiraud2016} which has been used to rederive coherence results for braided and symmetric monoidal categories \cite{Guiraud2012, Acclavio2016}, this theory is applicable only to strict higher categories. Because our approach is semistrict, it applies to fully weak braided and symmetric monoidal bicategories. This motivates a theory of higher dimensional rewriting in Gray categories. Since the first appearance of the results in this paper some progress was made in this direction~\cite{Forest2018}. 

\subsection{Outline of the paper}

We begin by introducing some background results. In Section~\ref{sec:computadssemistrictbackground} we review basic notions of computads and semistrictness, explaining how semistrictness can be used to apply our results to fully weak braided and symmetric monoidal bicategories. In Section~\ref{sec:presentedgraymonoidsintro} we define semistrict monoidal bicategories (Gray monoids) and their computads. In Section~\ref{sec:newsymmgrayaxioms} we define braided and symmetric Gray monoids and their computads and discuss the PT-B equality. In Section~\ref{sec:symmgraymonoidstechniques} we recall and derive some coherence results that we use in our main proof. In Section~\ref{sec:presentations} we define computads for naked, braided and symmetric pseudomonoids. In Section~\ref{sec:theoryofmonoids} we review results from the theory of monoids which were summarised in Table~\ref{tbl:monoidsresultsinintro}. 

We then move onto our results. In Section~\ref{sec:coherence} we define the combinatorial bicategories appearing in Table~\ref{tbl:pseudomonoidsintro}. In Section~\ref{sec:biequivalences} we define maps between these and the bicategories generated from the pseudomonoid computads. In Section~\ref{sec:theoremstatement} we show how these maps' being biequivalences implies MacLane's coherence theorems for braided and symmetric monoidal categories, and prove that they are essentially surjective on objects and 1-cells, and faithful on 2-cells. In Section~\ref{sec:fullnessproof} we show that the maps are full on 2-cells. In Section~\ref{sec:functorialityproof} we show that the maps are functorial, completing the proof that the maps are biequivalences.

There are two appendices. In Appendix~\ref{app:semistrictnesshexagonators} we sketch a proof that our definition of a Gray monoid, which has trivial hexagonators, is still semistrict. In Appendix~\ref{sec:proofsforgraymoncohappendix} we provide the proof of the main coherence result from Section~\ref{sec:symmgraymonoidstechniques}.

\subsection{Globular workspace}

In this work we have used  \emph{Globular} \cite{Bar2018}, a graphical proof assistant for semistrict higher category theory which allows one to easily view and manipulate higher compositions. Globular has a definition of semistrict 4-category~\cite{Bar2017}; in this definition, a semistrict 4-category with only one 0- and one 1-cell is a braided Gray monoid in the sense of Definition~\ref{def:braidedgraymonoid}, where $n$-cells in the 4-category are considered as $(n-2)$-cells in the braided Gray monoid.  Equalities are encoded by invertible cells in higher dimension.

We have have encoded certain graphical proofs from this paper into a \emph{Globular} workspace, which can be found at \url{http://globular.science/1705.001v2}. The propositions are equalities of 2-cells in the braided Gray monoid --- that is, invertible 5-cells in the workspace. The proofs of the propositions, which take these invertible 5-cells and and expand them as a series of generating equalities, are invertible 6-cells in the workspace. Note that the higher categorical structure in Globular is only being used at the 4-categorical level; the use of 5- and 6-cells to encode propositions and proofs is simply formal.

\subsection{Acknowledgements}

The author would like to thank Krzysztof Bar for support with \emph{Globular}, three anonymous referees from FSCD2018 for helpful comments on an early version of this work, Jamie Vicary for advice throughout the writing process, and Manuel B{\"a}renz, Vaia Patta and David Reutter for useful discussions and comments. This work was supported by the UK Engineering and Physical Sciences Research Council.

\section{Background}

\subsection{Computads and semistrictness}
\label{sec:computadssemistrictbackground}

\emph{Computads}, sometimes known as \emph{presentations} or \emph{polygraphs}, are generating data for a category~\cite{Batanin1998, Schommer-Pries2009}. The bicategories we study in this work are \emph{computadic}; that is, generated from computads. As already discussed, we seek biequivalences between these bicategories and some simpler combinatorial 2-categories. To this end, we make use of \emph{semistrictness} results, which show that any fully weak bicategory is biequivalent, in the appropriate sense, to a more tractable \emph{semistrict} bicategory. We then need only consider the semistrict bicategories.

For the semistrict bicategories we consider here, there is also a notion of semistrict computad, with a `quotient' functor from weak computads to semistrict computads compatible with the biequivalence in the following sense. Let $F, F_{SS}$ be functors which take a computad to the bicategory it generates, and let $\tilde{S}, S$ be the functors which take a computad or a bicategory to its corresponding semistrict computad or biequivalent semistrict bicategory. Then the following diagram commutes:
\begin{diagram}
\textbf{Comp} & \rTo^{F} & \textbf{Cats} \\
\dTo^{\tilde{S}} &		& \dTo_{S} \\
\textbf{SSComp}	& \rTo_{F_{SS}}	& \textbf{SSCats}
\end{diagram}
For detail, see~\cite{Schommer-Pries2009}.

In what follows we will only define semistrict bicategories and semistrict pseudomonoid computads; by the above discussion, our coherence results apply in the weak case also. Fully weak definitions can be found in the work of other authors.\footnote{The definition of a computad for a weak symmetric monoidal bicategory is given, along with a description of the weak symmetric monoidal bicategory it generates, in~\cite[Section 2.10]{Schommer-Pries2009}; the definition of a computad for a weak braided monoidal bicategory is identical except for the omission of the symbols $\sigma$, and the definition of a computad for a weak naked monoidal bicategory is identical except for the omission of  the symbols $\beta$, $R$, $S$ and $\sigma$. Naked, braided and symmetric pseudomonoid computads are weak naked, braided and symmetric monoidal bicategory computads with the generating cells given in~\cite[pp.79-81,86-87,90]{McCrudden2000}.}

\subsection{Semistrict monoidal bicategories and their computads}\label{sec:presentedgraymonoidsintro}

The semistrict monoidal bicategories we consider here are computadic \emph{Gray monoids}~\cite{Day1997,Gurski2006}. We will not be overly concerned with technical details, which have been treated elsewhere~\cite{Day1997, Hummon2012, Schommer-Pries2009}, but will rather provide an informal overview using the diagrammatic approach of Bar and Vicary~\cite{Bar2017}. 

Gray monoid computads are defined inductively: for each $0 \leq k \leq 2$ there is a set $C_k$ of generating $k$-cells. For $k \neq 0$, each generating $k$-cell has a $(k-1)$ cell as source and another as target. There is also a set $E$ of equalities of 2-cells, each of which has a 2-cell as source and another as target (although the choice of source and target here is arbitrary). In order to define the $k$-th level of the computad, one must know how $(k-1)$-cells are generated from the lower levels.

We will define the sets of $k$-cells of a computadic Gray monoid first, and then describe its compositional structure.

\begin{definition}[0-cells of a Gray monoid]\label{def:0cellsgraymoncomp}
The 0-cells generated from $(C_0)$ are ordered lists of elements of $C_0$.
\end{definition} 
Every generating 1-cell in $C_1$ has an ordered list of elements of $C_0$ as source and target; $(C_0, C_1)$ can therefore be considered as the computad for a monoidal category. We assume the reader is familiar with the string diagram calculus for monoidal categories, which is well-established~\cite{Selinger2010,Joyal1991a,Joyal1991b}. The 1-cells of a computadic Gray monoid are defined as string diagrams generated from $(C_0,C_1)$; however, the notion of topological equivalence used to identify two diagrams as referring to the same 1-cell is more rigid.

\begin{definition}
Given generating 0- and 1-cells $(C_0,C_1)$, an \emph{ordered string diagram} generated from this data is a string diagram, no pair of whose generating 1-cells occur at the same vertical height in the diagram.
\end{definition}

\begin{definition}
An \emph{ordered planar isotopy} between ordered string diagrams is a planar isotopy between them where, at each point of the isotopy, the string diagrams are ordered.
\end{definition}
\begin{definition}[1-cells of a Gray monoid]\label{def:1cellsgraymoncomp}
The 1-cells generated from $(C_0,C_1)$ are ordered string diagrams generated from $(C_0,C_1)$, identified up to ordered planar isotopy.
\end{definition}
\noindent
These notions are illustrated in Figure~\ref{fig:orderedstringdiags}.
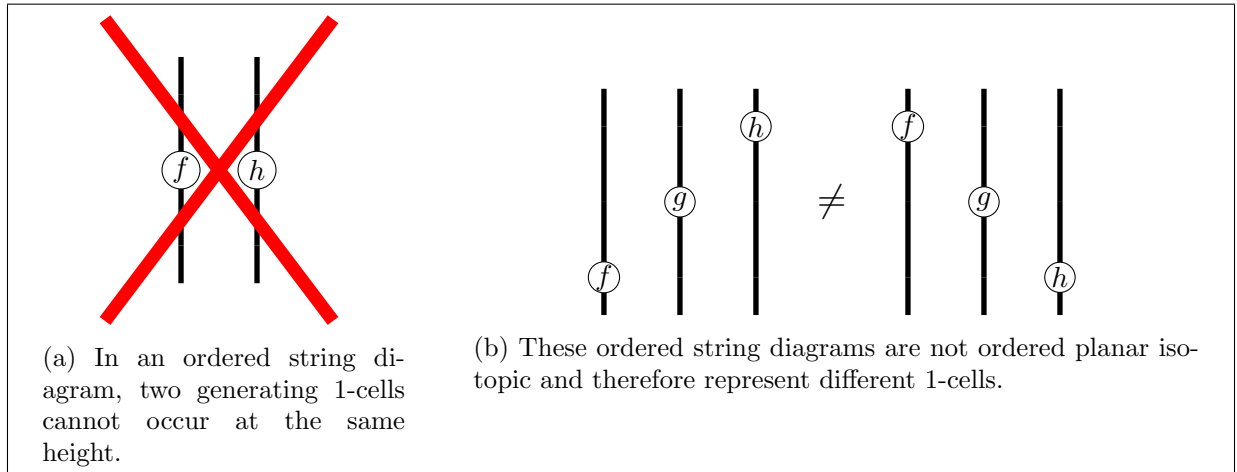
\begin{figure}
\centering
\begin{subfigure}{0.3\textwidth}\centering
\begin{tikzpicture}
	\begin{pgfonlayer}{nodelayer}
		\node [fill=white, minimum size=0.5cm, draw, circle, style=none] (0) at (1, 2) {$f$};
		\node [style=none] (1) at (2, 3) {};
		\node [fill=white, minimum size=0.5cm, draw, circle, style=none] (2) at (2, 2) {$h$};
		\node [style=none] (3) at (2, 1) {};
		\node [style=none] (4) at (1, 1) {};
		\node [style=none] (5) at (1, 3) {};
		\node [style=none] (6) at (1, 3.5) {};
		\node [style=none] (7) at (1, 0.5) {};
		\node [style=none] (8) at (2, 3.5) {};
		\node [style=none] (9) at (2, 0.5) {};
		\node [style=none] (10) at (3, 4) {};
		\node [style=none] (11) at (0, -0) {};
		\node [style=none] (12) at (0, 4) {};
		\node [style=none] (13) at (3, -0) {};
	\end{pgfonlayer}
	\begin{pgfonlayer}{edgelayer}
		\draw [style=simple] (5.center) to (0.center);
		\draw [style=simple] (0.center) to (4.center);
		\draw [style=simple] (1.center) to (2.center);
		\draw [style=simple] (2.center) to (3.center);
		\draw [style=simple] (6.center) to (5.center);
		\draw [style=simple] (7.center) to (4.center);
		\draw [style=simple] (9.center) to (3.center);
		\draw [style=simple] (8.center) to (1.center);
		\draw [line width=5pt, red] (13.center) to (12.center);
		\draw [line width=5pt, red] (11.center) to (10.center);
	\end{pgfonlayer}
\end{tikzpicture}
\caption{In an ordered string diagram, two generating 1-cells cannot occur at the same height.}
\end{subfigure}
\qquad
\begin{subfigure}{0.6\textwidth}\centering
\begin{tikzpicture}
	\begin{pgfonlayer}{nodelayer}
		\node [fill=white, minimum size=0.4cm, draw, circle, style=none] (0) at (0, 0.5) {$f$};
		\node [style=none] (1) at (0, 1.5) {};
		\node [style=none] (2) at (0, 2.5) {};
		\node [style=none] (3) at (1, 0.5) {};
		\node [style=none] (4) at (1, 2.5) {};
		\node [fill=white, minimum size=0.4cm, draw, circle, style=none] (5) at (1, 1.5) {$g$};
		\node [style=none] (6) at (2, 1.5) {};
		\node [fill=white, minimum size=0.4cm, draw, circle, style=none] (7) at (2, 2.5) {$h$};
		\node [style=none] (8) at (2, 0.5) {};
		\node [style=none] (9) at (5, 2.5) {};
		\node [style=none] (10) at (4, 0.5) {};
		\node [fill=white, minimum size=0.4cm, draw, circle, style=none] (11) at (6, 0.5) {$h$};
		\node [style=none] (12) at (4, 1.5) {};
		\node [fill=white, minimum size=0.4cm, draw, circle, style=none] (13) at (5, 1.5) {$g$};
		\node [style=none] (14) at (6, 1.5) {};
		\node [style=none] (15) at (6, 2.5) {};
		\node [fill=white, minimum size=0.4cm, draw, circle, style=none] (16) at (4, 2.5) {$f$};
		\node [style=none] (17) at (5, 0.5) {};
		\node [style=none] (18) at (0, 3) {};
		\node [style=none] (19) at (0, -0) {};
		\node [style=none] (20) at (1, 3) {};
		\node [style=none] (21) at (1, -0) {};
		\node [style=none] (22) at (2, 3) {};
		\node [style=none] (23) at (2, -0) {};
		\node [style=none] (24) at (4, 3) {};
		\node [style=none] (25) at (4, -0) {};
		\node [style=none] (26) at (5, 3) {};
		\node [style=none] (27) at (5, -0) {};
		\node [style=none] (28) at (6, 3) {};
		\node [style=none] (29) at (6, -0) {};
		\node [style=none] (30) at (3, 1.5) {{\Large $\neq$}};
	\end{pgfonlayer}
	\begin{pgfonlayer}{edgelayer}
		\draw [style=simple] (0.center) to (1.center);
		\draw [style=simple] (1.center) to (2.center);
		\draw [style=simple] (4.center) to (5.center);
		\draw [style=simple] (5.center) to (3.center);
		\draw [style=simple] (8.center) to (6.center);
		\draw [style=simple] (7.center) to (6.center);
		\draw [style=simple] (10.center) to (12.center);
		\draw [style=simple] (12.center) to (16.center);
		\draw [style=simple] (17.center) to (13.center);
		\draw [style=simple] (9.center) to (13.center);
		\draw [style=simple] (11.center) to (14.center);
		\draw [style=simple] (15.center) to (14.center);
		\draw [style=simple] (2.center) to (18.center);
		\draw [style=simple] (19.center) to (0.center);
		\draw [style=simple] (21.center) to (3.center);
		\draw [style=simple] (23.center) to (8.center);
		\draw [style=simple] (22.center) to (7.center);
		\draw [style=simple] (20.center) to (4.center);
		\draw [style=simple] (24.center) to (16.center);
		\draw [style=simple] (25.center) to (10.center);
		\draw [style=simple] (27.center) to (17.center);
		\draw [style=simple] (29.center) to (11.center);
		\draw [style=simple] (28.center) to (15.center);
		\draw [style=simple] (26.center) to (9.center);
	\end{pgfonlayer}
\end{tikzpicture}
\caption{These ordered string diagrams are not ordered planar isotopic and therefore represent different 1-cells.}
\end{subfigure}
\caption{Ordered string diagrams and ordered planar isotopy.}
\label{fig:orderedstringdiags}
\end{figure}

\begin{remark}\label{rem:precisedefn1cells}
The greater rigidity of ordered planar isotopy equivalence allows one to divide ordered string diagrams into vertical levels, where precisely one generating 1-cell occurs at each vertical level. The 1-cells may therefore be written as compositions of tensor products of generating 1-cells with the identity (`whiskerings'). This links the diagrammatic approach to more conventional presentations of Gray monoids~\cite{Schommer-Pries2009,Gurski2006}.
\end{remark}

\begin{example}
Some examples of ordered string diagrams constructed from generating 0- and 1-cells are the sources and targets of 2-cells in the pseudomonoid computad $\mathcal{P}$ (Definition~\ref{pseudomonoidpresentation}). 
\end{example}
We now consider 2-cells. In a Gray monoid, there is an additional family of generating 2-cells, not specified within the computad, but rather obtained from the generating 1-cells of the computad. These implement non-ordered planar isotopy, which was an equality for monoidal categories but in Gray monoids is controlled by nontrivial 2-cells.
\begin{definition}
We say that two generating 1-cells in an ordered string diagram are \emph{connected} if one may be reached from the other by a path through strings which always travels upwards in the diagram.
\end{definition}
\begin{definition}[Interchangers]
In any 1-cell diagram where two unconnected 1-cells $f,g$ are vertically adjacent, there is a generating \emph{interchanger} 2-cell $\iota$ whose source is the original 1-cell and whose target is the 1-cell with the heights of $f$ and $g$ interchanged. 
\begin{center}\begin{tikzpicture}
	\begin{pgfonlayer}{nodelayer}
		\node [style=none] (0) at (1.5, 0.75) {$\Rightarrow$};
		\node [style=none] (1) at (3, 1) {};
		\node [style=none] (2) at (2.25, -0) {};
		\node [style=none] (3) at (2.25, 0.5) {};
		\node [fill=white, minimum size=0.4cm, draw, circle, style=none] (4) at (3, 0.5) {$g$};
		\node [fill=white, minimum size=0.4cm, draw, circle, style=none] (5) at (2.25, 1) {$f$};
		\node [style=none] (6) at (3, -0) {};
		\node [style=none] (7) at (2.25, 1.5) {};
		\node [style=none] (8) at (3, 1.5) {};
		\node [style=none] (9) at (0, -0) {};
		\node [style=none, circle, draw, minimum size=0.4cm, fill=white] (10) at (0.75, 1) {$g$};
		\node [style=none] (11) at (0.75, -0) {};
		\node [style=none] (12) at (0, 1.5) {};
		\node [style=none] (13) at (0.75, 1.5) {};
		\node [style=none] (14) at (0.75, 1) {};
		\node [style=none, circle, draw, minimum size=0.4cm, fill=white] (15) at (0, 0.5) {$f$};
		\node [style=none] (16) at (1.5, 1) {$\iota$};
		\node [style=none] (17) at (0, 0.5) {};
	\end{pgfonlayer}
	\begin{pgfonlayer}{edgelayer}
		\draw [style=simple] (2.center) to (3.center);
		\draw [style=simple] (3.center) to (5.center);
		\draw [style=simple] (6.center) to (4.center);
		\draw [style=simple] (1.center) to (4.center);
		\draw [style=simple] (7.center) to (5.center);
		\draw [style=simple] (8.center) to (1.center);
		\draw [style=simple] (11.center) to (10.center);
		\draw [style=simple] (14.center) to (10.center);
		\draw [style=simple] (12.center) to (15.center);
		\draw [style=simple] (13.center) to (14.center);
		\draw [style=simple] (9.center) to (17.center);
		\draw [style=simple] (17.center) to (15.center);
	\end{pgfonlayer}
\end{tikzpicture}
\end{center}
An interchanger 2-cell exists regardless of horizontal separation of the two vertically adjacent generating 1-cells by other strings.
\end{definition}
Generic 2-cells will be sequences of applications of generating 2-cells to subregions of a 1-cell diagram. For their definition we therefore need a good notion of subregion.
\begin{definition}
We define a \emph{rectangular subregion} of an ordered string diagram as the interior $I(R)$ of an embedded rectangle $R$ satisfying the following properties:
\begin{itemize}
\item $I(R)$ contains all generating 1-cells occuring at a vertical level between the bottom and the top of the rectangle.
\item The boundary of $R$ does not intersect generating 1-cells.
\item Intersections of the boundary of $R$ with strings are all on the bottom and top edge of $R$.
\end{itemize}
These conditions ensure that the rectangular subregion itself contains an ordered string diagram, corresponding to a 1-cell. See Figure~\ref{fig:rectangularsubregion} for examples.
\begin{figure}\centering
\subcaptionbox{The source of an interchanger.}
{\includegraphics[scale=.3]{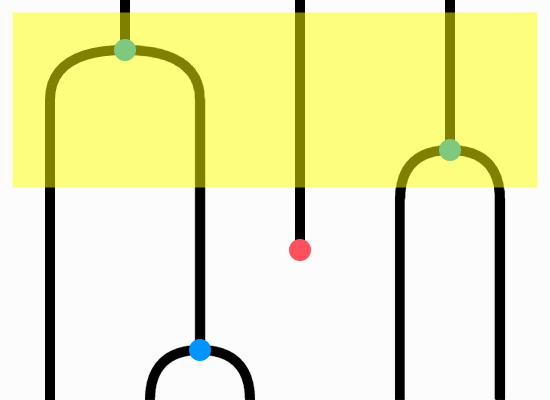}}
\qquad \qquad 
\subcaptionbox{The source of an inverse associator (Definition~\ref{pseudomonoidpresentation}).}
{\includegraphics[scale=.3]{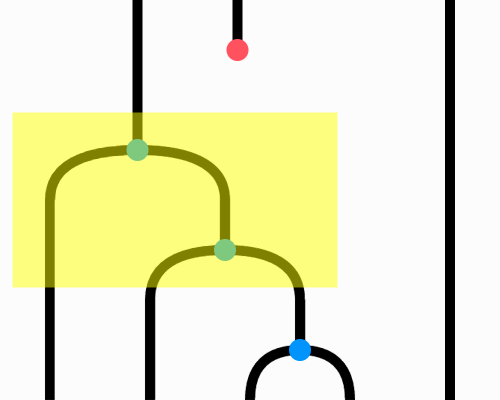}}
\caption{Examples of rectangular subregions.}\label{fig:rectangularsubregion}
\end{figure}
\end{definition}

In order to define 1-cells, we introduced ordered string diagrams
and then stipulated that ordered planar isotopic ordered string diagrams represented the same 1-cell. Likewise, in order to define 2-cells we introduce \emph{movies} and identify these up to a certain equivalence relation.
\begin{definition}[Movie]
Let $D_1, \dots, D_n$ be 1-cells in the free Gray monoid on the computad $(C_0,C_1,C_2,C_3)$. A \emph{movie} $D_1 \to D_n$ is a sequence $$D_1 \overset{\gamma_{1,2}}{\Rightarrow} D_2 \overset{\gamma_{2,3}}{\Rightarrow} \dots \overset{\gamma_{n-1,n}}{\Rightarrow} D_n,$$ where $D_i$ are 1-cells and $\gamma_{i,i+1}$ are generating 2-cells such that $D_i$ and $D_{i+1}$ differ only by the application of $\gamma_{i,i+1}$ to a rectangular subregion.
\end{definition}
As the 1-cells are represented by ordered string diagrams, if we draw these we get a series of transitions of planar diagrams. This is the reason for the name `movie'; each ordered string diagram is a \emph{frame} of the movie. 
We call a contiguous subsequence of frames in a movie a \emph{clip}.

We now introduce the following \emph{structural equalities} by which movies will be identified.
\begin{definition}[Structural equalities of a Gray monoid]\label{structuralequalitiesgraymondef}
\begin{enumerate} \item \emph{Type I rewrites}. If two generating 2-cells occur consecutively in the movie, and the sets of generating 1-cells involved in each have zero intersection, their order may be interchanged. 

For example, take the movie generated from the pseudomonoid computad $\mathcal{P}$~(Definition~\ref{pseudomonoidpresentation}), in which two associators are applied to four left-bracketed multiplication 1-cells, the first on the bottom pair and the second on the top pair. There is a Type I rewrite that switches this movie for an equal one where the associator is applied to the top pair of nodes first, then to the bottom pair.
\begin{center}
\raisebox{0.45cm}{\huge [}\includegraphics[width=1cm,height=1.5cm]{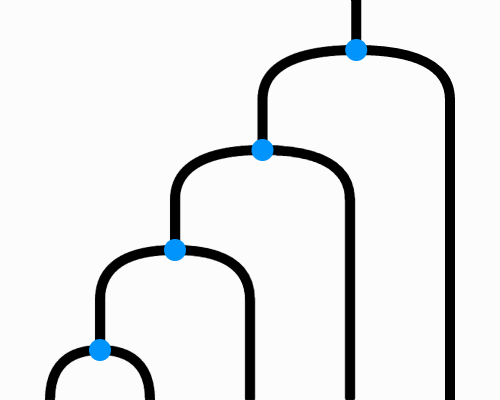}
\raisebox{0.6cm}{$\Rightarrow$}\includegraphics[width=1cm,height=1.5cm]{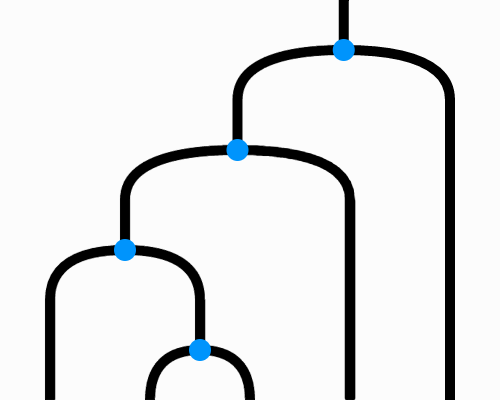}
\raisebox{0.6cm}{$\Rightarrow$}
\includegraphics[width=1cm,height=1.5cm]{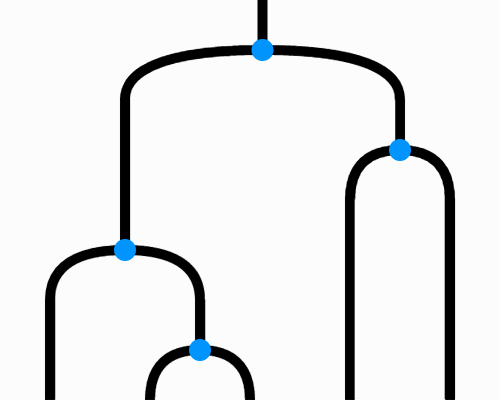} \raisebox{0.45cm}{\huge ]} \raisebox{0.45cm}{$=$}
\raisebox{0.45cm}{\huge [}\includegraphics[width=1cm,height=1.5cm]{typeIrewriteexample1}
\raisebox{0.6cm}{$\Rightarrow$}\includegraphics[width=1cm,height=1.5cm]{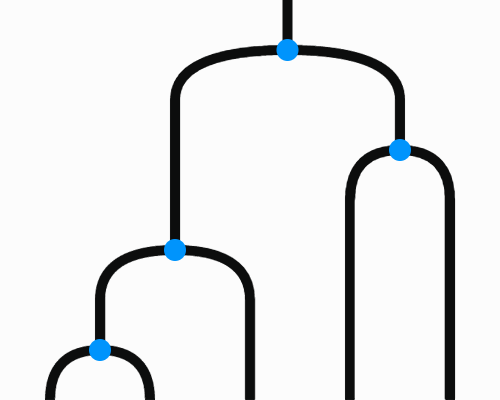}
\raisebox{0.6cm}{$\Rightarrow$}
\includegraphics[width=1cm,height=1.5cm]{typeIrewriteexample3}\raisebox{0.45cm}{\huge ]}
\end{center} 

\item \emph{Type II rewrites}. These rewrites state that the downwards interchanger is the inverse of the upwards interchanger. If a 1-cell is unconnected to the 1-cell directly above it, we may insert an interchanger and its inverse into the movie; likewise, we may remove an interchanger and its inverse when they occur together. 

Here is an example generated from $\mathcal{P}$, where an interchanger between a unit and a multiplication node may be inserted:
\begin{center}
\raisebox{0.45cm}{\huge [}\includegraphics[width=1cm,height=1.5cm]{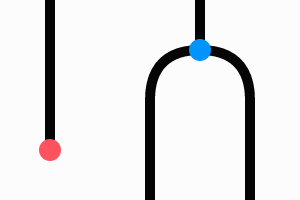}
\raisebox{0.45cm}{\huge ]} 
\raisebox{0.6cm}{$=$}
\raisebox{0.45cm}{\huge [}\includegraphics[width=1cm,height=1.5cm]{typeIIinterchangerrewrite1}
\raisebox{0.6cm}{$\Rightarrow$}\includegraphics[width=1cm,height=1.5cm]{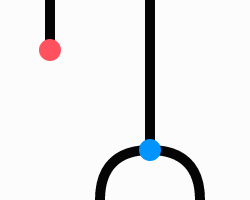}
\raisebox{0.6cm}{$\Rightarrow$}
\includegraphics[width=1cm,height=1.5cm]{typeIIinterchangerrewrite1}\raisebox{0.45cm}{\huge ]} 
\end{center}
\item \emph{Type III rewrites}. When one 1-cell interchanges with another 1-cell, and either 1-cell is immediately acted upon by the  following 2-cell, the movie may be rewritten so that the 2-cell occurs before the interchanger.  

The following example is generated by $\mathcal{P}$. Here a unit interchanges with a 1-cell after application of an associator; this may be rewritten to a movie where the associator occurs before the interchanger:
\begin{center}
\raisebox{0.45cm}{\huge [}\includegraphics[width=1cm,height=1.5cm]{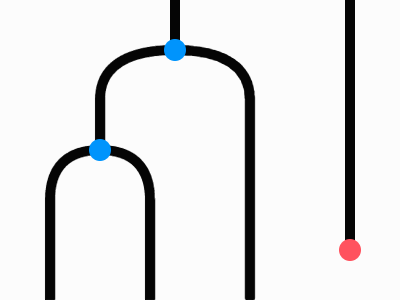}
\raisebox{0.6cm}{$\Rightarrow$}\includegraphics[width=1cm,height=1.5cm]
{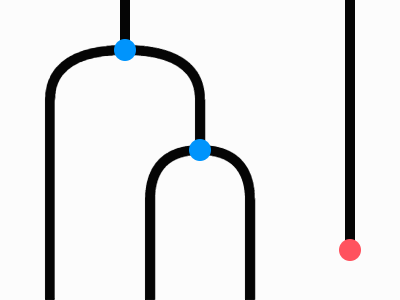}
\raisebox{0.6cm}{$\Rightarrow$}\includegraphics[width=1cm,height=1.5cm]
{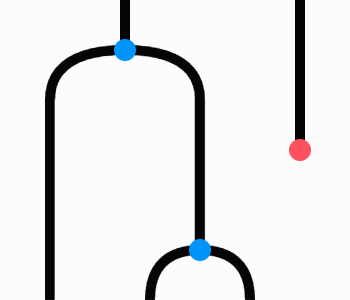}
\raisebox{0.6cm}{$\Rightarrow$}\includegraphics[width=1cm,height=1.5cm]
{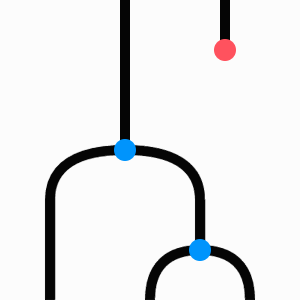}
\raisebox{0.45cm}{\huge ]} 
\raisebox{0.45cm}{$=$}
\raisebox{0.45cm}{\huge [}\includegraphics[width=1cm,height=1.5cm]{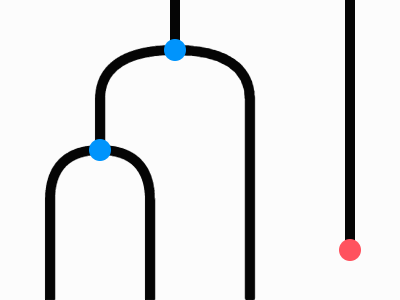}
\raisebox{0.6cm}{$\Rightarrow$}\includegraphics[width=1cm,height=1.5cm]
{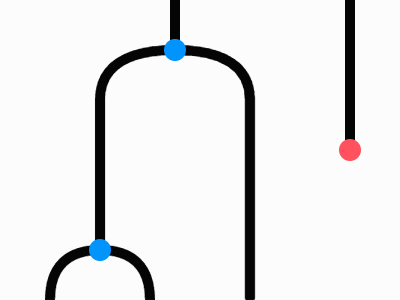}
\raisebox{0.6cm}{$\Rightarrow$}\includegraphics[width=1cm,height=1.5cm]
{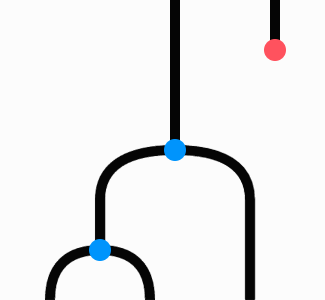}
\raisebox{0.6cm}{$\Rightarrow$}\includegraphics[width=1cm,height=1.5cm]
{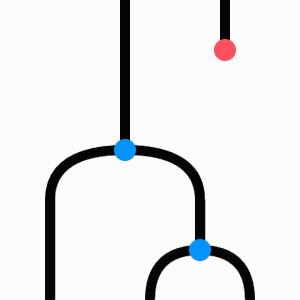}
\raisebox{0.45cm}{\huge ]} 
\end{center} 
\end{enumerate}
\end{definition}
\begin{remark}
The above equalities correspond to those in~\cite[Definition 1]{Day1997}. In particular, the Type I rewrites correspond to naturality in the strict 2-category; the Type II rewrites correspond to the fact that the interchanger is an isomorphism; and the Type III rewrites correspond to equality (iii) in Day and Street's definition. 
\end{remark}

\begin{definition}[2-cells of a Gray monoid]\label{def:graymon2cells}
A 2-cell generated from $(C_0,C_1,C_2)$ is a movie constructed from $(C_0,C_1,C_2)$, where movies are identified up to the structural equalities of Definition~\ref{structuralequalitiesgraymondef}.
\end{definition}

Finally, the computad contains a set $E$ of specified equalities of 2-cells; for each element of this set there is a pair of $2$-cells (Definition~\ref{def:graymon2cells}) which are defined to be equal. In order to apply 2-cells to subregions of 1-cells, we introduced the notion of a rectangular subregion. Now, in order to apply rewrites to subregions of 2-cells, we introduce a higher-dimensional notion. 

\begin{definition}
Let $M = D_1 \Rightarrow \dots \Rightarrow D_n$ be a 2-cell. A \emph{cuboidal subregion} of $M$ is a choice of clip $c = D_{i_1} \Rightarrow \dots \Rightarrow D_{i_2}$, $1 \leq i_1 \leq i_2 \leq n$, together with a fixed rectangular subregion $R$ of $D_{i_1}$, such that every transition in the clip acts on a rectangular subregion within $R$. A cuboidal subregion specifies a `sub-2-cell' with source 1-cell $D_{i_1}|_R$ and target 1-cell $D_{i_2}|_R$ in the obvious way.
\end{definition}

\begin{definition}[Movie rewrites on cuboidal subregions]
Let $M$ be a movie. Whenever a cuboidal subregion of $M$ specifies a sub-2-cell related to another by an equality in $E$, $M$ is equal to the movie $M'$ which is identical except for the replacement of the clip in the cuboidal subregion by the equal clip. 
\end{definition}

\begin{definition}[2-cells following quotient by $E$]\label{def:2cellsgraymon}
The 2-cells generated from $(C_0,C_1,C_2,E)$ are precisely movies constructed from $(C_0,C_1,C_2)$, identified up to the structural equalities of Definition~\ref{structuralequalitiesgraymondef} and the specified equalities in the computad.
\end{definition}

Now we have completely defined the 0-, 1- and 2-cells generated by the computad $\mathcal{C}$; all that remains is to define composition, and we have a full definition of our semistrict monoidal bicategory $G(\mathcal{C})$. 

\begin{definition}[Compositional structure of Gray monoid]
The 0-, 1- and 2-cell of a Gray monoid compose as shown in Table~\ref{compositioningeneratedgraymonoid}.
\end{definition} 

Note that the only thing that prevents this from being a \emph{strict} monoidal bicategory is the failure of the interchange law. This weakness is sufficient for semistrictness.

\begin{table}[p]
\resizebox{!}{12cm}{
}
\caption{The compositional structure of a Gray monoid.}
\label{compositioningeneratedgraymonoid}
\end{table}
\noindent
We finish this section by defining some vocabulary.
\begin{definition}[Isomorphism]\label{def:isomorphism}
We make signatures of computads more concise by saying that a particular 2-cell $\mu$ is an \emph{isomorphism}. This means that there is another 2-cell $\mu^{-1}$ in the signature with $s(\mu^{-1}) = t(\mu)$ and $t(\mu^{-1})=s(\mu)$, satisfying $\mu \circ_H \mu^{-1} =  id_{t(\mu)}$ and $\mu^{-1}\circ_H \mu =  id_{s(\mu)}$.
\end{definition}
\begin{example}
In the pseudomonoid signature $\mathcal{P}$ of Definition~\ref{pseudomonoidpresentation}, by specifying $\alpha$ as an isomorphism we avoid having to specify the 2-cell $\alpha^{-1}$ and two equalities.
\end{example} 	

\begin{definition}
\label{cellsmovies2projremark}
We call a movie whose source and target are equal a \emph{loop}. We call a sequence of rewrites which take this movie to the trivial movie a \emph{contraction} of the loop.
\end{definition}

\begin{definition}
When all the 2-cells featuring in an equality are isomorphisms, an equality still holds if the direction of all 2-cells on both sides of the equality is reversed. We call this the \emph{flip} of the equality.
\end{definition}

\subsection{Braided and symmetric Gray monoids}\label{sec:newsymmgrayaxioms}

The 0-, 1- and 2-cells of semistrict braided and symmetric monoidal bicategories (\emph{braided and symmetric Gray monoids}) are generated and composed in exactly the same way as for naked Gray monoids. The difference is that the computads for braided and symmetric Gray monoids include additional \emph{structural} generating cells which are not specified explicitly in the computad, but are rather constructed from the other specified data.

\subsubsection{Braided Gray monoids} 

\begin{definition}\label{def:braidedgraymonoid}
A braided Gray monoid has the following additional structural generating cells and equalities.\footnote{For the equalities we also give the hieroglyphic notation of Kapranov and Voevodsky where defined~\cite{Kapranov1994}. We highlight the rectangular subregion containing the source of an applied 2-cell where this might be unclear.}

\begin{itemize}[leftmargin=*]
\item \underline{Additional generating 1-cells}:
\begin{itemize}[leftmargin=*,label={-}]
\item For every pair of 0-cells $A,B$, an `overbraiding' 1-cell $R^{+}_{A,B}: A \otimes B \to B \otimes A$ and an `underbraiding' 1-cell $R^{-}_{A,B}: A \otimes B \to B \otimes A$ . We depict these 1-cells as braidings in the 1-morphism diagram:
\begin{center}
\begin{tabular}{l l}
\raisebox{0.6cm}{$R^{+}_{A,B}=$}\includegraphics[width=1cm,height=1.5cm]{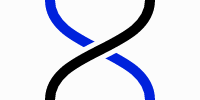}& 
\raisebox{0.6cm}{$\sim \quad R^{-}_{A,B}=$}\includegraphics[width=1cm,height=1.5cm]{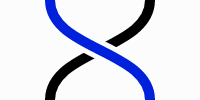}
\end{tabular}
\end{center}
The hexagonators are trivial; that is, for all 0-cells $A,B,C$ in $G(\Sigma)$, we have:
\begin{align*}
R^{\pm}_{A \otimes B,C} &= (id_B \otimes R^{\pm}_{A,C}) \circ (R^{\pm}_{A,B} \otimes id_C) \\ R^{\pm}_{A,B \otimes C} &= (R^{\pm}_{A,C} \otimes id_B) \circ (id_A \otimes R^{\pm}_{B,C})
\end{align*}
\end{itemize}
\item \underline{Additional generating 2-cells}:
\begin{itemize}[label={-},leftmargin=*]
\item `Braiding inverse-insert' 2-isomorphisms $i^{+}_{AB},i^{-}_{AB}$ for each pair, $A, B$ of $0$-cells:
\begin{center}
\begin{tabular}{l l}
\includegraphics[width=1cm,height=1.5cm]{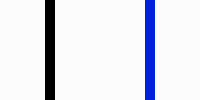}
\raisebox{0.6cm}{$\overset{i^+_{AB}}{\Rightarrow}$}\includegraphics[width=1cm,height=1.5cm]{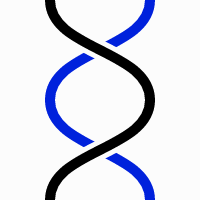} &\qquad
\includegraphics[width=1cm,height=1.5cm]{identitypair}
\raisebox{0.6cm}{$\overset{i^-_{AB}}{\Rightarrow}$}\includegraphics[width=1cm,height=1.5cm]{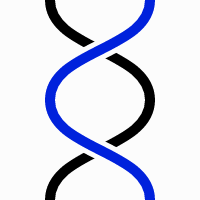} 
\end{tabular}
\end{center}
These isomorphisms are strictly monoidal; that is, on products they are equal to the composites defined in the obvious way. For example:
\begin{center}
\label{eq:trivialhexagonatorbraidings}
\raisebox{0.6cm}{{\huge [}} 
\includegraphics[width=1cm,height=1.5cm]{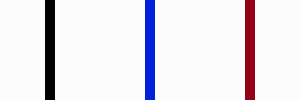}
\raisebox{0.6cm}{$\overset{i_{A,B \otimes C}^{+}}{\Rightarrow}$}
\includegraphics[width=1cm,height=1.5cm]{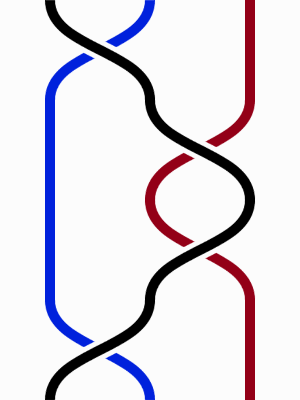}
\raisebox{0.6cm}{{\huge ]}}
\raisebox{0.6cm}{=}
\raisebox{0.6cm}{{\huge [}} 
\includegraphics[width=1cm,height=1.5cm]{braidinginverseinsertsontensor1}
\raisebox{0.6cm}{$\overset{i_{A,B}^{+}}{\Rightarrow}$}
\includegraphics[width=1cm,height=1.5cm]{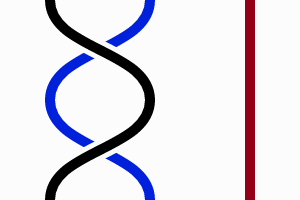}
\raisebox{0.6cm}{$\overset{i_{B,C}^{+}}{\Rightarrow}$}
\includegraphics[width=1cm,height=1.5cm]{braidinginverseinsertsontensor3}
\raisebox{0.6cm}{{\huge ]}} 
\end{center}

\item `Pull-over' and `pull-under' 2-isomorphisms $PO_{f,B}$, $PO_{A,g}$, $PU_{f,B}$ and $PU_{A,g}$ for all $1$-cells $f:A \to C$, $g: B \to D$ in $\mathcal{C}$.
\begin{center}
\begin{tabular}{l l}
\includegraphics[width=1cm,height=1.5cm]{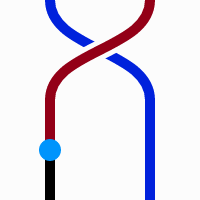}
\raisebox{0.6cm}{$\overset{PO_{f,B}}{\Rightarrow}$}\includegraphics[width=1cm,height=1.5cm]{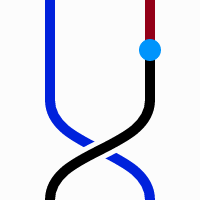} &\qquad
\includegraphics[width=1cm,height=1.5cm]{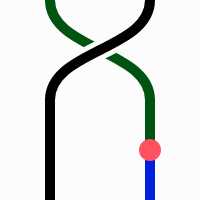}
\raisebox{0.6cm}{$\overset{PU_{A,g}}{\Rightarrow}$}\includegraphics[width=1cm,height=1.5cm]{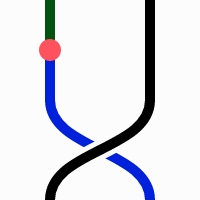}
\end{tabular}
\end{center}
\end{itemize}

\item \underline{Additional equalities}:
\begin{itemize}[leftmargin=*,label={-}]
\item $(\rightarrow \otimes \rightarrow)$. For $f:A \to C$, $g:B \to D$, we have the following equality:
\begin{center}
\begin{tabular}{l l}
\raisebox{0.45cm}{\Huge [} \includegraphics[width=1cm,height=1.5cm]{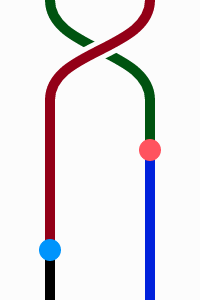}
\raisebox{0.6cm}{$\overset{PU_{C,g}}{\Rightarrow}$}\includegraphics[width=1cm,height=1.5cm]{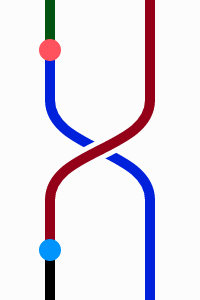}
\raisebox{0.6cm}{$\overset{PO_{f,B}}{\Rightarrow}$}\includegraphics[width=1cm,height=1.5cm]{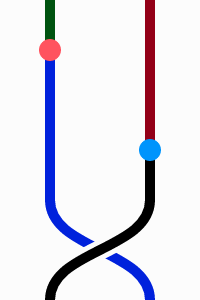}
\raisebox{0.6cm}{$\overset{\iota}{\Rightarrow}$}\includegraphics[width=1cm,height=1.5cm]{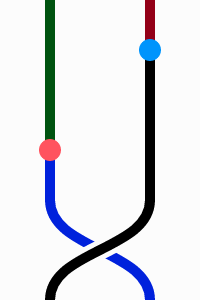}
\raisebox{0.45cm}{\Huge ]}
\raisebox{0.6cm}{$=$}\raisebox{0.45cm}{\Huge [} \includegraphics[width=1cm,height=1.5cm]{doublepullthroughcoherencestart}
\raisebox{0.6cm}{$\overset{\iota}{\Rightarrow}$}\includegraphics[width=1cm,height=1.5cm]{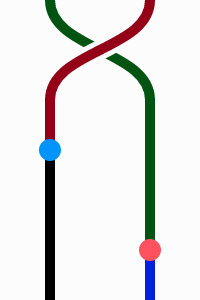}
\raisebox{0.6cm}{$\overset{PO_{f,C}}{\Rightarrow}$}\includegraphics[width=1cm,height=1.5cm]{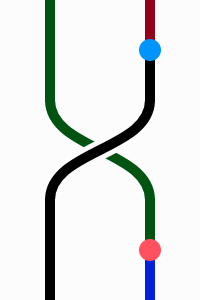}
\raisebox{0.6cm}{$\overset{PU_{A,g}}{\Rightarrow}$}\includegraphics[width=1cm,height=1.5cm]{doublepullthroughcoherenceend}
\raisebox{0.45cm}{\Huge ]}
\end{tabular}
\end{center}
\item $(\cdot\; \otimes \Downarrow)$. For morphisms $f,g: B \to D$ and a 2-morphism $\alpha: f \to g$, we have the following equality. (Here and elsewhere, we use highlighting to make clear where a 2-morphism is about to be applied.)
\begin{center}
\begin{tabular}{l l}
\raisebox{0.45cm}{\Huge [} \includegraphics[width=1cm,height=1.5cm]{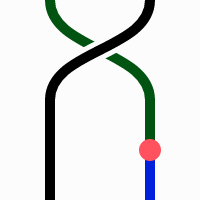}
\raisebox{0.6cm}{$=$}\includegraphics[width=1cm,height=1.5cm]{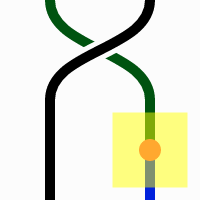}
\raisebox{0.6cm}{$\overset{\alpha}{\Rightarrow}$}\includegraphics[width=1cm,height=1.5cm]{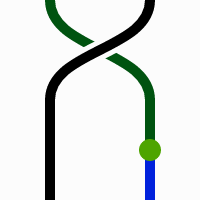}
\raisebox{0.6cm}{$\overset{PU_{A,g}}{\Rightarrow}$}\includegraphics[width=1cm,height=1.5cm]{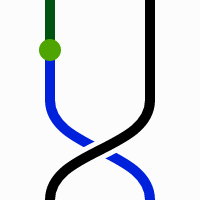}
\raisebox{0.45cm}{\Huge ]}
\raisebox{0.6cm}{$=$}
\raisebox{0.45cm}{\Huge [} 
\includegraphics[width=1cm,height=1.5cm]{pullthroughwith2morphstart}
\raisebox{0.6cm}{$\overset{PU_{A,f}}{\Rightarrow}$}\includegraphics[width=1cm,height=1.5cm]{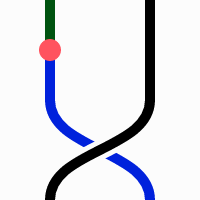}
\raisebox{0.6cm}{$=$}\includegraphics[width=1cm,height=1.5cm]{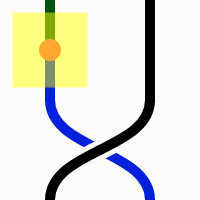}
\raisebox{0.6cm}{$\overset{\alpha}{\Rightarrow}$}\includegraphics[width=1cm,height=1.5cm]{pullthroughw2morphend}
\raisebox{0.45cm}{\Huge ]}
\end{tabular}
\end{center}
\item $(\Downarrow \otimes \;\cdot)$. As above, but for $f,g: A \to C$ and with a pull-over. 
\item $(\rightarrow \rightarrow \otimes \;\cdot)$. We have the following equality for morphisms $f:A \to C, g: C \to E$, where the composition $c$ is a strict equality since we are in a Gray monoid.
\begin{center}
\begin{tabular}{l l}
\raisebox{0.45cm}{\Huge [} \includegraphics[width=1cm,height=1.5cm]{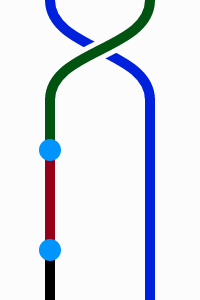}
\raisebox{0.6cm}{$\overset{PO_{g,A}}{\Rightarrow}$}\includegraphics[width=1cm,height=1.5cm]{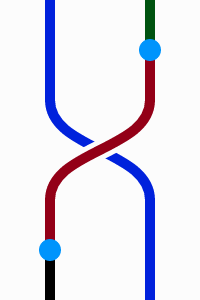}
\raisebox{0.6cm}{$\overset{PO_{f,A}}{\Rightarrow}$}\includegraphics[width=1cm,height=1.5cm]{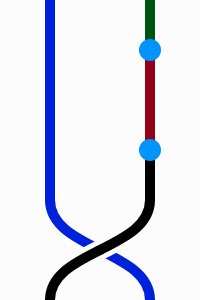}
\raisebox{0.6cm}{$=$}\includegraphics[width=1cm,height=1.5cm]{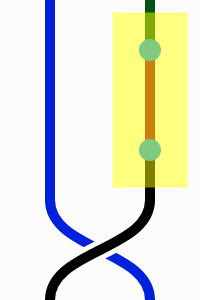}
\raisebox{0.6cm}{$\overset{c}{=}$}\includegraphics[width=1cm,height=1.5cm]{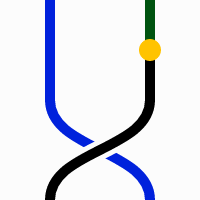}
\raisebox{0.45cm}{\Huge ]}\\
\raisebox{0.6cm}{$=$}
\raisebox{0.45cm}{\Huge [} \includegraphics[width=1cm,height=1.5cm]{compopulloverstart}
\raisebox{0.6cm}{$=$}\includegraphics[width=1cm,height=1.5cm]{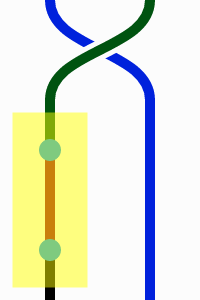}
\raisebox{0.6cm}{$\overset{c}{=}$}\includegraphics[width=1cm,height=1.5cm]{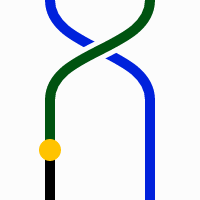}
\raisebox{0.6cm}{$\overset{PO_{g \circ f,A}}{\Rightarrow}$}
\includegraphics[width=1cm,height=1.5cm]{compopulloverend}
\raisebox{0.45cm}{\Huge ]}
\end{tabular}
\end{center}
\item $( \cdot\; \otimes \rightarrow \rightarrow)$. As above, but with $f:B \to D, g: D \to F$ and pull-unders.
\item \emph{PT-B}. For any 1-morphism $g: B \to D$, the following 2-morphisms are equal:
\begin{center}
\begin{tabular}{l l}
\raisebox{0.45cm}{\Huge [}
\includegraphics[width=1cm,height=1.5cm]{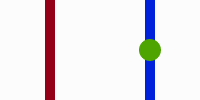}
\raisebox{0.6cm}{$\overset{i^+_{AD}}{\Rightarrow}$}\includegraphics[width=1cm,height=1.5cm]{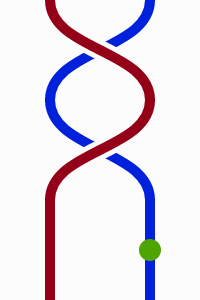}
\raisebox{0.6cm}{$\overset{PU_{A,g}}{\Rightarrow}$}\includegraphics[width=1cm,height=1.5cm]{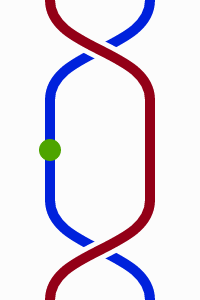}
\raisebox{0.45cm}{\Huge ]}
\raisebox{0.6cm}{$=$}
\raisebox{0.45cm}{\Huge [}
\includegraphics[width=1cm,height=1.5cm]{CYCstart}
\raisebox{0.6cm}{$\overset{i^+_{AB}}{\Rightarrow}$}\includegraphics[width=1cm,height=1.5cm]{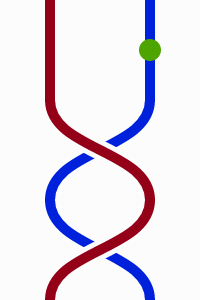}
\raisebox{0.6cm}{$\overset{PU^{-1}_{g,A}}{\Rightarrow}$}\includegraphics[width=1cm,height=1.5cm]{CYCend}
\raisebox{0.45cm}{\Huge ]}
\end{tabular}
\end{center}
Similar equations hold where $i^{+}_{A,B}$ is changed for $i^{-}_{A,B}$, and/or the 1-cell pulled through is $f: A \to C$ rather than $g: B \to D$.
\item \emph{ADJ}. $i^{+}_{A,B}$ and $i^{-}_{A,B}$ are the units of adjoint equivalences.
\item $S^+ = S^-$. The two possible instantiations of the braid move $\sigma_i \sigma_{i+1} \sigma_i \to \sigma_{i+1} \sigma_i \sigma_{i+1}$ are equal.
\end{itemize}
\end{itemize}
\end{definition}
\noindent
The axioms we have specified are  those of a twice-degenerate semistrict 4-category in the definition of Bar and Vicary~\cite{Bar2017} and are slightly stricter than in any previous definition of a braided Gray monoid~ \cite{Kapranov1994,Baez1996,Crans1998,Gurski2011}. The two points of difference with the strictest previous definition~\cite{Crans1998} are the following.

\begin{itemize}[leftmargin=*]
\item The `hexagonators' in the Bar-Vicary definition are trivial; this corresponds to strictness of composition in the semistrict 4-category. The strict monoidality of the braiding inverse-insert 2-cells follows from this.
\item There is an extra axiom in the Bar-Vicary definition, PT-B, which relates a braiding inverse-insert and pullthrough above to a braiding inverse-insert and pullthrough below. This axiom follows from Homotopy Generator VI in the definition of Bar and Vicary.
\end{itemize}
The Crans definition with trivial hexagonators remains semistrict.
\begin{theorem}[Semistrictness with trivial hexagonators]
For any computadic Crans braided monoidal bicategory, the quotient homomorphism $\phi$ which identifies all braiding 1-cells with the corresponding `expanded' composite of braidings of generating 1-cells~\eqref{eq:trivialhexagonatorbraidings}, and sends all hexagonators to the identity, is a braided monoidal biequivalence. 
\end{theorem}
\begin{proof}
See Appendix~\ref{app:semistrictnesshexagonators}.
\end{proof}
Having resolved the issue of the trivial hexagonators, we turn to the PT-B equality. It seems that this equality, which is topologically well-motivated, is not implied by the other equalities. The arguments for the correctness of Bar and Vicary's definition support it; indeed, the proof that every equivalence in a semistrict 4-category can be promoted to an adjoint equivalence satisfying the butterfly equations depends on Homotopy Generator VI  \cite{Bar2017}. It is also an essential ingredient in our algorithm for putting 1-cells in TSNF (Theorem~\ref{tsnfprocedure}). 

PT-B may have been omitted previously because previous authors worked only with braided monoidal bicategories with no generating 1-cells. In that case, PT-B follows straightforwardly from ADJ, rendering a separate axiom unnecessary. In the symmetric setting, PT-B is implied by stronger coherence results. 

\subsubsection{Symmetric Gray monoids}

The axioms we have chosen for symmetric Gray monoids are somewhat weaker than those of the quasistrict definition of Schommer-Pries \cite{Schommer-Pries2009}. The primary advantage of using a weaker definition is that our proofs of fullness and functoriality of the biequivalences we define apply equally to non-braided, braided and symmetric Gray monoids. 

\begin{definition}
A symmetric Gray monoid is a braided Gray monoid with the following additional structural generating cells and equalities.

\begin{itemize}[label={-},leftmargin=*]
\item \underline{Additional generating 2-cells}:
\begin{itemize}[leftmargin=*,label={-}]
\item A `syllepsis' 2-isomorphism $\sigma_{AB}$ for all 0-cells $A,B$, which controls the symmetry of the braiding:
\begin{center}
\begin{tabular}{l l}
\includegraphics[width=1cm,height=1.5cm]{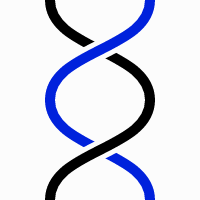}
\raisebox{0.6cm}{$\overset{\sigma_{AB}}{\Rightarrow}$}\includegraphics[width=1cm,height=1.5cm]{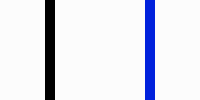}
\end{tabular}
\end{center}
As with the braiding inverse-insert, the syllepsis on monoidal products is the composite of the syllepses on the factors.
\end{itemize}

\item \underline{Additional equalities of 2-cells}:
\begin{itemize}[leftmargin=*,label={-}]
\item \emph{PT-SYL}. For every 0-cell $A$ and 1-cell $g: B \to D$, the following equality:
\begin{center}
\begin{tabular}{l l}
\raisebox{0.45cm}{\Huge [} \includegraphics[width=1cm,height=1.5cm]{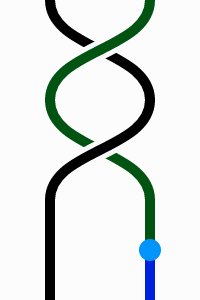}
\raisebox{0.6cm}{$\overset{\sigma_{AD}}{\Rightarrow}$}\includegraphics[width=1cm,height=1.5cm]{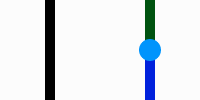}\raisebox{0.45cm}{\Huge ]}
\raisebox{0.6cm}{$=$}
\raisebox{0.45cm}{\Huge [} \includegraphics[width=1cm,height=1.5cm]{PTSYLAgstart}
\raisebox{0.6cm}{$\overset{PU_{A,g}}{\Rightarrow}$}\includegraphics[width=1cm,height=1.5cm]{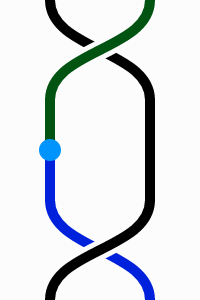}
\raisebox{0.6cm}{$\overset{PO_{g,B}}{\Rightarrow}$}\includegraphics[width=1cm,height=1.5cm]{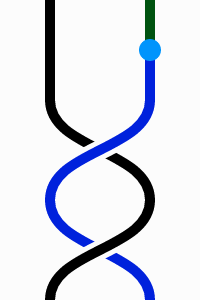}\raisebox{0.6cm}{$\overset{\sigma_{AB}}{\Rightarrow}$}\includegraphics[width=1cm,height=1.5cm]{PTSYLAgend}\raisebox{0.45cm}{\Huge ]}
\end{tabular}
\end{center}
The equality PT-SYL$_{f,B}$ is defined similarly for all 0-cells $B$ and 1-cells $f: A \to C$.
\item \emph{SYM}. The following 2-morphisms are equal:
\begin{center}
\begin{tabular}{l l}
\raisebox{0.45cm}{\Huge [} \includegraphics[width=1cm,height=1.5cm]{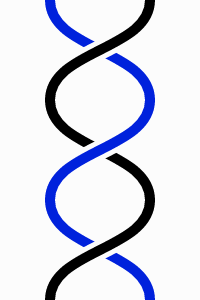}
\raisebox{0.6cm}{$=$}\includegraphics[width=1cm,height=1.5cm]{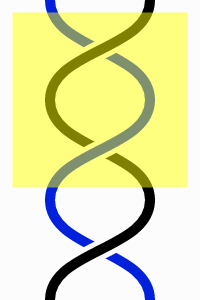}
\raisebox{0.6cm}{$\overset{\sigma_{B,A}}{\Rightarrow}$}\includegraphics[width=1cm,height=1.5cm]{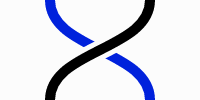}\raisebox{0.45cm}{\Huge ]}
\raisebox{0.6cm}{$=$}
\raisebox{0.45cm}{\Huge [} \includegraphics[width=1cm,height=1.5cm]{syllepsisissymmstart}
\raisebox{0.6cm}{$=$}\includegraphics[width=1cm,height=1.5cm]{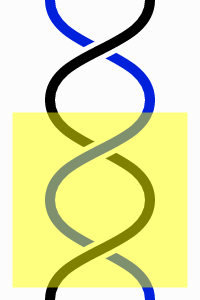}
\raisebox{0.6cm}{$\overset{\sigma_{A,B}}{\Rightarrow}$}\includegraphics[width=1cm,height=1.5cm]{syllepsisissymmend}\raisebox{0.45cm}{\Huge ]}
\end{tabular}
\end{center}
\end{itemize}
\end{itemize}
\end{definition}
Here, rather than take the braiding inverse-inserts and syllepses to be identities, as in the quasistrict definition of Schommer-Pries, we only require PT-B, which follows from triviality of the braiding inverse-inserts in the quasistrict definition; and PT-SYL, another axiom similar to PT-B but involving the syllepsis, which follows from triviality of the syllepsis in the quasistrict definition. The semistrictness of this definition of symmetric Gray monoids is therefore implied by the semistrictness of Schommer-Pries' quasistrict definition~\cite[Theorem 2.96]{Schommer-Pries2009}.

\begin{definition}
Every computad for a naked Gray monoid can also be taken as a computad for a braided or a symmetric Gray monoid; likewise, every computad for a braided Gray monoid can also be taken as a computad for a symmetric Gray monoid.
\end{definition}

\subsection{Coherence for braided and symmetric Gray monoids}\label{sec:symmgraymonoidstechniques}

In this section we treat coherence results for braided and symmetric Gray monoids which will be used in the proof of our main theorem.  We first recall two important results about braided and symmetric Gray monoids; here we only state the computadic versions, although they hold in greater generality.

\begin{theorem}[{\cite[Theorem 25]{Gurski2011}}]\label{gurskicoherencethmbraids}
Let $C$ be a braided Gray monoid computad with no non-structural generating 2-cells, whose non-structural generating 1-cells all have exactly one 0-cell as source and one 0-cell as target. In the braided Gray monoid generated from $C$, all parallel 2-cells are equal, and two structural 1-cells are isomorphic iff they have the same underlying element of the braid group.
\end{theorem}

\begin{theorem}[{\cite[Theorem 1.23]{Gurski2013}}]\label{gurskiosornocoherencethm}
Let $C$ be a symmetric Gray monoid computad with no non-structural generating 2-cells, whose non-structural generating 1-cells all have exactly one 0-cell as source and one 0-cell as target. In the symmetric Gray monoid generated from $C$, all parallel 2-cells are equal, and two structural 1-cells are isomorphic iff they have the same underlying permutation.
\end{theorem}
\noindent We now introduce a useful normal form for 1-cells in braided and symmetric Gray monoids.

\begin{definition}[Output string]
Consider a 1-cell diagram in a braided Gray monoid. For any generating 1-cell $N$ in the diagram with a single generating 0-cell in its output, we define its \emph{output string} to be the string extending from $N$ to the next non-structural 1-cell to which the string is input; or to the roof of the diagram if it is not input to another non-structural 1-cell.
\end{definition}
\noindent
Let $N$ be a non-structural generating 1-cell in some frame of a movie. Provided that no non-structural 2-cells occur on rectangular subregions containing $N$, it is always possible to identify $N$ in previous and subsequent frames, where it may have been moved by interchangers and pullthroughs. We may therefore speak about $N$ as being in a clip, rather than just in a frame.

\begin{definition}[TSNF]\label{tsnfdefinition}
Let $M$ be a clip in a computadic braided Gray monoid, containing some non-structural generating 1-cell $N$ whose output is a single generating 0-cell. We say that $N$ is in \emph{top string normal form} (TSNF) in $M$ if no generating 2-cells act on rectangular subregions containing the output string of $N$ during $M$. 
\end{definition}

\begin{example}
In the clip below the output string of the lowest generating 1-cell is highlighted. The lowest generating 1-cell is not in TSNF because a braiding cancellation, a braiding inverse-insert, and then a pullthrough occur on the output string during the clip.
\begin{center}
\raisebox{0.45cm}{\Huge [} \includegraphics[width=1cm,height=1.5cm]{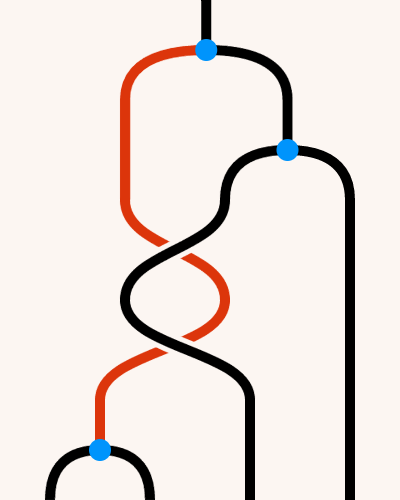}
\raisebox{0.6cm}{$\Rightarrow$}\includegraphics[width=1cm,height=1.5cm]{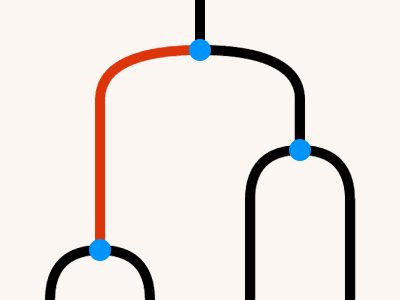}
\raisebox{0.6cm}{$\Rightarrow$}\includegraphics[width=1cm,height=1.5cm]{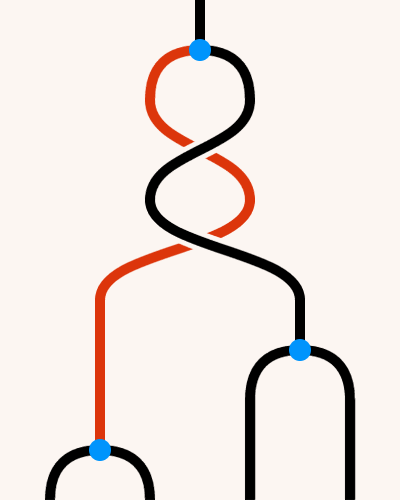}
\raisebox{0.6cm}{$\Rightarrow$}\includegraphics[width=1cm,height=1.5cm]{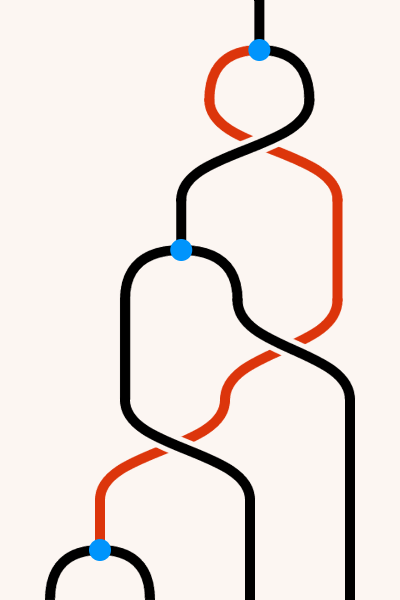}\raisebox{0.45cm}{\Huge ]}
\end{center}
\end{example}
\noindent
As we will now show, we can always put a generating 1-cell in TSNF, provided that no non-structural 2-cells occur on a rectangular subregion containing some part of the output string. First we define a useful movie rewriting technique.

\begin{definition}[Insert IPI]
Let $N$ be a 1-cell in a clip in a braided or symmetric Gray monoid. We can rewrite the clip by using Type II rewrites and invertibility of the pullthroughs to move $N$ up or down the frame. We say that we `insert IPI'. For example, here we insert IPI to move the lowest generating 1-cell to the top of the frame and back to the bottom again:
\begin{center}
\raisebox{0.45cm}{\huge [}\includegraphics[width=.9cm,height=1.1cm]{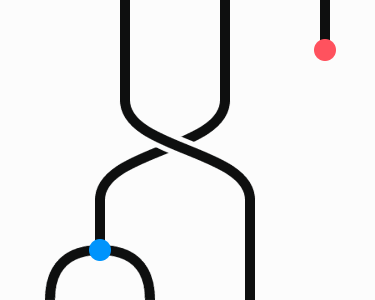}
\raisebox{0.45cm}{\huge ]}
\raisebox{0.45cm}{$=$}
\raisebox{0.45cm}{\huge [}
\includegraphics[width=.9cm,height=1.1cm]
{insertipi1}
\raisebox{0.6cm}{$\Rightarrow$}\includegraphics[width=.9cm,height=1.1cm]
{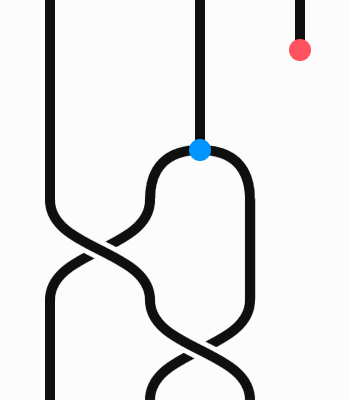}
\raisebox{0.6cm}{$\Rightarrow$}\includegraphics[width=.9cm,height=1.1cm]
{insertipi1}
\raisebox{0.45cm}{\huge ]} 
\raisebox{0.45cm}{$=$}
\raisebox{0.45cm}{\huge [}\includegraphics[width=.9cm,height=1.1cm]{insertipi1}
\raisebox{0.6cm}{$\Rightarrow$}\includegraphics[width=.9cm,height=1.1cm]
{insertipi2}
\raisebox{0.6cm}{$\Rightarrow$}\includegraphics[width=.9cm,height=1.1cm]
{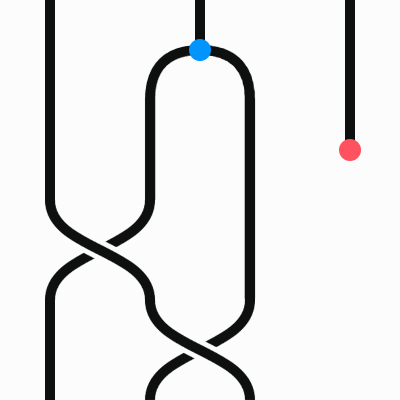}
\raisebox{0.6cm}{$\Rightarrow$}\includegraphics[width=.9cm,height=1.1cm]
{insertipi2}
\raisebox{0.6cm}{$\Rightarrow$}\includegraphics[width=.9cm,height=1.1cm]
{insertipi1}
\raisebox{0.45cm}{\huge ]} 
\end{center} 
\end{definition}
\noindent
Now we state the main results of this section. We postpone their proofs to Appendix~\ref{sec:proofsforgraymoncohappendix}.
\begin{theorem}[Putting a 1-cell in TSNF]\label{tsnfprocedure}
Let $M$ be a clip in a computadic braided Gray monoid. Let $N$  be a non-structural generating 1-cell whose output is a single generating 0-cell. If no non-structural 2-cells occur on a rectangular subregion containing the output string of $N$ during $M$, then there exists a series of rewrites to put $N$ in TSNF.
\end{theorem} 
\begin{theorem}[Extended coherence for computadic braided and symmetric Gray monoids]\label{gurskiosornocoherencethmextended}
Let $C$ be a computad for a braided or symmetric Gray monoid with no nonstructural generating 2-cells, whose generating 1-cells all have a single generating 0-cell as output. Then  all parallel 2-cells in the braided or symmetric Gray monoid generated from $C$ are equal.
\end{theorem}

\subsection{Computads for braided and symmetric pseudomonoids}\label{sec:presentations}

The following computads follow the definitions of Day and Street \cite{Day1997}.

\begin{definition}\label{pseudomonoidpresentation}
The \emph{pseudomonoid computad} $\mathcal{P}$ is the naked Gray monoid computad defined as follows.
\begin{itemize}[leftmargin=*,label={}]
\item \underline{0-cells}: $\{C\}$.
\item \underline{1-cells}: $m: C \otimes C \to C$ and $u: I \to C$.
\begin{center}
\begin{tabular}{l l}
\raisebox{0.4cm}{$m=$}
\includegraphics[width=1.75cm,height=1cm]{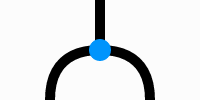}&\qquad
\raisebox{0.4cm}{$u=$}\includegraphics[width=1cm,height=1cm]{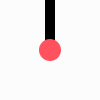} 
\end{tabular}
\end{center}
\item \underline{2-cells}: $\alpha$ (associator), $\lambda$ (left unitor) and $\rho$ (right unitor), all isomorphisms (Definition~\ref{def:isomorphism}).
\begin{center}
\begin{tabular}{l l l l}
\includegraphics[width=2cm,height=1.5cm]{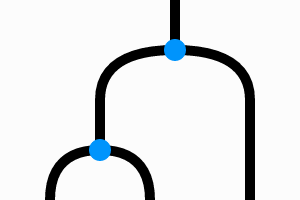}
\raisebox{0.6cm}{$\overset{\alpha}{\Rightarrow}$}\includegraphics[width=2cm,height=1.5cm]{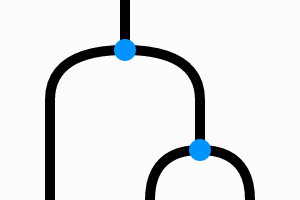} &\qquad
\includegraphics[width=1cm,height=1.5cm]{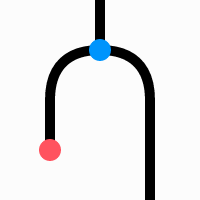}
\raisebox{0.6cm}{$\overset{\lambda}{\Rightarrow}$}\includegraphics[width=1cm,height=1.5cm]{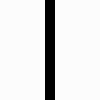}&\qquad
\includegraphics[width=1cm,height=1.5cm]{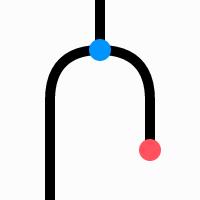}
\raisebox{0.6cm}{$\overset{\rho}{\Rightarrow}$}\includegraphics[width=1cm,height=1.5cm]{identity}
\end{tabular}
\end{center}
We will occasionally call the unitors $\lambda, \rho$ \emph{unit destruction operators} and the inverse unitors $\lambda^{-1}, \rho^{-1}$ \emph{unit creation operators}.
\item \underline{Equalities}:
\begin{itemize}[leftmargin=*,label={-}]
\item Pentagon: \begin{center}
\begin{tabular}{l l}
\raisebox{0.45cm}{\Huge [} \includegraphics[width=1cm,height=1.5cm]{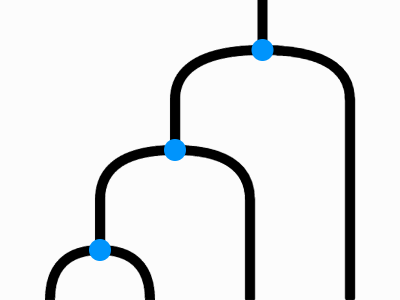}
\raisebox{0.6cm}{$\overset{\alpha}{\Rightarrow}$}\includegraphics[width=1cm,height=1.5cm]{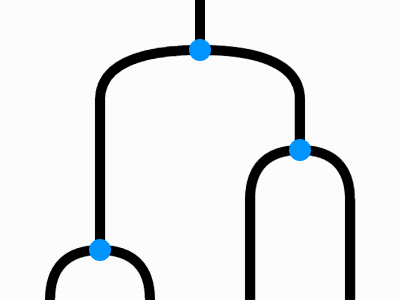}
\raisebox{0.6cm}{$\overset{\iota}{\Rightarrow}$}\includegraphics[width=1cm,height=1.5cm]{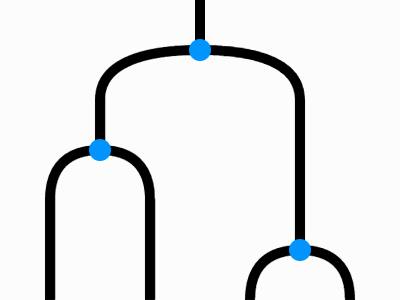}
\raisebox{0.6cm}{$\overset{\alpha}{\Rightarrow}$}\includegraphics[width=1cm,height=1.5cm]{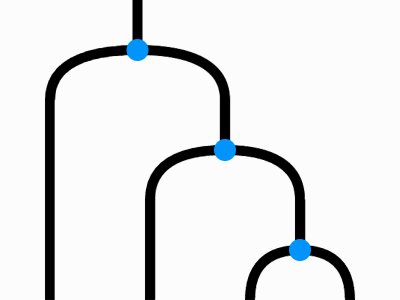}
\raisebox{0.45cm}{\Huge ]}
\raisebox{0.6cm}{$=$}
\raisebox{0.45cm}{\Huge [} \includegraphics[width=1cm,height=1.5cm]{pentagonstart}
\raisebox{0.6cm}{$\overset{\alpha}{\Rightarrow}$}\includegraphics[width=1cm,height=1.5cm]{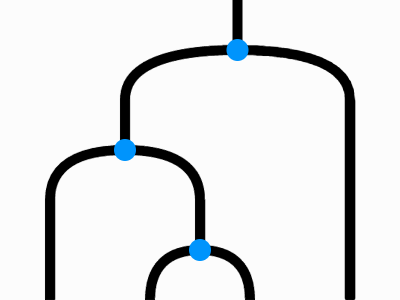}
\raisebox{0.6cm}{$\overset{\alpha}{\Rightarrow}$}\includegraphics[width=1cm,height=1.5cm]{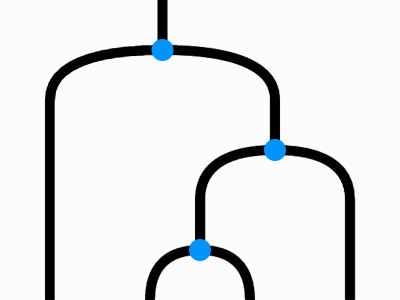}
\raisebox{0.6cm}{$\overset{\alpha}{\Rightarrow}$}
\includegraphics[width=1cm,height=1.5cm]{pentagonend}
\raisebox{0.45cm}{\Huge ]}
\end{tabular}
\end{center}
\item Triangle:\begin{center}
\begin{tabular}{l l}
\raisebox{0.45cm}{\Huge [} \includegraphics[width=1cm,height=1.5cm]{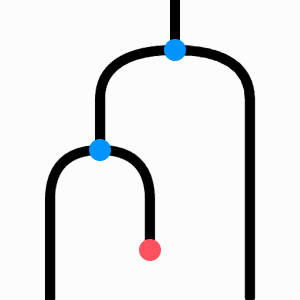}
\raisebox{0.6cm}{$\overset{\alpha}{\Rightarrow}$}\includegraphics[width=1cm,height=1.5cm]{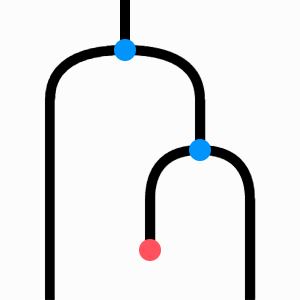}
\raisebox{0.6cm}{$\overset{\lambda}{\Rightarrow}$}\includegraphics[width=1cm,height=1.5cm]{mult}
\raisebox{0.45cm}{\Huge ]}
\raisebox{0.6cm}{$=$}
\raisebox{0.45cm}{\Huge [} \includegraphics[width=1cm,height=1.5cm]{trianglestart}
\raisebox{0.6cm}{$\overset{\rho}{\Rightarrow}$}\includegraphics[width=1cm,height=1.5cm]{mult}
\raisebox{0.45cm}{\Huge ]}
\end{tabular}
\end{center}
\end{itemize}
\end{itemize}
\end{definition}

\begin{definition}\label{braidedpseudomonoidpres}
The \emph{braided pseudomonoid computad} $\mathcal{P}^{\text{br}}$ is the braided Gray monoid computad with all the generating cells of $\mathcal{P}$, and the following additional data.

\begin{itemize}[leftmargin=*,label={}]
\item \underline{2-cells}: An isomorphism $c$ (the commutator):
\begin{center}
\includegraphics[width=.75cm,height=.75cm]{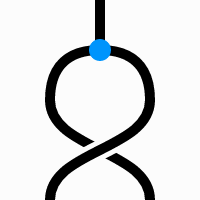}
\raisebox{0.3cm}{$\overset{c}{\Rightarrow}$}\includegraphics[width=.75cm,height=.75cm]{mult}
\end{center}
\item \underline{Equalities}:
\begin{itemize}[leftmargin=*,label={-}]
\item Hexagon 1:\begin{center}
\raisebox{0.45cm}{\Huge [} \includegraphics[width=.75cm,height=1.25cm]{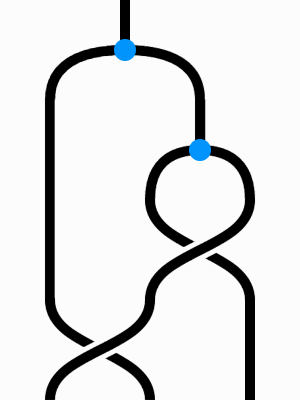}
\raisebox{0.6cm}{$\overset{c}{\Rightarrow}$}\includegraphics[width=.75cm,height=1.25cm]{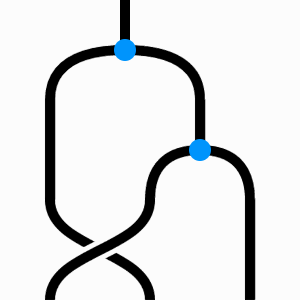}
\raisebox{0.6cm}{$\overset{\alpha^{-1}}{\Rightarrow}$}\includegraphics[width=.75cm,height=1.25cm]{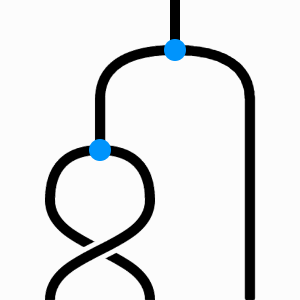}
\raisebox{0.6cm}{$\overset{c}{\Rightarrow}$}\includegraphics[width=.75cm,height=1.25cm]{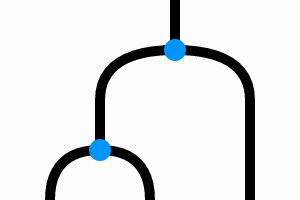}
\raisebox{0.45cm}{\Huge ]}
\raisebox{0.6cm}{$=$}
\raisebox{0.45cm}{\Huge [} \includegraphics[width=.75cm,height=1.25cm]{hexagononeend}
\raisebox{0.6cm}{$\overset{\alpha^{-1}}{\Rightarrow}$}\includegraphics[width=.75cm,height=1.25cm]{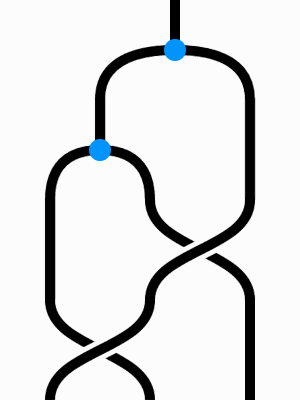}
\raisebox{0.6cm}{$\overset{PU^{-1}_{C,m}}{\Rightarrow}$}\includegraphics[width=.75cm,height=1.25cm]{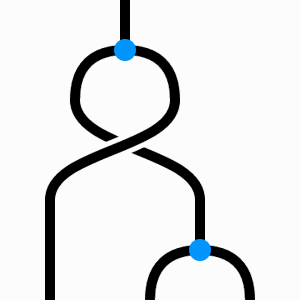}
\raisebox{0.6cm}{$\overset{c}{\Rightarrow}$}\includegraphics[width=.75cm,height=1.25cm]{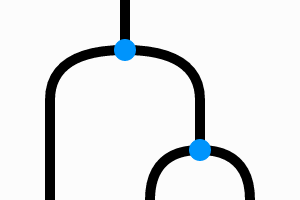}
\raisebox{0.6cm}{$\overset{\alpha^{-1}}{\Rightarrow}$}\includegraphics[width=.75cm,height=1.25cm]{hexagononestart}
\raisebox{0.45cm}{\Huge ]}
\end{center}
\item Hexagon 2:
\begin{center}
\raisebox{0.45cm}{\Huge [} \includegraphics[width=.75cm,height=1.25cm]{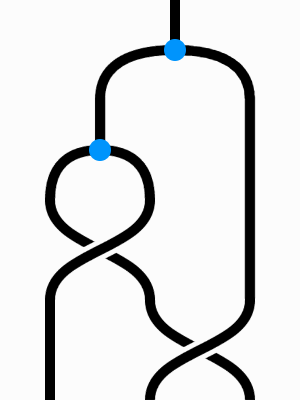}
\raisebox{0.6cm}{$\overset{\alpha}{\Rightarrow}$}\includegraphics[width=.75cm,height=1.25cm]{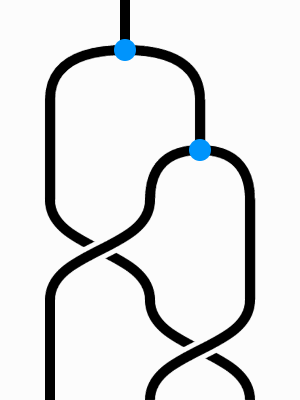}
\raisebox{0.6cm}{$\overset{PO^{-1}_{m,C}}{\Rightarrow}$}\includegraphics[width=.75cm,height=1.25cm]{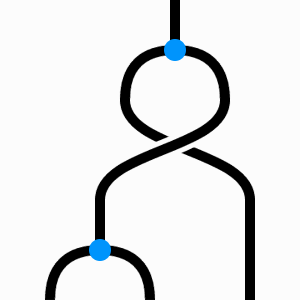}
\raisebox{0.6cm}{$\overset{c}{\Rightarrow}$}\includegraphics[width=.75cm,height=1.25cm]{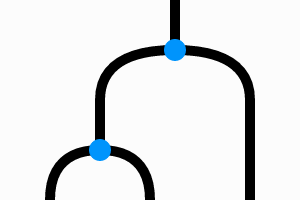}
\raisebox{0.6cm}{$\overset{\alpha}{\Rightarrow}$}\includegraphics[width=.75cm,height=1.25cm]{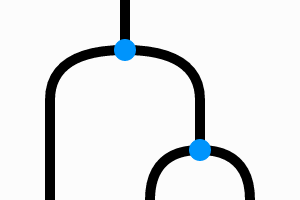}
\raisebox{0.45cm}{\Huge ]}
\raisebox{0.6cm}{$=$}
\raisebox{0.45cm}{\Huge [} \includegraphics[width=.75cm,height=1.25cm]{hexagontwoend}
\raisebox{0.6cm}{$\overset{c}{\Rightarrow}$}\includegraphics[width=.75cm,height=1.25cm]{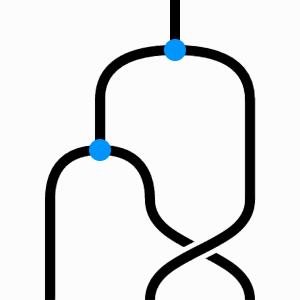}
\raisebox{0.6cm}{$\overset{\alpha}{\Rightarrow}$}\includegraphics[width=.75cm,height=1.25cm]{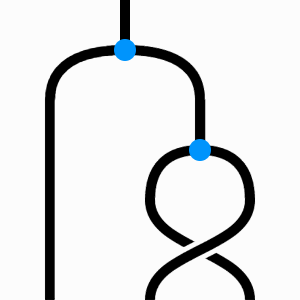}
\raisebox{0.6cm}{$\overset{c}{\Rightarrow}$}\includegraphics[width=.75cm,height=1.25cm]{hexagontwostart}
\raisebox{0.45cm}{\Huge ]}
\end{center}
\end{itemize}
\end{itemize}
\end{definition}

\begin{definition}\label{symmpseudomonoidpres}
The \emph{symmetric pseudomonoid}  computad $\mathcal{P}^{\text{sym}}$ is the symmetric Gray monoid computad with all the generating cells of $\mathcal{P}^{\text{br}}$, and the following additional data.
\begin{itemize}[leftmargin=*,label={}]
\item \underline{Equalities}:
\begin{itemize}[leftmargin=*,label={-}]
\item Symmetry:
\begin{center}
\raisebox{0.3cm}{\Huge [} \includegraphics[width=.75cm,height=1.25cm]{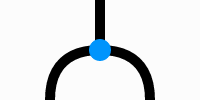} \raisebox{0.3cm}{\Huge ]}
\raisebox{0.6cm}{=}
\raisebox{0.3cm}{\Huge [} \includegraphics[width=.75cm,height=1.25cm]{symmetricpseudomonoidstart}
\raisebox{0.6cm}{$\overset{\sigma^{-1}}{\Rightarrow}$}\includegraphics[width=.75cm,height=1.25cm]{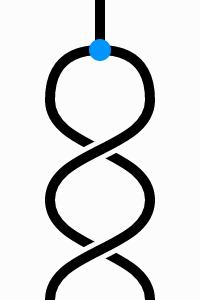}
\raisebox{0.6cm}{$\overset{c}{\Rightarrow}$}\includegraphics[width=.75cm,height=1.25cm]{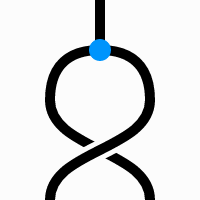}
\raisebox{0.6cm}{$\overset{c}{\Rightarrow}$}\includegraphics[width=.75cm,height=1.25cm]{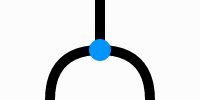}\raisebox{0.3cm}{\Huge ]}
\end{center}
\end{itemize}
\end{itemize}
\end{definition}

\subsection{Theories of monoids}\label{sec:theoryofmonoids}

Before concluding the background section, we review the theory of PROs, PROBs and PROPs for monoids and commutative monoids which was summarised in Table~\ref{tbl:monoidsresultsinintro} of the introduction. 

\begin{definition}
The \emph{monoid computad} ${\bf M}$ is generating data for a PRO, derived from the pseudomonoid computad by considering the  generating 2-cells as equalities of 1-cells and forgetting equalities of 2-cells.
\end{definition}

\begin{definition}
The \emph{commutative monoid computad} ${\bf CM}$ is generating data for a PROB, derived from the pseudomonoid computad by considering the generating 2-cells as equalities of 1-cells and forgetting equalities of 2-cells.
\end{definition}
\noindent
We define a combinatorial category isomorphic to the PRO for monoids.

\begin{definition}
The objects of ${\bf \Delta}$ are natural numbers $\underline{n} \in \mathbb{N}$, and its 1-cells $\underline{m} \to \underline{n}$ are monotone functions $\{1,\dots,m\} \to \{1,\dots,n\}$ (i.e. functions satisfying $f(i)<f(j) \;\,\forall\, i<j$). Composition is composition of functions, and monoidal product is the coproduct in ${\bf Set}$.
\end{definition}
\noindent
In order to identify 1-cells in the PRO for monoids with monotone functions, one considers connectedness between inputs and outputs of a 1-cell. In the absence of a braiding, this must correspond to a monotone function. An output which is not connected to any input must come from a unit. An example is shown in Figure \ref{deltatreesexample}. The following proposition makes this precise.
\begin{proposition}\label{proformonoidsprop}
The PRO on the monoid computad, {\bf FM}, is isomorphic to ${\bf \Delta}$. 
\end{proposition}

\begin{proof}
We define a functor ${\bf \Delta} \to {\bf FM}$.
\begin{itemize}
\item \emph{On 0-cells}: The objects of both categories are natural numbers; let the map on 0-cells be the identity function.
\item \emph{On 1-cells}: Given a function $f:\underline{m} \to \underline{n}$, we define a morphism $\tilde{f}$ in {\bf FM} as follows. Let $\mu: \underline{2} \to \underline{1}$ be the multiplication 1-cell in ${\bf FM}$, and $u: \underline{0} \to \underline{1}$ be the unit. Let $p_i = |f^{-1}(i)|$ be the cardinality of the preimage of $i \in \{1,\dots,n\}$. Let $\mu^{n}: \underline{n} \to \underline{1}$ be the composition of $n-1$ multiplications, left bracketed; for example, $\mu^{4}=\mu \circ (\mu \otimes \text{Id}) \circ (\mu \otimes \text{Id} \otimes \text{Id})$. We set $\mu^1 = \text{Id}$ and $\mu^0 = u$.  Then we define
\begin{equation}\label{eq:monotonefunctionimage}
\tilde{f} = (\mu^{p_1} \otimes \cdots \otimes \mu^{p_n}).
\end{equation}
\end{itemize}
It is easy to check that this is an isomorphism of categories~\cite[Section 2.1]{Davydov2010}.
\begin{figure}
\centering
\begin{tikzpicture}
	\begin{pgfonlayer}{nodelayer}
		\node [style=none] (0) at (0, -0) {};
		\node [style=none] (1) at (1, -0) {};
		\node [style=rn] (2) at (1, 2.25) {};
		\node [style=rn] (3) at (0.5, 1.5) {};
		\node [style=none] (4) at (2, -0) {};
		\node [style=gn] (5) at (2.5, 2.75) {};
		\node [style=rn] (6) at (4, 3.5) {};
		\node [style=none] (7) at (3.5, -0) {};
		\node [style=none] (8) at (4.5, -0) {};
		\node [style=none] (9) at (1, 5) {};
		\node [style=none] (10) at (2.5, 5) {};
		\node [style=none] (11) at (4, 5) {};
	\end{pgfonlayer}
	\begin{pgfonlayer}{edgelayer}
		\draw [style=simple, bend left, looseness=0.75] (0.center) to (3);
		\draw [style=simple, bend right=15, looseness=1.00] (1.center) to (3);
		\draw [style=simple, bend right=15, looseness=1.00] (4.center) to (2);
		\draw [style=simple, bend left=15, looseness=0.75] (7.center) to (6);
		\draw [style=simple, bend right=15, looseness=0.75] (8.center) to (6);
		\draw [style=simple] (6) to (11.center);
		\draw [style=simple] (5) to (10.center);
		\draw [style=simple, bend left=15, looseness=1.25] (3) to (2);
		\draw [style=simple] (2) to (9.center);
	\end{pgfonlayer}
\end{tikzpicture}
\caption{The image of the function $f: \underline{5} \to \underline{3}$, $f(1)=f(2)=f(3)=1$, $f(4)=f(5)=3$ under the isomorphism ${\bf \Delta} \to {\bf FM}$.}
\label{deltatreesexample}
\end{figure}
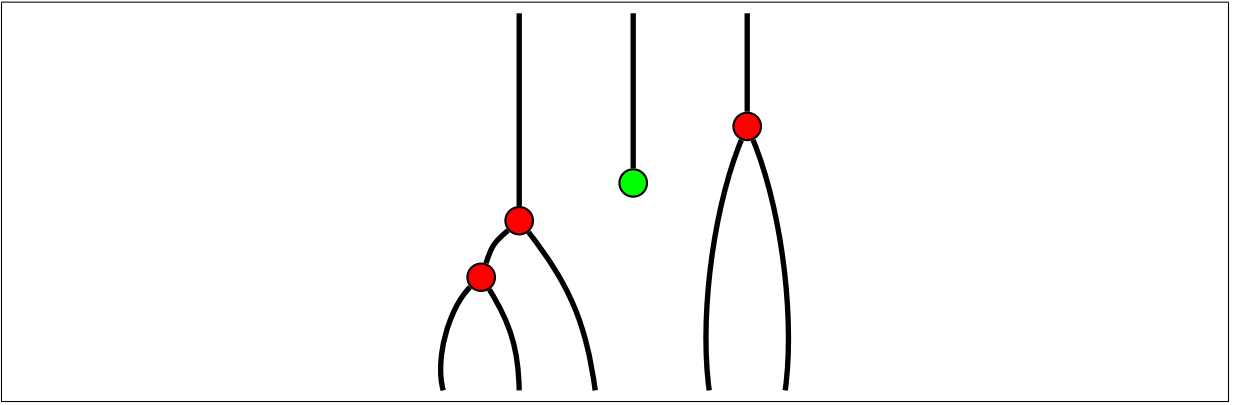
\end{proof}

\begin{definition}
We call the compositions $\mu^n$ \emph{multiplication trees}, or just~\emph{trees}.
\end{definition}

We now consider the free braided monoidal category (PROB) ${\bf F^{br}M}$ on the monoid computad. We can always pull the monoid structure through any braidings, which allows us to split any 1-cell into a braid followed by a monoid map. The 1-cells of the isomorphic combinatorial category will therefore have a braid part and a monotone function part. `Pulling through' is formally a distributive law~\cite{Day2003} between the braid structure and the monoid structure, which may be used to define composition. The following lemma and proposition make this precise. 

We write the image of the embedding ${\bf \Delta} \hookrightarrow {\bf F^{br}M}$, which picks out the morphisms without braiding, as ${\bf \Delta} \subset {\bf F^{br}M}$. Let ${\bf B}$ be the free braided monoidal category on a single object; we write the image of the  embedding ${\bf B} \hookrightarrow {\bf F^{br}M}$, which picks out the morphisms made entirely of braids, as ${\bf B} \subset {\bf F^{br}M}$.

\begin{lemma}[{\cite[Section 4]{Day2003}}]\label{distributivelawlemma}
For any $b \in {\bf B}$, $f \in {\bf \Delta}$, there exist unique morphisms $b' \in {\bf B}$, $f' \in {\bf \Delta}$ such that $b \circ f = f' \circ b'$ in ${\bf F^{br}M}$; there is a corresponding distributive law
\begin{equation}\label{eq:distlawbraid} B_n \times {\bf \Delta}(m,n) \overset{\delta_{m,n}}{\to} {\bf \Delta}(m,n) \times B_m
\end{equation}
where $B_n$ is the braid group on $n$ points and ${\bf \Delta}(m,n)$ are the morphisms $\underline{m}\to\underline{n}$ in ${\bf \Delta}$.
\end{lemma}
\noindent
Under the distributive law, let $\delta_{m,n}^{B}:= \pi_{B_m} \circ \delta(m,n)$ be the braid part of the image, and let $\delta_{m,n}^{\Delta}:= \pi_{\Delta} \circ \delta(m,n)$ be the monoid part. 
\begin{definition}\label{bdeltadefinition}
The braided monoidal category ${\bf B \Delta}$ has objects natural numbers. Morphisms $\underline{m} \to \underline{n}$ are pairs $(\sigma,f)$, where $\sigma \in B_m$ and $f \in {\bf \Delta}(m,n)$. Composition is defined using the distributive law and composition in ${\bf \Delta}$ and the braid group in the following way. For $(\sigma_1,f_1) \in  \Hom(\underline{m},\underline{n})$ and $(\sigma_2,f_2) \in  \Hom(\underline{n},\underline{o})$, we define:
\begin{equation}
\label{eq:compositionbdelta}
(\sigma_2,f_2) \circ (\sigma_1,f_1):=  (\delta^{B}_{m,n}(\sigma_2,f_1) \circ \sigma_1, f_2 \circ \delta^{\Delta}_{m,n}(\sigma_2,f_1)).
\end{equation}
The monoidal product on objects is addition of natural numbers; on morphisms, for $(\sigma_1,f_1) \in \Hom(\underline{m},\underline{m'})$, $(\sigma_2,f_2) \in \Hom(\underline{n},\underline{n'})$ we define  
\begin{equation}
\label{eq:productbdelta}
(\sigma_1, f_1) \otimes (\sigma_2,f_2) := (\sigma_1 \times \sigma_2, f_1 \sqcup f_2) \in \Hom(\underline{m}+\underline{n}, \underline{m'}+\underline{n'}),
\end{equation}
where $\sqcup$ is the coproduct in ${\bf Set}$ and $\times$ is the Cartesian product of groups. The braiding $\sigma_{m,n}: \underline{m} \otimes \underline{n} \to \underline{n} \otimes \underline{m}$ is simply the corresponding braid $(\sigma, \id_{\underline{m+n}}) \in \Hom(\underline{m+n},\underline{m+n})$. 
\end{definition}

\begin{proposition}\label{probformonoidsprop}
The PROB on the monoid computad, ${\bf F^{br}M}$, is isomorphic to ${\bf B \Delta}$.
\end{proposition}

\begin{proof}
See {\cite[Section 4]{Day2003}}. The isomorphism is the identity function on 0-cells; on 1-cells, it is simply $(\sigma, f) \mapsto \tilde{f} \circ \sigma$, where $\tilde{f}$ is defined as in~\eqref{eq:monotonefunctionimage}. The diagrammatic representation is shown in Figure~\ref{monoidisomorphismfunctor}.
\begin{figure}
\centering
\raisebox{1.5cm}{{\huge $(\sigma,f) \;\mapsto\;\;$}}
\begin{tikzpicture}
	\begin{pgfonlayer}{nodelayer}
		\node [style=none] (0) at (0, -0) {};
		\node [style=none] (1) at (0, 0.5) {};
		\node [style=none] (2) at (2, -0) {};
		\node [style=none] (3) at (2, 0.5) {};
		\node [style=none] (4) at (0, 1.5) {};
		\node [style=none] (5) at (2, 1.5) {};
		\node [style=none] (6) at (1, 1) {$\sigma$};
		\node [style=none] (7) at (0, 1.5) {};
		\node [style=none] (8) at (0, 1.75) {};
		\node [style=none] (9) at (2, 1.75) {};
		\node [style=none] (10) at (0, 2) {};
		\node [style=none] (11) at (2, 2) {};
		\node [style=none] (12) at (0, 3) {};
		\node [style=none] (13) at (2, 3) {};
		\node [style=none] (14) at (2, 3.5) {};
		\node [style=none] (15) at (0, 3.5) {};
		\node [style=none] (16) at (1, 3.5) {$\dots$};
		\node [style=none] (17) at (1, 1.75) {$\dots$};
		\node [style=none] (18) at (1, 2.5) {$\tilde{f}$};
		\node [style=none] (19) at (1, -0) {$\dots$};
		\node [style=none] (20) at (0.25, 3.5) {};
		\node [style=none] (21) at (1.75, 3.5) {};
		\node [style=none] (22) at (0.25, -0) {};
		\node [style=none] (23) at (1.75, -0) {};
		\node [style=none] (24) at (0.25, 0.5) {};
		\node [style=none] (25) at (1.75, 0.5) {};
		\node [style=none] (26) at (0.25, 3) {};
		\node [style=none] (27) at (1.75, 3) {};
		\node [style=none] (28) at (0.25, 2) {};
		\node [style=none] (29) at (0.25, 1.5) {};
		\node [style=none] (30) at (1.75, 2) {};
		\node [style=none] (31) at (1.75, 1.5) {};
	\end{pgfonlayer}
	\begin{pgfonlayer}{edgelayer}
		\draw [style=simple] (15.center) to (12.center);
		\draw [style=simple] (12.center) to (13.center);
		\draw [style=simple] (14.center) to (13.center);
		\draw [style=simple] (12.center) to (10.center);
		\draw [style=simple] (10.center) to (11.center);
		\draw [style=simple] (11.center) to (13.center);
		\draw [style=simple] (10.center) to (8.center);
		\draw [style=simple] (11.center) to (9.center);
		\draw [style=simple] (8.center) to (4.center);
		\draw [style=simple] (9.center) to (5.center);
		\draw [style=simple] (4.center) to (1.center);
		\draw [style=simple] (1.center) to (3.center);
		\draw [style=simple] (3.center) to (5.center);
		\draw [style=simple] (5.center) to (4.center);
		\draw [style=simple] (1.center) to (0.center);
		\draw [style=simple] (3.center) to (2.center);
		\draw [style=simple] (20.center) to (26.center);
		\draw [style=simple] (21.center) to (27.center);
		\draw [style=simple] (22.center) to (24.center);
		\draw [style=simple] (23.center) to (25.center);
		\draw [style=simple] (29.center) to (28.center);
		\draw [style=simple] (31.center) to (30.center);
	\end{pgfonlayer}
\end{tikzpicture}
\caption{The image of the pair of a braiding and a monotone function under the isomorphisms defined in this section.}
\label{monoidisomorphismfunctor}
\end{figure}
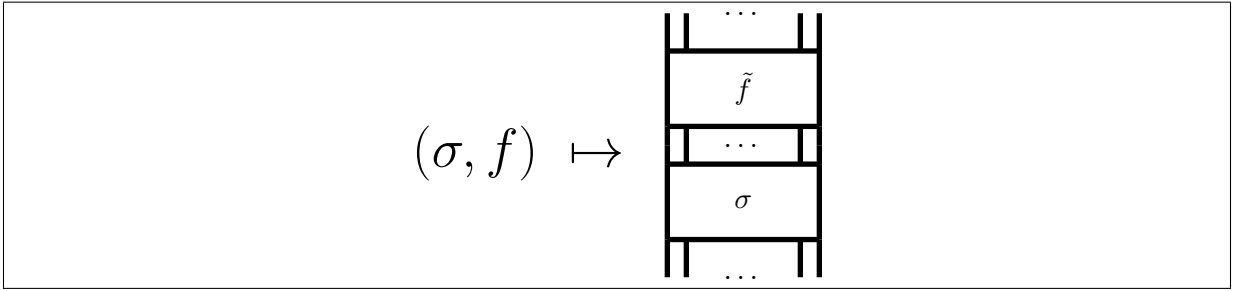
\end{proof}
\noindent
The PROP for monoids, ${\bf F^{sym}M}$, can be treated similarly. Let $S_n$ be the symmetric group on $n$ points. There is a surjective homomorphism $q: B_n \to S_n$, which takes a braid to its underlying permutation, and is suitably compatible with the distributive law. We therefore obtain another distributive law 
\begin{equation*}
S_n \times {\bf \Delta}(m,n) \overset{\delta^s_{m,n}}{\to} {\bf \Delta}(m,n) \times S_m
\end{equation*}
encoding the effect of pulling the monoid structure through the permutations.
\begin{definition}
The symmetric monoidal category ${\bf S \Delta}$ has natural numbers for objects. Morphisms $\underline{m} \to \underline{n}$ are pairs of a morphism in ${\bf \Delta}(m,n)$ and an element of $S_m$. Composition, monoidal product and braiding are defined as for ${\bf B \Delta}$~(\ref{eq:compositionbdelta}-\ref{eq:productbdelta}), using the distributive law $\delta^s$.
\end{definition}

\begin{proposition}
The PROP on the monoid computad, ${\bf F^{sym}M}$, is isomorphic to ${\bf S \Delta}$.
\end{proposition}
\begin{proof}
See~{\cite[Section 4]{Day2003}}. Again, the isomorphism is the identity function on 0-cells, and on 1-cells it is $(s, f) \mapsto \tilde{f} \circ s$, as in Figure \ref{monoidisomorphismfunctor}.
\end{proof}

We now turn to commutative monoids. One needs a braiding in order to define the commutativity equality, so it is meaningless to consider the PRO in this case.

We begin with the PROB ${\bf F^{br} CM}$. Since all the equations in the monoid computad are satisfied, ${\bf F^{br} CM}$ will be a quotient category of ${\bf F^{br} M}$. The quotient is defined as follows. Given some morphism $(\sigma,f)$, the commutativity axiom allows us to alter $\sigma$ by absorbing or emitting braidings from the trees of $f$. For example, Figure~\ref{braidemittanceexample} shows emission of the braiding $\sigma_2^{-1} \in B_3$ from a single tree with 3 inputs.

\begin{figure}
\centering
\begin{tikzpicture}
	\begin{pgfonlayer}{nodelayer}
		\node [style=rn] (0) at (0.5, 2) {};
		\node [style=rn] (1) at (1, 2.5) {};
		\node [style=none] (2) at (1, 3.25) {};
		\node [style=none] (3) at (0, -0) {};
		\node [style=none] (4) at (1, -0) {};
		\node [style=none] (5) at (2, -0) {};
		\node [style=rn] (6) at (5, 2.5) {};
		\node [style=none] (7) at (4, 1.5) {};
		\node [style=none] (8) at (5, -0) {};
		\node [style=none] (9) at (6, -0) {};
		\node [style=none] (10) at (4, -0) {};
		\node [style=none] (11) at (5, 3.25) {};
		\node [style=none] (12) at (4, 1.5) {};
		\node [style=none] (13) at (6, 1.5) {};
		\node [style=none] (14) at (3, 1.5) {$=$};
		\node [style=rn] (15) at (5.5, 2) {};
		\node [style=none] (16) at (2, 1.5) {};
		\node [style=none] (17) at (1, 1.5) {};
		\node [style=none] (18) at (0, 1.5) {};
		\node [style=none] (19) at (0, 1.5) {};
		\node [style=none] (20) at (5, 1.5) {};
		\node [style=none] (21) at (5.5, 1.25) {};
		\node [style=none] (22) at (5.75, 1) {};
		\node [style=none] (23) at (5.5, 0.25) {};
		\node [style=none] (24) at (5.75, 0.5) {};
		\node [style=none] (25) at (7.5, 1.5) {};
		\node [style=none] (26) at (8.5, 3.25) {};
		\node [style=none] (27) at (9, 0.75) {};
		\node [style=none] (28) at (8.75, -0) {};
		\node [style=rn] (29) at (9, 2) {};
		\node [style=rn] (30) at (8.5, 2.5) {};
		\node [style=none] (31) at (9.25, 1) {};
		\node [style=none] (32) at (7.5, -0) {};
		\node [style=none] (33) at (9.5, -0) {};
		\node [style=none] (34) at (7.5, 1.5) {};
		\node [style=none] (35) at (9, 2) {};
		\node [style=none] (36) at (9, 2) {};
		\node [style=none] (37) at (6.75, 1.5) {$=$};
	\end{pgfonlayer}
	\begin{pgfonlayer}{edgelayer}
		\draw [style=simple] (0) to (1);
		\draw [style=simple] (1) to (2.center);
		\draw [style=simple] (6) to (11.center);
		\draw [style=simple] (10.center) to (12.center);
		\draw [style=simple] (7.center) to (6);
		\draw [style=simple] (13.center) to (15);
		\draw [style=simple] (15) to (6);
		\draw [style=simple] (16.center) to (1);
		\draw [style=simple] (5.center) to (16.center);
		\draw [style=simple] (4.center) to (17.center);
		\draw [style=simple] (17.center) to (0);
		\draw [style=simple] (18.center) to (0);
		\draw [style=simple] (3.center) to (18.center);
		\draw [style=simple] (20.center) to (15);
		\draw [style=simple, bend left=60, looseness=2.00] (9.center) to (13.center);
		\draw [style=simple] (8.center) to (23.center);
		\draw [style=simple, bend right=45, looseness=1.00] (24.center) to (22.center);
		\draw [style=simple] (21.center) to (20.center);
		\draw [style=simple] (30) to (26.center);
		\draw [style=simple] (32.center) to (34.center);
		\draw [style=simple] (25.center) to (30);
		\draw [style=simple] (29) to (30);
		\draw [style=simple] (28.center) to (27.center);
		\draw [style=simple, bend left=15, looseness=1.00] (33.center) to (35.center);
		\draw [style=simple, bend right=45, looseness=1.00] (31.center) to (36.center);
	\end{pgfonlayer}
\end{tikzpicture}
\caption{Emission of the braiding $\sigma_2^{-1} \in B_3$ from a single tree with 3 inputs. The first equality uses associativity and the braided structure of the category, and the second uses commutativity.}
\label{braidemittanceexample}
\end{figure}
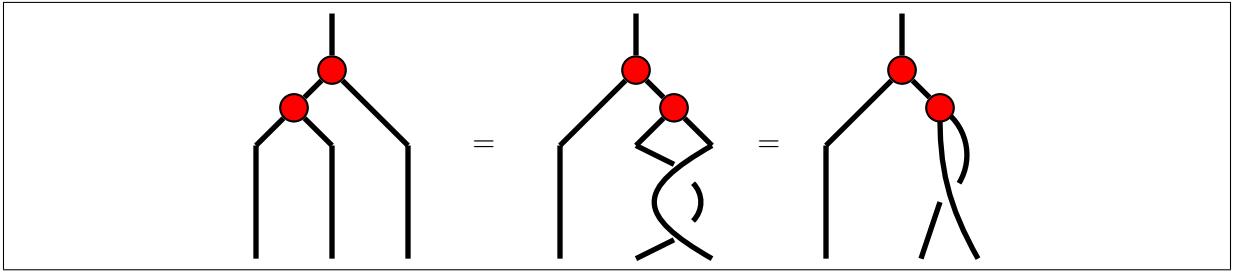

Each $f \in {\bf \Delta}(m,n)$ has fibres $f^{-1}(i)$, $i \in \{1,\dots,n\}$; we write $p_i = |f^{-1}(i)|$. Using the commutativity equality, we can create braidings or inverse braidings underneath the trees, move them downwards and absorb them into $\sigma$. Letting $p_i = |f^{-1}(i)|$, we obtain an action of $\prod_{i=1}^n B_{p_i}$ on $B_m$ by postcomposition. Figure \ref{braidemittanceandabsorption} depicts this for one morphism $\underline{7} \to \underline{3}$.

\begin{figure}
\centering
\begin{tikzpicture}
	\begin{pgfonlayer}{nodelayer}
		\node [style=none] (0) at (0, 4) {};
		\node [style=none] (1) at (1, 4) {};
		\node [style=none] (2) at (2, 4) {};
		\node [style=rn] (3) at (0.5, 4.5) {};
		\node [style=rn] (4) at (1, 5) {};
		\node [style=none] (5) at (1, 8.25) {};
		\node [style=none] (6) at (0, 4) {};
		\node [style=none] (7) at (0, 1) {};
		\node [style=none] (8) at (7, 4) {};
		\node [style=none] (9) at (7, 1) {};
		\node [style=none] (10) at (3.5, 1.5) {$\sigma \in B_7$};
		\node [style=gn] (11) at (3, 5.5) {};
		\node [style=none] (12) at (3, 8.25) {};
		\node [style=none] (13) at (4, 6) {};
		\node [style=none] (14) at (5, 6) {};
		\node [style=none] (15) at (6, 6) {};
		\node [style=none] (16) at (7, 6) {};
		\node [style=rn] (17) at (4.5, 6.5) {};
		\node [style=rn] (18) at (5, 7) {};
		\node [style=rn] (19) at (5.5, 7.5) {};
		\node [style=none] (20) at (5.5, 8.25) {};
		\node [style=none] (21) at (4, 4) {};
		\node [style=none] (22) at (5, 4) {};
		\node [style=none] (23) at (6, 4) {};
		\node [style=none] (24) at (0, 2) {};
		\node [style=none] (25) at (7, 2) {};
		\node [style=none] (26) at (1, 2) {};
		\node [style=none] (27) at (2, 2) {};
		\node [style=none] (28) at (4, 2) {};
		\node [style=none] (29) at (5, 2) {};
		\node [style=none] (30) at (6, 2) {};
		\node [style=none] (31) at (0, -0) {};
		\node [style=none] (32) at (1, -0) {};
		\node [style=none] (33) at (2, -0) {};
		\node [style=none] (34) at (4, -0) {};
		\node [style=none] (35) at (5, -0) {};
		\node [style=none] (36) at (6, -0) {};
		\node [style=none] (37) at (7, -0) {};
		\node [style=none] (38) at (1, 1) {};
		\node [style=none] (39) at (2, 1) {};
		\node [style=none] (40) at (4, 1) {};
		\node [style=none] (41) at (5, 1) {};
		\node [style=none] (42) at (6, 1) {};
		\node [style=none] (43) at (0, 3.5) {};
		\node [style=none] (44) at (2, 3.5) {};
		\node [style=none] (45) at (2, 2.5) {};
		\node [style=none] (46) at (0, 2.5) {};
		\node [style=none] (47) at (1, 3) {$B_3$};
		\node [style=none] (48) at (4, 3.5) {};
		\node [style=none] (49) at (4, 2.5) {};
		\node [style=none] (50) at (7, 2.5) {};
		\node [style=none] (51) at (5.5, 3) {$B_4$};
		\node [style=none] (52) at (7, 3.5) {};
		\node [style=none] (53) at (1, 2.5) {};
		\node [style=none] (54) at (5, 2.5) {};
		\node [style=none] (55) at (6, 2.5) {};
		\node [style=none] (56) at (1, 3.5) {};
		\node [style=none] (57) at (5, 3.5) {};
		\node [style=none] (58) at (6, 3.5) {};
		\node [style=none] (59) at (6.5, 3) {$\Updownarrow$};
		\node [style=none] (61) at (1.75, 3) {$\Updownarrow$};
	\end{pgfonlayer}
	\begin{pgfonlayer}{edgelayer}
		\draw [style=simple] (0.center) to (3);
		\draw [style=simple] (1.center) to (3);
		\draw [style=simple] (3) to (4);
		\draw [style=simple] (2.center) to (4);
		\draw [style=simple] (4) to (5.center);
		\draw [style=simple] (11) to (12.center);
		\draw [style=simple] (13.center) to (17);
		\draw [style=simple] (14.center) to (17);
		\draw [style=simple] (17) to (18);
		\draw [style=simple] (15.center) to (18);
		\draw [style=simple] (16.center) to (19);
		\draw [style=simple] (18) to (19);
		\draw [style=simple] (19) to (20.center);
		\draw [style=simple] (21.center) to (13.center);
		\draw [style=simple] (14.center) to (22.center);
		\draw [style=simple] (15.center) to (23.center);
		\draw [style=simple] (8.center) to (16.center);
		\draw [style=simple] (24.center) to (25.center);
		\draw [style=simple] (9.center) to (7.center);
		\draw [style=simple] (24.center) to (7.center);
		\draw [style=simple] (25.center) to (9.center);
		\draw [style=simple] (7.center) to (31.center);
		\draw [style=simple] (38.center) to (32.center);
		\draw [style=simple] (39.center) to (33.center);
		\draw [style=simple] (40.center) to (34.center);
		\draw [style=simple] (41.center) to (35.center);
		\draw [style=simple] (42.center) to (36.center);
		\draw [style=simple] (9.center) to (37.center);
		\draw [style=simple] (43.center) to (44.center);
		\draw [style=simple] (45.center) to (44.center);
		\draw [style=simple] (43.center) to (46.center);
		\draw [style=simple] (46.center) to (45.center);
		\draw [style=simple] (48.center) to (52.center);
		\draw [style=simple] (50.center) to (52.center);
		\draw [style=simple] (48.center) to (49.center);
		\draw [style=simple] (49.center) to (50.center);
		\draw [style=simple] (46.center) to (24.center);
		\draw [style=simple] (53.center) to (26.center);
		\draw [style=simple] (45.center) to (27.center);
		\draw [style=simple] (49.center) to (28.center);
		\draw [style=simple] (54.center) to (29.center);
		\draw [style=simple] (30.center) to (55.center);
		\draw [style=simple] (25.center) to (50.center);
		\draw [style=simple] (0.center) to (43.center);
		\draw [style=simple] (56.center) to (1.center);
		\draw [style=simple] (44.center) to (2.center);
		\draw [style=simple] (48.center) to (21.center);
		\draw [style=simple] (57.center) to (22.center);
		\draw [style=simple] (58.center) to (23.center);
		\draw [style=simple] (8.center) to (52.center);
	\end{pgfonlayer}
\end{tikzpicture}
\caption{}
\label{braidemittanceandabsorption}
\end{figure}
Let $\sim_f$ be the equivalence relation which identifies elements of $B_n$ if they are in the same orbit under this action. The distributive law is suitably compatible with the action, which allows us to define the following category.

\begin{definition}
The category ${\bf B\Delta/\sim}$ is defined in the same way as ${\bf B\Delta}$, but where the morphisms $f:\underline{m} \to \underline{n}$ are now pairs $(\bar{\sigma},f)$, where $f \in \Delta(m,n)$ and $\bar{\sigma} \in B_m/\sim_f$. 
\end{definition}

\begin{proposition}\label{probforcommmonprop}
The PROB on the commutative monoid computad is isomorphic to ${\bf B\Delta/\sim}$.
\end{proposition}
\begin{proof}
See~\cite[Theorem 2]{Lavers1997}. Again, the isomorphism takes $(\bar{\sigma},f) \mapsto \tilde{f} \circ \bar{\sigma}$, as in Figure \ref{monoidisomorphismfunctor}.
\end{proof}
\noindent
Finally, we consider the PROP ${\bf F^{sym} CM}$. Again, this will be a quotient of ${\bf F^{sym} M}$. Rather than braidings, we now emit permutations from the trees, giving rise to an action of $\prod_{i=1}^n S_{p_i}$ on $S_m$ by postcomposition, which induces a quotient $S_m/\sim_f$. As before, we define the following category.

\begin{definition}
The category ${\bf S\Delta/\sim}$ is defined in the same way as ${\bf S\Delta}$, but where the morphisms $f:\underline{m} \to \underline{n}$ are now pairs $(\bar{\sigma},f)$ of $f \in \Delta(m,n)$ and $\bar{\sigma} \in  S_m/\sim_f$. 
\end{definition}
\noindent
It turns out that ${\bf S\Delta/\sim}$ is isomorphic to a familiar category.

\begin{definition}
The category ${\bf FS}$ has objects natural numbers, and morphisms $\underline{m} \to \underline{n}$ functions $\{1,\dots,m\} \to \{1,\dots,n\}$, where $\underline{0}$ is the empty set.
\end{definition}

\begin{proposition}\label{FSiscommonoidspropprop}
The PROP ${\bf F^{sym}CM}$ on the commutative monoid computad, the category ${\bf S\Delta/\sim}$, and ${\bf FS}$ are all isomorphic. 
\end{proposition}

\begin{proof}\cite[Proposition 2.3.3]{Davydov2010}. The isomorphism between ${\bf F^{sym}CM}$ and ${\bf S\Delta/\sim}$ is as in Proposition \ref{probforcommmonprop}. For clarity, we define explicitly the isomorphism ${\bf FS} \to {\bf F^{sym}CM}$. Again, the morphisms in the image are of the form $\tilde{f} \circ \bar{\sigma}$ shown in Figure \ref{monoidisomorphismfunctor}; we need only define $\bar{\sigma}_f \in S_n/\sim_f$ and $\tilde{f} \in {\bf \Delta}(m,n)$ for a given function $f: \underline{m} \to \underline{n}$.

We define $\tilde{f}$ as before. Again, Let $\mu: \underline{2} \to \underline{1}$ be the multiplication in ${\bf F^{br}CM}$, and $u: \underline{0} \to \underline{1}$ be the unit. Again, let $p_i = |f^{-1}(i)|$ be the cardinality of the preimage of $i \in \{1,\dots,n\}$. Let $\mu^{n}: \underline{n} \to \underline{1}$ be the composition of $n-1$ multiplications, left bracketed; for example, $\mu^{4}=\mu \circ (\mu \otimes \text{Id}) \circ (\mu \otimes \text{Id} \otimes \text{Id})$. Set $\mu^1 = \text{Id}$ and $\mu^0 = u$. Then
$$\tilde{f} := (\mu^{p_1} \otimes \cdots \otimes \mu^{p_n}).$$
Now we define $\bar{\sigma}_f$. We write a permutation of $m$ elements as rearrangement of $1,\dots,m$; for example, the cycle $(1)(23)$ may be written as $\langle 132 \rangle$. For $S \subset \mathbb{N}$, let $[S]$ be the set $S$ with elements written in ascending order. In this notation, we define 
$$\sigma_f = \langle[f^{-1}(1)][f^{-1}(2)]\dots[f^{-1}(n)]\rangle.$$
Then $\bar{\sigma}_f$ is the equivalence class of this permutation under the quotient.
\end{proof}
\noindent
The results of this section are summarised in Table~ \ref{tbl:monoidsresultsinintro}.

\section{Coherence for braided and symmetric pseudomonoids}
\label{sec:coherence}
In this section we state and prove our main results, which were summarised in Table~\ref{tbl:pseudomonoidsintro}.

\subsection{Combinatorial bicategories}\label{sec:combibicats}
First we define the combinatorial bicategories that appear in the table.
\begin{definition}
A \emph{locally discrete} bicategory is one with only identity 2-cells. 
\end{definition}
\noindent
Any category may be considered as a strict locally discrete bicategory by adding identity 2-cells; this preserves monoidality, braiding and symmetry. We therefore obtain the  Gray monoid ${\bf \Delta}$, the braided Gray monoids ${\bf B\Delta}$ and ${\bf B\Delta/\sim}$, and the symmetric Gray monoids ${\bf S\Delta}$ and ${\bf FS}$. 

We now define another combinatorial categorification of ${\bf FS}$, which will correspond to the theory of braided pseudomonoids in a symmetric monoidal bicategory. In this case, not all diagrams commute, so we need to add more 2-cells to ${\bf FS}$ than just identities.

\begin{definition}
A \emph{locally totally disconnected} 2-category is one whose 2-cells are all endomorphisms.
\end{definition}

\begin{definition}\label{def:fsbr}
${\bf FS}^{\text{br}}$ is a locally totally disconnected strict 2-category obtained from ${\bf FS}$ by adding 2-cells as follows. 

For any 1-cell $f: \underline{m} \to \underline{n}$, we add a set of 2-cells $\Hom(f,f) = \prod_{i=1}^n PB_{p_i}$, where $p_i = |f(i)^{-1}|$ and $PB_n$ is the pure braid group on $n$ points. We define the following compositional structure on these 2-cells. 

\begin{itemize}[leftmargin=*]
\item \emph{Horizontal composition.} For any $f: \underline{m} \to \underline{n}$, and $\sigma, \tau \in \Hom(f,f)$, we define $\tau \circ_{H} \sigma \in \Hom(f,f)$ to be the composition $\tau \circ \sigma \in \prod_{i=1}^n PB_{p_i}$. 
\item \emph{Vertical composition.} For any $f: \underline{m} \to \underline{n}$, $g: \underline{n} \to \underline{o}$, and $\sigma_m \in \Hom(f,f)$, $\sigma_n \in \Hom(g,g)$, we define $\sigma_n \circ_V \sigma_m\in \Hom(g \circ f, g \circ f)$ as follows: $$\sigma_n \circ_V \sigma_m :=  \sigma_m  \delta_{m,n}^{B}(f_{\Delta},\sigma_n).$$
Here $\delta^{B}_{m,n}$ is the braided part of the image of the distributive law~\eqref{eq:distlawbraid} and $f_{\Delta}$ is the ${\bf \Delta}$-factor of $f$ in the decomposition ${\bf FS} = {\bf S \Delta}$ of Proposition \ref{FSiscommonoidspropprop}.
\item \emph{Monoidal product.} For any $f: \underline{m} \to \underline{n}, g: \underline{o} \to \underline{p}$, and $\sigma_m \in \Hom(f,f)$, $\sigma_n \in \Hom(g,g)$, we define $\sigma_m \otimes \sigma_n: f \otimes g \to f \otimes g$ to be the Cartesian product of braids $(\sigma_m, \sigma_n) \in B_m \times B_n \subset B_{m+n}$. 
\end{itemize}
\noindent
It is straightforward to check that this defines a strict symmetric Gray monoid.
\end{definition}

\subsection{Biequivalences}\label{sec:biequivalences}
\begin{definition}
A \emph{(braided/symmetric) biequivalence} of Gray monoids is a homomorphism of (braided/symmetric) Gray monoids \cite{Day1997} which is essentially surjective on objects and 1-morphisms, and fully faithful on 2-morphisms. 
\end{definition}

\noindent
We define the biequivalences of Table~\ref{tbl:pseudomonoidsintro} by categorifying the isomorphisms of Section \ref{sec:theoryofmonoids}. Those isomorphisms map $(\sigma,f)$ to $\tilde{f} \circ \tilde{\sigma}$, where:
\begin{itemize}
\item $\sigma$ is an element of the braid group, the symmetric group or a quotient of those groups.
\item $\tilde{\sigma}$ is the same element considered as a morphism in the free braided or symmetric monoidal category on a single object.
\item $f$ is a monotone function.
\item $\tilde{f}$ is the corresponding braid-free morphism in the PRO, PROB or PROP.
\end{itemize}
\noindent
This definition needs to be adapted slightly for the higher setting. Firstly, in order to specify $\tilde{f}$, one must give the height of each generating 1-cell, since planar isotopy is now an isomorphism rather than an equality. Secondly, in order to specify $\tilde{\sigma}$, one must now specify a word in the generators of the Artin presentation of the braid group, rather than simply the isotopy class of braids, the permutation, or the equivalence class under the quotient, since the braid relations are now also isomorphisms. To resolve these issues, we make the following definitions. We first consider the braid-free part.
\begin{definition}
Let ${\bf \Delta}$ be the subcategory of braid-free morphisms of any of the PROs, PROBs or PROPs for monoids or commutative monoids. We say that a diagram in ${\bf \Delta}$ is in \emph{standard form} if the heights of the 1-cells in the trees rise from left to right, as in Figures \ref{deltatreesexample} and \ref{braidemittanceandabsorption}.
\end{definition}
\noindent
We now consider the braid part.
\begin{definition}
We define the \emph{representative word} of a braid or permutation as follows. 
\begin{itemize}[leftmargin=*]
\item For $\sigma \in B_n$: Its representative word is its Artin normal form.
\item For $\overline{\sigma} \in B_n/\sim_f:$  We use the Axiom of Choice to pick a coset representative for each $\overline{\sigma} \in B_n/\sim$, and say that the Artin normal form of this chosen representative is the representative word of $\overline{\sigma}$. 
\item For $s \in S_n$: We write the permutation as a braid diagram in the following way. We draw a straight line from each input to the output which is its image under the permutation. At crossings, we use the convention that the string connected to the leftmost input crosses on top. If there are any triple crossings then we pull the top string downwards in order to remove them. If two crossings occur at the same height, we deform the diagram so that the leftmost crossing occurs first. We may then read off a word in the braid generators from the diagram; we say that this is the representative word of $s$.
\item For $\overline{s} \in S_n/\sim_f:$ We use the Axiom of Choice to pick a coset representative for each $\overline{s} \in S_n/\sim_f$, and obtain a representative word for this representative as for $s \in S_n$.
\end{itemize}
\end{definition}
\noindent
Using these definitions, it is straightforward to specify the biequivalences on 0- and 1-cells. We define them as maps from the combinatorial category to the higher PRO.
\begin{definition}[Biequivalences on 0- and 1-cells]
Each biequivalence in Table~\ref{tbl:pseudomonoidsintro} is defined on 0-cells  and 1-cells as follows. 

\begin{itemize}
\item On 0-cells the map takes $\underline{n} \in \mathbb{N}$ to $A^{\otimes n}$, where $A$ is the unique generating 0-cell of the higher PRO.

\item On 1-cells, $(\sigma, f)$ is mapped to $\tilde{f} \circ \tilde{\sigma}$, where $\tilde{f}$ is in standard form and $\tilde{\sigma}$ is the braid defined by the representative word of $\sigma$.
\end{itemize}
\end{definition}

\begin{definition}[Biequivalences on 2-cells for locally discrete bicategories]
For the locally discrete categorifications, the biequivalences in Table~\ref{tbl:pseudomonoidsintro} are defined on 2-cells by taking identity 2-cells to identity 2-cells.
\end{definition}
\noindent
Finally, we define the biequivalence on 2-cells for the only non-locally discrete categorification.

\begin{definition}\label{braidedbiequivalencedefinition}
The biequivalence from ${\bf FS^{\text{br}}}$ in Table~\ref{tbl:pseudomonoidsintro} is defined on 2-cells as follows. Let $\sigma \in \Hom_{{\bf FS}^{\text{br}}}(f,f) = \prod_{i=1}^n PB_{p_i}$, and let $\sigma_i$ be the factors of $\sigma$ in the product. The 2-cell in the image of $\sigma$ is a movie defined as follows.
\begin{enumerate}
\item We use the symmetric structure of the category to create the braid $\sigma_1$ directly beneath the tree $m^{p_1}$, working from left to right. By Theorem \ref{gurskiosornocoherencethm} there is no ambiguity regarding the 2-morphism we use to do this.
\item We remove the braid using commutators and inverse commutators as follows. Recall that the tree is initially left bracketed.
	\begin{enumerate} 
	\item Use associators to bring the lowest 1-cell in the tree $m^{p_i}$ directly above the highest braid, using the following iterative method:
Associate the bottom 1-cell to the right. If this is impossible associate the 1-cell above it to the right and return. If this is impossible associate the 1-cell above that to the right and return. Etc. Repeat until the lowest multiplication is directly above the braiding to be absorbed. 
	\item Use a commutator or an inverse commutator to remove the braiding. A commutator removes a positive braiding; an inverse commutator removes a negative braiding by producing a positive braiding, then cancelling the two. 
	\item Use the inverse of the original sequence of associators to left bracket the tree again. 
	\item Repeat until $\sigma_1$ has been entirely removed by commutators.
	\end{enumerate}
An example is shown in Figure \ref{braidabsexamplefig}.
\item Now create $\sigma_2$ beneath the tree $m^{p_2}$ and absorb. Repeat for all trees, working from left to right. This completes the loop.
\end{enumerate}
A schematic is shown in Figure \ref{imageoffbrschematic}.
\end{definition}
\begin{figure}[p]
\resizebox{!}{3cm}{
\begin{tabular}{c c c c c c c c c c c c}
& \includegraphics[width=3cm,height=4.5cm]{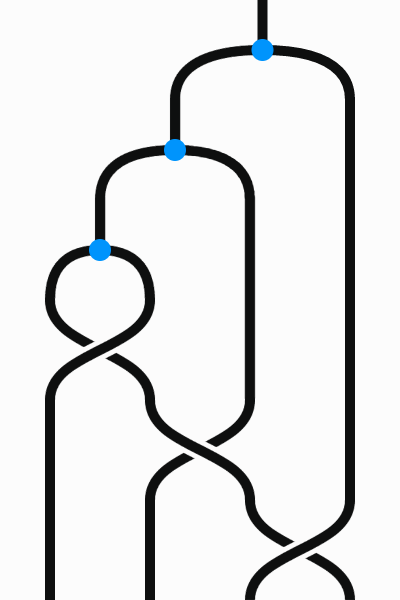} & \raisebox{1.5cm}{\makecell{$\Rightarrow$}} & \includegraphics[width=3cm,height=4.5cm]{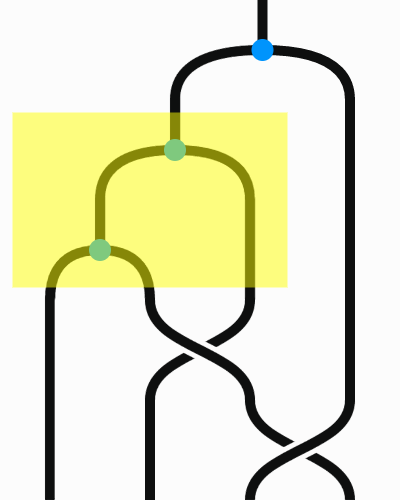} & \raisebox{1.5cm}{\makecell{$\Rightarrow$}} & \includegraphics[width=3cm,height=4.5cm]{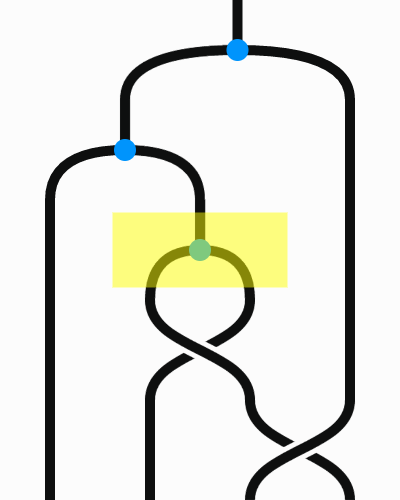} &
\raisebox{1.5cm}{\makecell{$\Rightarrow$}} &\includegraphics[width=3cm,height=4.5cm]{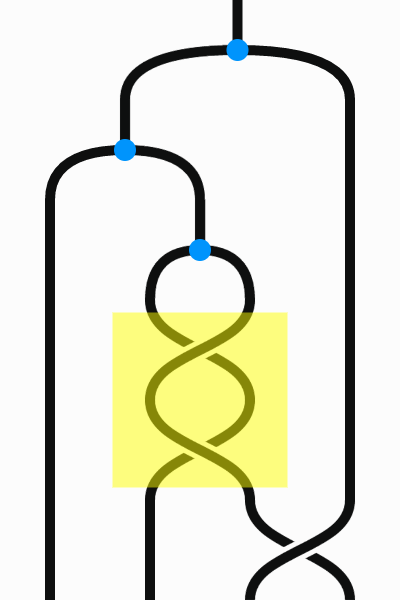} & \raisebox{1.5cm}{\makecell{$\Rightarrow$}} & \includegraphics[width=3cm,height=4.5cm]{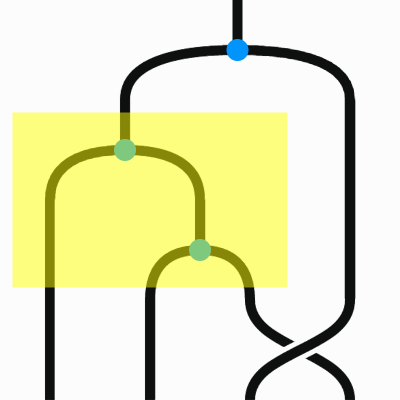} & \raisebox{1.5cm}{\makecell{$\Rightarrow$}} & \includegraphics[width=3cm,height=4.5cm]{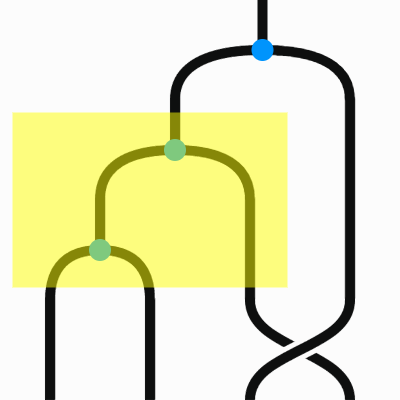} \\
\raisebox{1.5cm}{\makecell{$\Rightarrow$}} &\includegraphics[width=3cm,height=4.5cm]{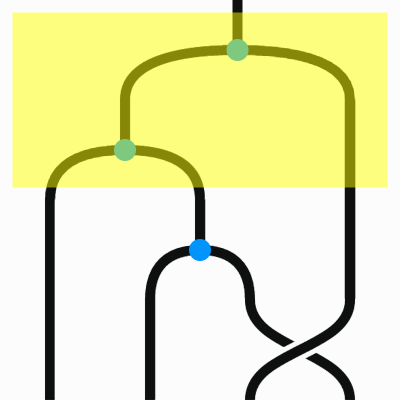} & \raisebox{1.5cm}{\makecell{$\Rightarrow$}} & \includegraphics[width=3cm,height=4.5cm]{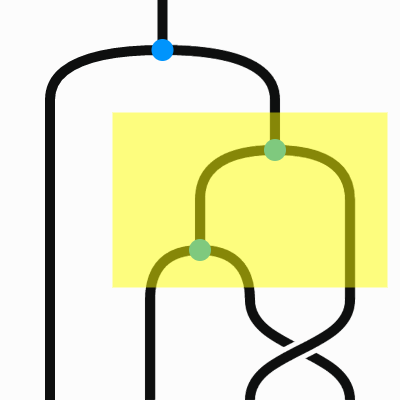} & \raisebox{1.5cm}{\makecell{$\Rightarrow$}} & \includegraphics[width=3cm,height=4.5cm]{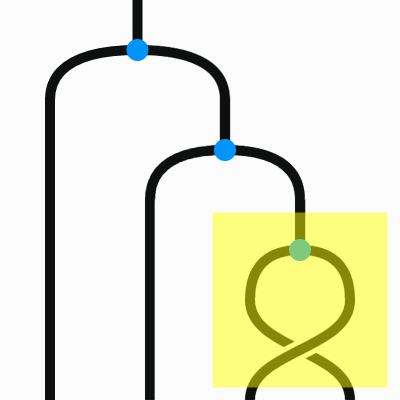} &
\raisebox{1.5cm}{\makecell{$\Rightarrow$}} &\includegraphics[width=3cm,height=4.5cm]{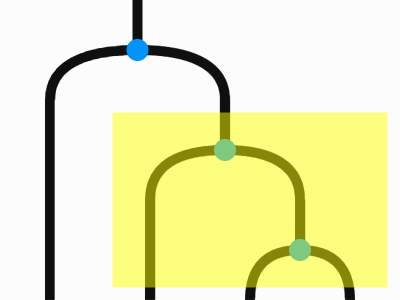} & \raisebox{1.5cm}{\makecell{$\Rightarrow$}} & \includegraphics[width=3cm,height=4.5cm]{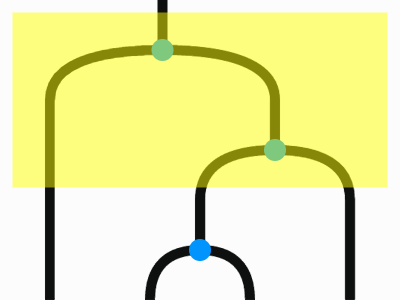} & \raisebox{1.5cm}{\makecell{$\Rightarrow$}} & \includegraphics[width=3cm,height=4.5cm]{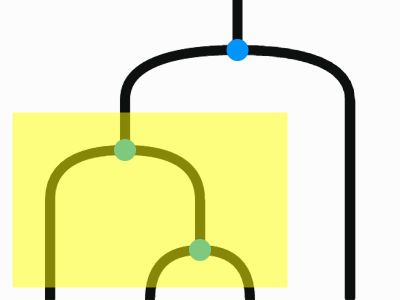}
\end{tabular}
}
\caption{Absorption of the braid $\sigma_1 \sigma_2^{-1} \sigma_3$.}
\label{braidabsexamplefig}
\end{figure}
\begin{figure}
\centering
\begin{tikzpicture}[scale=0.7]
	\begin{pgfonlayer}{nodelayer}
		\node [style=none] (0) at (0, -6.25) {};
		\node [style=none] (1) at (13.5, -6.25) {};
		\node [style=none] (2) at (13.5, -5.25) {};
		\node [style=none] (3) at (0, -5.25) {};
		\node [style=none] (4) at (0, -6.5) {};
		\node [style=none] (5) at (6.75, -6.5) {{\Huge $\dots$}};
		\node [style=none] (6) at (12.5, -6.5) {};
		\node [style=none] (7) at (13.5, -6.5) {};
		\node [style=none] (8) at (12.5, -6.25) {};
		\node [style=none] (9) at (6.75, -5.75) {Braid of 1-morphism};
		\node [style=none] (10) at (0, -4.75) {};
		\node [style=none] (11) at (0, -3.25) {};
		\node [style=none] (12) at (3.5, -3.25) {};
		\node [style=none] (13) at (3.5, -4.75) {};
		\node [style=none] (14) at (3.5, -5.25) {};
		\node [style=none] (15) at (1.5, -5) {{\Large $\dots$}};
		\node [style=none] (16) at (1.5, -4) {\makecell{$\sigma_1$}};
		\node [style=none] (17) at (13.5, -5.25) {};
		\node [style=none] (18) at (10.5, -5.25) {};
		\node [style=none] (19) at (0.5, -3.25) {};
		\node [style=none] (20) at (2.5, -3.25) {};
		\node [style=none] (21) at (1.25, -2.25) {};
		\node [style=none] (22) at (1.75, -2.25) {};
		\node [style=none] (23) at (1.75, -1.75) {};
		\node [style=none] (24) at (1.25, -1.75) {};
		\node [style=none] (25) at (1, -1.75) {};
		\node [style=none] (26) at (2, -1.75) {};
		\node [style=none] (27) at (1.5, -1.25) {};
		\node [style=none] (28) at (6, 2) {{\Huge $\iddots$}};
		\node [style=none] (29) at (1.5, -3) {{\Large $\dots$}};
		\node [style=none] (30) at (0, -1) {};
		\node [style=gn] (31) at (0.25, -0.75) {};
		\node [style=gn] (32) at (1.25, 0.25) {};
		\node [style=gn] (33) at (1.75, 0.75) {};
		\node [style=none] (34) at (1.75, 9) {};
		\node [style=none] (35) at (0.75, -0.25) {$\iddots$};
		\node [style=none] (36) at (3.5, -1) {};
		\node [style=none] (37) at (2.5, -1) {};
		\node [style=none] (38) at (-1.5, -2) {Step 2.};
		\node [style=none] (39) at (-1.5, -4) {Step 1.};
		\node [style=none] (40) at (8.5, 5.75) {Step $2n$.};
		\node [style=none] (41) at (8.5, 3.75) {\makecell{Step \\$2n-1$.}};
		\node [style=none] (42) at (11.5, -1.25) {{\Large $\dots$}};
		\node [style=none] (43) at (1, 0) {};
		\node [style=none] (44) at (12.5, -5.25) {};
		\node [style=none] (45) at (10, -5.25) {};
		\node [style=none] (46) at (0.5, -5.25) {};
		\node [style=none] (47) at (0.5, -4.75) {};
		\node [style=none] (48) at (2.5, -5.25) {};
		\node [style=none] (49) at (2.5, -4.75) {};
		\node [style=none] (50) at (0.5, -6.25) {};
		\node [style=none] (51) at (0.5, -6.5) {};
		\node [style=none] (52) at (10, -6.25) {};
		\node [style=none] (53) at (10, -6.5) {};
		\node [style=none] (54) at (6.75, -5) {{\Huge $\dots$}};
		\node [style=none] (55) at (0.5, -1) {};
		\node [style=none] (56) at (2.5, -6.25) {};
		\node [style=none] (57) at (2.5, -6.5) {};
		\node [style=none] (58) at (3.5, -6.25) {};
		\node [style=none] (59) at (3.5, -6.5) {};
		\node [style=none] (60) at (0.5, -0.5) {};
		\node [style=none] (61) at (11.5, 4.75) {{\Large $\dots$}};
		\node [style=none] (62) at (11.75, 6) {};
		\node [style=none] (63) at (13.5, 3) {};
		\node [style=gn] (64) at (11.75, 8.5) {};
		\node [style=gn] (65) at (10.25, 7) {};
		\node [style=none] (66) at (12.5, 3) {};
		\node [style=none] (67) at (11.5, 6.5) {};
		\node [style=none] (68) at (12.5, 4.5) {};
		\node [style=none] (69) at (11.75, 9.25) {};
		\node [style=none] (70) at (10.5, 3) {};
		\node [style=none] (71) at (12, 6) {};
		\node [style=none] (72) at (12.5, 6.75) {};
		\node [style=none] (73) at (10.5, 4.5) {};
		\node [style=none] (74) at (11, 7.75) {};
		\node [style=none] (75) at (11.5, 3.75) {\makecell{$\sigma_n$}};
		\node [style=none] (76) at (11, 6) {};
		\node [style=none] (77) at (10.75, 7.5) {$\iddots$};
		\node [style=gn] (78) at (11.25, 8) {};
		\node [style=none] (79) at (10.5, 7.25) {};
		\node [style=none] (80) at (10, 3) {};
		\node [style=none] (81) at (10, 6.75) {};
		\node [style=none] (82) at (11.25, 5.5) {};
		\node [style=none] (83) at (10, 4.5) {};
		\node [style=none] (84) at (11.75, 5.5) {};
		\node [style=none] (85) at (13.5, 4.5) {};
		\node [style=none] (86) at (11.25, 6) {};
		\node [style=none] (87) at (13.5, 6.75) {};
		\node [style=none] (88) at (10.5, 6.75) {};
		\node [style=none] (89) at (10.5, -6.5) {};
		\node [style=none] (90) at (10.5, -6.25) {};
		\node [style=none] (91) at (9.75, 5.75) {};
		\node [style=none] (92) at (11, 5.75) {};
		\node [style=none] (93) at (9.75, 3.75) {};
		\node [style=none] (94) at (11, 3.75) {};
		\node [style=none] (95) at (0.75, -4) {};
		\node [style=none] (96) at (-0.5, -4) {};
		\node [style=none] (97) at (0.75, -2) {};
		\node [style=none] (98) at (-0.5, -2) {};
	\end{pgfonlayer}
	\begin{pgfonlayer}{edgelayer}
		\draw [style=simple] (3.center) to (2.center);
		\draw [style=simple] (2.center) to (1.center);
		\draw [style=simple] (1.center) to (0.center);
		\draw [style=simple] (0.center) to (3.center);
		\draw [style=simple] (0.center) to (4.center);
		\draw [style=simple] (6.center) to (8.center);
		\draw [style=simple] (1.center) to (7.center);
		\draw [style=simple] (14.center) to (13.center);
		\draw [style=simple] (3.center) to (10.center);
		\draw [style=simple] (10.center) to (13.center);
		\draw [style=simple] (13.center) to (12.center);
		\draw [style=simple] (12.center) to (11.center);
		\draw [style=simple] (11.center) to (10.center);
		\draw [style=simple] (21.center) to (24.center);
		\draw [style=simple] (24.center) to (25.center);
		\draw [style=simple] (26.center) to (23.center);
		\draw [style=simple] (23.center) to (22.center);
		\draw [style=simple] (22.center) to (21.center);
		\draw [style=simple] (25.center) to (27.center);
		\draw [style=simple] (26.center) to (27.center);
		\draw [style=simple] (30.center) to (31);
		\draw [style=simple] (11.center) to (30.center);
		\draw [style=simple] (32) to (33);
		\draw [style=simple] (33) to (36.center);
		\draw [style=simple] (32) to (37.center);
		\draw [style=simple] (20.center) to (37.center);
		\draw [style=simple] (12.center) to (36.center);
		\draw [style=simple] (33) to (34.center);
		\draw [style=simple] (32) to (43.center);
		\draw [style=simple] (47.center) to (46.center);
		\draw [style=simple] (49.center) to (48.center);
		\draw [style=simple] (50.center) to (51.center);
		\draw [style=simple] (52.center) to (53.center);
		\draw [style=simple] (19.center) to (55.center);
		\draw [style=simple] (55.center) to (31);
		\draw [style=simple] (57.center) to (56.center);
		\draw [style=simple] (59.center) to (58.center);
		\draw [style=simple] (31) to (60.center);
		\draw [style=simple] (80.center) to (63.center);
		\draw [style=simple] (63.center) to (85.center);
		\draw [style=simple] (85.center) to (83.center);
		\draw [style=simple] (83.center) to (80.center);
		\draw [style=simple] (82.center) to (86.center);
		\draw [style=simple] (86.center) to (76.center);
		\draw [style=simple] (71.center) to (62.center);
		\draw [style=simple] (62.center) to (84.center);
		\draw [style=simple] (84.center) to (82.center);
		\draw [style=simple] (76.center) to (67.center);
		\draw [style=simple] (71.center) to (67.center);
		\draw [style=simple] (81.center) to (65);
		\draw [style=simple] (83.center) to (81.center);
		\draw [style=simple] (78) to (64);
		\draw [style=simple] (64) to (87.center);
		\draw [style=simple] (78) to (72.center);
		\draw [style=simple] (68.center) to (72.center);
		\draw [style=simple] (85.center) to (87.center);
		\draw [style=simple] (64) to (69.center);
		\draw [style=simple] (78) to (74.center);
		\draw [style=simple] (73.center) to (88.center);
		\draw [style=simple] (88.center) to (65);
		\draw [style=simple] (65) to (79.center);
		\draw [style=simple] (45.center) to (80.center);
		\draw [style=simple] (70.center) to (18.center);
		\draw [style=simple] (44.center) to (66.center);
		\draw [style=simple] (63.center) to (2.center);
		\draw [style=simple] (90.center) to (89.center);
		\draw [style=simple] (91.center) to (92.center);
		\draw [style=simple] (93.center) to (94.center);
		\draw [style=simple] (98.center) to (97.center);
		\draw [style=simple] (96.center) to (95.center);
	\end{pgfonlayer}
\end{tikzpicture}
\caption{The image of the 2-cell $\prod_{i=1}^n \sigma_i$ in  the map of Definition~\ref{braidedbiequivalencedefinition}. Steps 1 and 2 are creation and absorption of the braid $\sigma_1$ underneath the first tree; steps $2n-1$ and $2n$ are creation and absorption of the braid $\sigma_n$ under the last tree.}
\label{imageoffbrschematic}
\end{figure}
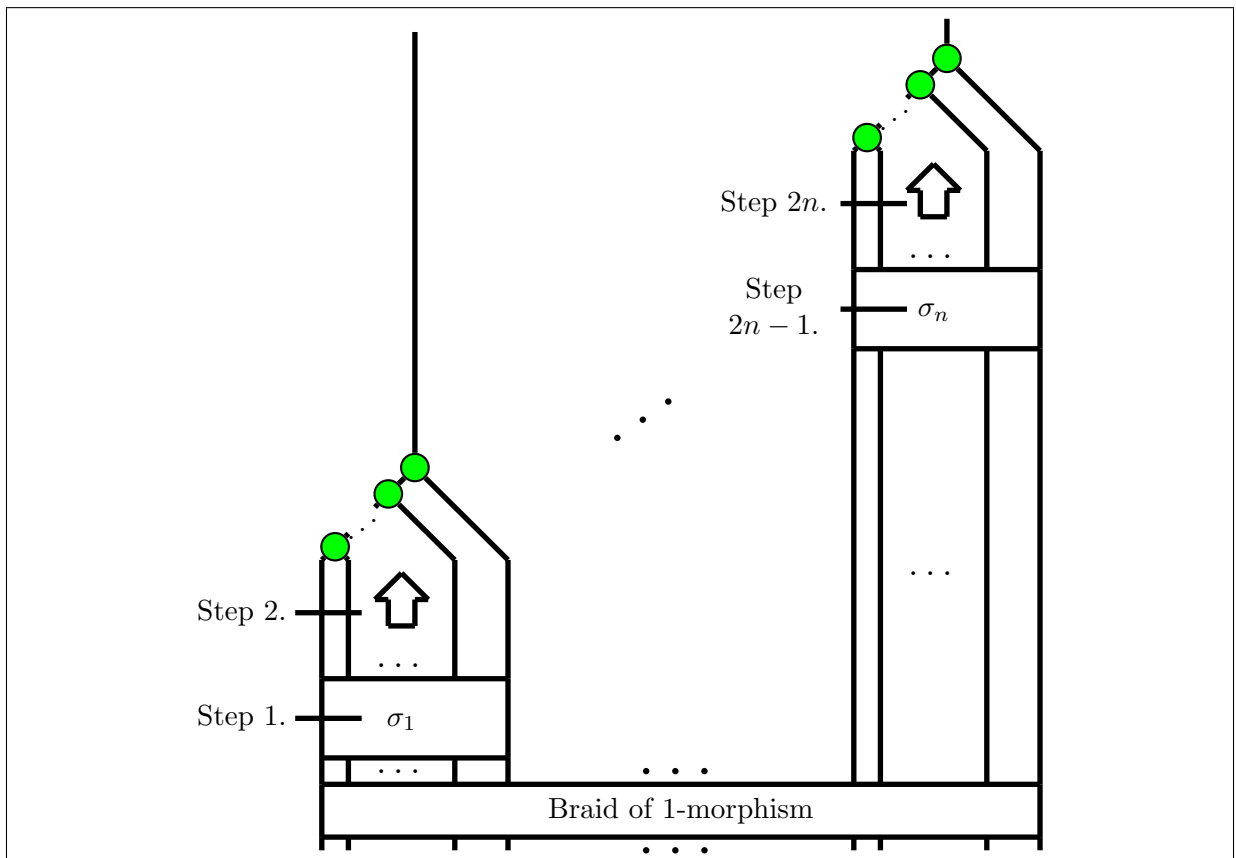
 
\subsection{Theorem statement}\label{sec:theoremstatement}

\begin{theorem}\label{thm:maintheorem}
The maps defined in Section~\ref{sec:biequivalences} are (braided/symmetric) monoidal biequivalences.
\end{theorem}
\noindent
Before commencing the proof of Theorem \ref{thm:maintheorem}, we consider how the biequivalence results for braided and symmetric pseudomonoids in symmetric monoidal bicategories imply MacLane's coherence theorems for braided and symmetric monoidal categories, which are braided and symmetric pseudomonoids in the symmetric monoidal bicategory {\bf Cat}.

\begin{corollary}[MacLane's coherence theorems]\label{corr:MacLane}
In a braided monoidal category, all diagrams of natural isomorphisms with the same underlying braid commute. In a symmetric monoidal category, all diagrams of natural isomorphisms with the same underlying permutation commute.
\end{corollary}

\begin{proof}[Proof (sketch)]
First, note that these coherence theorems only refer to 2-cell diagrams \emph{internal} to the category. From a 1-cell diagram in {\bf Cat}, the only 1-cell data which are preserved internally are: the bracketing of the tree, the permutation of the tree's input strings by the braid beneath the tree, and any units attached to the tree. From a 2-cell in {\bf Cat}, the only generating 2-cells which are preserved internally are the associator, unitors, and commutator.

The case of a single tree is sufficiently general. The \emph{internal data} underlying the diagram for the tree $f \in \Hom(\underline{m},{1})$ are $(\sigma,u,B)$, where $\sigma \in S_m$ is the permutation of the inputs by the braid underneath the tree, $u \in \mathbb{N}^{m+2}$ is a vector detailing the number of units between each input, and $B$ is a choice of bracketing of the resulting multiplication tree. Two diagrams are isomorphic without using the associator, unitors, or commutator if and only if they have the same internal data. Likewise, given internal data $(\sigma, u, B)$ we can define a diagram $D(\sigma, u, B)$ in {\bf Cat} in a certain normal form which we will not detail precisely.

An internal 2-cell in a braided or symmetric monoidal category is a pasting diagram of associators, unitors and commutators which maps source internal data $(\sigma_1, u_1, B_1)$ into target internal data $(\sigma_2, u_2, B_2)$. Given such an internal 2-cell $L: (\sigma_1, u_1, B_1) \to  (\sigma_2, u_2, B_2)$, we may define a movie $D(L): D(\sigma_1, u_1, B_1) \to D(\sigma_2, u_2, B_2)$ in {\bf Cat} which executes it (the precise choice is irrelevant from the internal perspective). 

Given two internal 2-cells $L_1, L_2: (\sigma_1, u_1, B_1) \to  (\sigma_2, u_2, B_2)$, consider the loop $D(L_2)^{-1} D(L_1)$. By Theorem~\ref{thm:maintheorem}:
\begin{itemize}
\item For a braided pseudomonoid, this loop is the identity if and only if the absorbed braid is trivial, indicating that the internal 2-cells are equal if and only if they have the same underlying braid. 
\item For a symmetric pseudomonoid, this loop is always the identity, indicating that two 2-morphisms inducing the same  permutation $\sigma_2 \sigma_1^{-1}$ are always equal.
\end{itemize}
\end{proof}
\noindent
We now commence the proof of Theorem \ref{thm:maintheorem}. We begin with some easy steps.

\begin{lemma}
The maps described above are essentially surjective on objects.
\end{lemma}
\begin{proof}
Clear; they are actually surjective.
\end{proof}

\begin{lemma}
The maps described above are essentially surjective on 1-cells. 
\end{lemma}
\begin{proof}
The decategorified functors are isomorphisms, so there must be a chain of equalities reducing any 1-cell diagram to one in the image of the isomorphism. In the categorified setting these equalities become isomorphisms, which implies essential surjectivity. 
\end{proof}

\begin{lemma}
The maps described above are faithful on 2-cells. 
\end{lemma}

\begin{proof}
For the locally discrete bicategories, this is trivial. For ${\bf FS^{\text{br}}}$, note that the isotopy class of the braid absorbed by each tree is different for every 2-cell in the domain; since none of the generating 2-cells of the computad for a braided pseudomonoid in a symmetric Gray monoid change the isotopy class of the absorbed braid, the map must therefore be faithful on 2-cells.
\end{proof}
\noindent
All that remains to show is that the maps are full on 2-cells and are functorial.

\section{Proof of fullness}\label{sec:fullnessproof}

\subsection{Putting the loop into normal form}\label{sec:normalformfullness}

We now demonstrate that the maps defined in the last section are full on 2-cells. To do this, we will provide an explicit series of rewrites that puts any loop on a 1-cell in the image of one of the maps into a normal form $N$ which we now define. Recall that ordered string diagrams in the image all consist of a braid (possibly trivial) followed by trees.

\begin{definition}\label{movienormalformfullnessdefn}
The normal form $N$ is as follows: A braid (possibly trivial) is created directly beneath the leftmost tree, then absorbed according to Definition \ref{braidedbiequivalencedefinition}. This process is repeated for each tree, moving from left to right.
\end{definition}
\noindent
In Table~\ref{tbl:pseudomonoidsintro} there are two variables: the braided structure of the ambient monoidal bicategory, and the braided structure of the pseudomonoid. Because of our choice of axioms, we need only consider the case of a braided pseudomonoid in a symmetric monoidal bicategory when defining our series of rewrites. This is because loops in other categories may be considered as loops in this Gray monoid which use only a restricted set of 1- and 2-cells. 
Only once we have rewritten the loop in the normal form $N$ will we need to distinguish the various cases.

Firstly we remove all unitors from the loop.

\subsubsection{Removing unitors}\label{sec:removingunits}

Since there are no attached unit nodes in any 1-cell diagram in the image, any unit creation operator in the loop is paired with a unit destruction operator which destroys the created unit. The intuitive idea of the series of rewrites we are about to define is to move a unit creation operator towards the end of the loop, where at some point it will meet its paired destruction operator and the two can be cancelled. The process may then be iterated to remove all unit creation and destruction operators.

We first demonstrate how this can be done in a simple case.

\begin{lemma}\label{unbraidedbubbleelimination}
Any unit creation operator may be eliminated along with its paired destruction operator if it satisfies the following conditions:
\begin{itemize}[leftmargin=*]
\item The only 2-cells acting on the created multiplication 1-cell throughout the loop are its creation operator and its destruction operator.
\item No further unit creation operators occur on the output string of the created unit 1-cell.
\end{itemize}
\end{lemma}

\begin{proof}
The rewrite procedure is as follows. 
\begin{enumerate}[leftmargin=*]
\item Put the created unit in TSNF using Procedure \ref{tsnfprocedure}.
\item Try to move the unit creation operator towards the end of the loop using Type I interchanger rewrites. The possible obstructions are as follows.
	\begin{itemize}[leftmargin=*]
	\item \emph{We have reached the paired destruction operator}. Cancel the pair and we have finished. 
	\item \emph{A braiding inverse-insert, an inverse syllepsis, or a unit creation operator occurs at a vertical level between the created multiplication and the unit.} Since the unit is in TSNF, and by assumption there are no further units created on the output string, this 2-cell will affect a rectangular subregion on one side of the unit's output string. We may therefore:
		\begin{enumerate}[leftmargin=*]
		\item Insert interchangers and their inverses so that, after the 2-cell occurs, the unit interchanges upwards with the created 1-cells, and then returns. 
		\item Use a Type III rewrite so that the 2-cell occurs below the unit, then the unit interchanges downwards.
		\end{enumerate}
		This reduces by one the number of 2-cells occuring between the two 1-cells.
	\item \emph{A series of downwards interchangers and pullthroughs of the unit occurs}. In this case, go to the last 2-cell in the series and try to delay it using Type I rewrites. If this is impossible, there must be an obstruction. If the obstruction is a 2-cell acting on the unit, then since only interchangers and pullthroughs act on the unit, the 2-cell must be an upwards interchanger or pullthrough. This may be cancelled with the downwards interchanger or pullthrough, reducing the number of interchangers and pullthroughs of the unit by two. If the 2-cell acts on the level directly above the unit, there are two possibilities, depending on what the 2-cell is:
		\begin{itemize}[leftmargin=*]
		\item \emph{A braiding inverse-insert, an inverse syllepsis or a unit creation operator at a level between the two 1-cells.} Since the unit is in TSNF, and by assumption there are no further units created on the output string, this 2-cell will target a region on one side of the unit. We: 
			\begin{enumerate}[leftmargin=*]
				\item Insert interchangers and their inverses so that, after the 2-cell occurs, the unit interchanges upwards with the created 1-cells, and then returns. 
				\item Use a Type III rewrite so that the 2-cell occurs below the unit, then the unit interchanges downwards.
			\end{enumerate}
		This reduces by one the number of 2-cells occuring between the multiplication and the unit.	
		\item \emph{Any other 2-cell}. In this case, the interchanger will have interchanged downwards with all the involved 1-cells; we may therefore use a Type III rewrite, reducing by one the number of 2-cells occuring between the multiplication and the unit.	
		\end{itemize}
	\end{itemize}
\item Iterate the procedure. Since all paths above either cancel the unit creation and destruction operators, reduce the number of interchangers of the unit, or reduce the the number of 2-cells occuring at a level between the two 1-cells, it is clear that this result in cancellation of the creation and destruction operators.\end{enumerate}\end{proof}
\begin{example}\label{ex:unitelim}
See the 5-cell `Example~\ref{ex:unitelim} - Statement' and the 6-cell `Example~\ref{ex:unitelim} - Pf' in the \emph{Globular} workspace. The first 20 rewriting steps put the created unit into TSNF, and the remaining 26 steps eliminate it.
\end{example}

\begin{definition}
We call unit creation operators satisfying the conditions of Lemma~\ref{unbraidedbubbleelimination} \emph{unnested with fixed multiplication}.
\end{definition}

\noindent
We now show how to rewrite any loop so that the final unit creation operator is unnested with fixed multiplication.

\begin{lemma}\label{bubblenormalformwithfixedmultnoderewrite}
Any loop can be rewritten so that the last unit creation operator is unnested with fixed multiplication, without increasing the number of unit creation operators. 
\end{lemma}

\begin{proof}
The last unit creation operator is clearly unnested. We now show how to fix the created multiplication node without introducing nesting. Consider the first 2-cell involving the created multiplication 1-cell. If this 2-cell is a unit destruction operator then we are finished. If not:
\begin{enumerate}[leftmargin=*]
\item Insert interchangers, pullthroughs and their inverses (IPI) immediately prior to the 2-cell so that the unit node goes straight up to the multiplication 1-cell, returns to where it started and then the 2-cell occurs. 
\item Insert a unit destruction operator and its inverse immediately before the pulldowns. 
\item Eliminate the first unit creation operator and the inserted destruction operator using Lemma~\ref{unbraidedbubbleelimination}. We can do this since by assumption this was the first 2-cell acting on the created unit, and the unit is unnested.
\item Use Type I rewrites so that the 2-cell occurs immediately after the unit creation (or after additional interchangers/pullthroughs if necessary).
\end{enumerate}
We now have a movie with the same number of unit creation operators where the last creation operator occurs, the unit interchanges or pulls through downwards to directly beneath the region acted on by the 2-cell involving the multiplication, and then the 2-cell occurs. We now show how to eliminate each possible 2-cell case-by-case. 

\begin{itemize}[leftmargin=*]
\item \emph{The multiplication interchanges downwards}. In this case the creation operator occurs and then both the unit and the multiplication interchange once downwards. Use a Type III rewrite so that the creation operator occurs immediately below the 1-cell involved in the interchanger:

\begin{center}
\begin{tabular}{c c c c c c}
\raisebox{0.5cm}{{\Huge $[$}} &
\includegraphics[width=1cm,height=1.5cm]{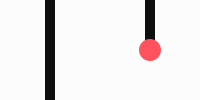} 
\raisebox{0.5cm}{$\Rightarrow$} 
\includegraphics[width=1cm,height=1.5cm]{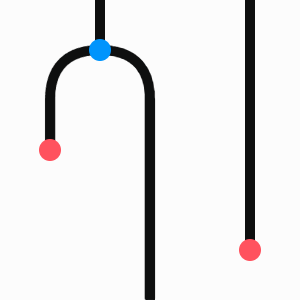} 
\raisebox{0.5cm}{$\Rightarrow$} 
\includegraphics[width=1cm,height=1.5cm]{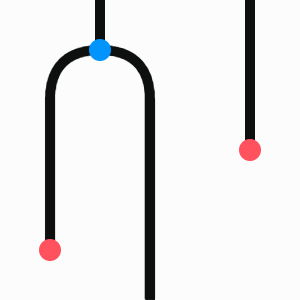} 
\raisebox{0.5cm}{$\Rightarrow$}
\includegraphics[width=1cm,height=1.5cm]{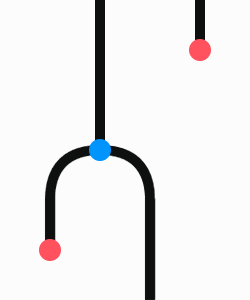} &
\raisebox{0.5cm}{{\Huge $]$}} &
\raisebox{0.5cm}{{\Large $=$}} &
\raisebox{0.5cm}{{\Huge $[$}} &
\includegraphics[width=1cm,height=1.5cm]{multintdownwardsstart} 
\raisebox{0.5cm}{$\Rightarrow$} 
\includegraphics[width=1cm,height=1.5cm]{multintdownwardsend} 
\raisebox{0.5cm}{{\Huge $]$}} 
\end{tabular}
\end{center}

\item \emph{The multiplication interchanges upwards}. In this case the creation operator occurs and then the multiplication interchanges upwards. Insert a upwards interchanger of the unit and its inverse immediately following the upwards interchanger of the multiplication. Use a Type II then a Type III rewrite so that the creation operator occurs immediately above the 1-cell involved in the interchanger:
\begin{center}
\begin{tabular}{c c c}
\raisebox{0.5cm}{{\Huge $[$}}
\includegraphics[width=1cm,height=1.5cm]{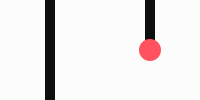} 
\raisebox{0.5cm}{$\Rightarrow$} 
\includegraphics[width=1cm,height=1.5cm]{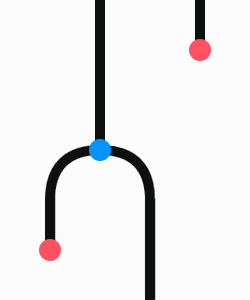} 
\raisebox{0.5cm}{$\Rightarrow$} 
\includegraphics[width=1cm,height=1.5cm]{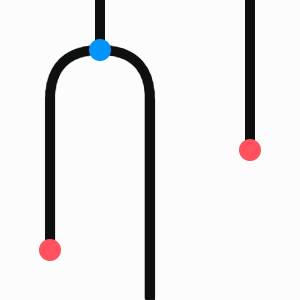} 
\raisebox{0.5cm}{{\Huge $]$}} &
\raisebox{0.5cm}{{\Large $=$}} &
\raisebox{0.5cm}{{\Huge $[$}}
\includegraphics[width=1cm,height=1.5cm]{multintupwardsstart} 
\raisebox{0.5cm}{$\Rightarrow$} 
\includegraphics[width=1cm,height=1.5cm]{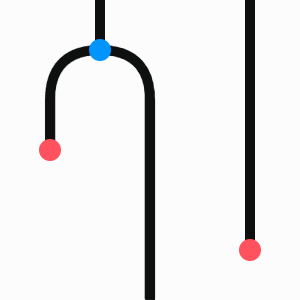} 
\raisebox{0.5cm}{$\Rightarrow$} 
\includegraphics[width=1cm,height=1.5cm]{multintupwardsend}\raisebox{0.5cm}{{\Huge $]$}}
\end{tabular}
\end{center}

\item \emph{The multiplication pulls through downwards}. In this case the creation operator occurs immediately above a braiding, the unit pulls through, and is followed by the multiplication node. Use $(\cdot \otimes \Rightarrow)$ or $(\Rightarrow \otimes \cdot)$ so that the creation operator occurs beneath the braiding:

\begin{center}
\begin{tabular}{c c c}
\raisebox{0.5cm}{{\Huge $[$}}
\includegraphics[width=1cm,height=1.5cm]{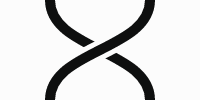} 
\raisebox{0.5cm}{$\Rightarrow$} 
\includegraphics[width=1cm,height=1.5cm]{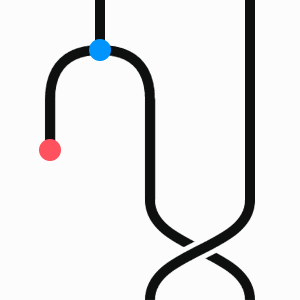} 
\raisebox{0.5cm}{$\Rightarrow$} 
\includegraphics[width=1cm,height=1.5cm]{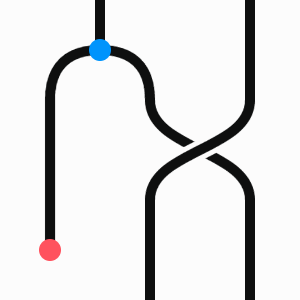}  
\raisebox{0.5cm}{$\Rightarrow$} 
\includegraphics[width=1cm,height=1.5cm]{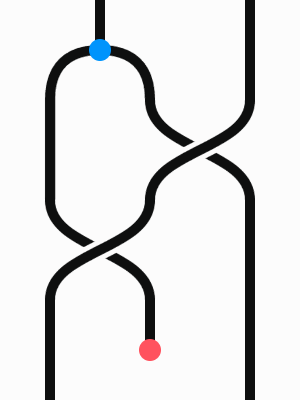} 
\raisebox{0.5cm}{$\Rightarrow$} 
\includegraphics[width=1cm,height=1.5cm]{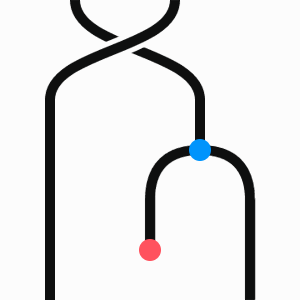}
\raisebox{0.5cm}{{\Huge $]$}} &
\raisebox{0.5cm}{{\Large $=$}} &
\raisebox{0.5cm}{{\Huge $[$}}
\includegraphics[width=1cm,height=1.5cm]{multptdownwardsstart}
\raisebox{0.5cm}{$\Rightarrow$} 
\includegraphics[width=1cm,height=1.5cm]{multptdownwardsend}\raisebox{0.5cm}{{\Huge $]$}}
\end{tabular}
\end{center}

\item \emph{The multiplication pulls through upwards}.
Here the creation operator occurs immediately below a braiding, and the multiplication then pulls through upwards. Insert an upwards pullthrough of the unit followed by a downwards pullthrough immediately after the pullthrough of the multiplication. Use $(\cdot \otimes \Rightarrow)$ or $(\Rightarrow \otimes \cdot)$ so that the creation operator occurs above the braiding and the unit then pulls through downwards:
\begin{center}
\begin{tabular}{c c c}
\raisebox{0.5cm}{{\Huge $[$}}
\includegraphics[width=1cm,height=1.5cm]{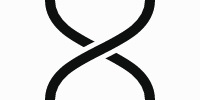} 
\raisebox{0.5cm}{$\Rightarrow$} 
\includegraphics[width=1cm,height=1.5cm]{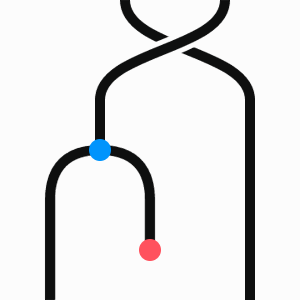} 
\raisebox{0.5cm}{$\Rightarrow$} 
\includegraphics[width=1cm,height=1.5cm]{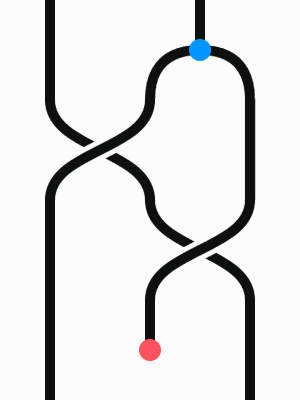}
\raisebox{0.5cm}{{\Huge $]$}} &
\raisebox{0.5cm}{{\Large $=$}} &
\raisebox{0.5cm}{{\Huge $[$}}
\includegraphics[width=1cm,height=1.5cm]{multptupwardsstart}
\raisebox{0.5cm}{$\Rightarrow$} 
\includegraphics[width=1cm,height=1.5cm]{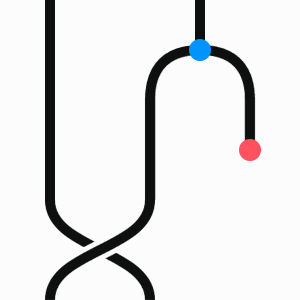}
\raisebox{0.5cm}{$\Rightarrow$} 
\includegraphics[width=1cm,height=1.5cm]{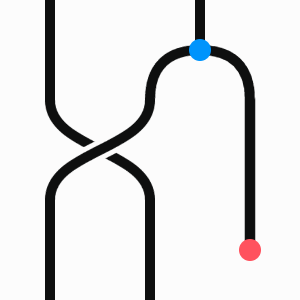}
\raisebox{0.5cm}{$\Rightarrow$} 
\includegraphics[width=1cm,height=1.5cm]{multptupwardsend}\raisebox{0.5cm}{{\Huge $]$}}
\end{tabular}
\end{center}

\item \emph{The multiplication is the lower partner in an associator or inverse associator.} In this case the creation operator is performed immediately below a multiplication 1-cell, with which the created multiplication 1-cell immediately associates. Here we require four equalities, all of which are implied by the triangle equality. Two are shown below; the others are the same, but with all diagrams flipped in a vertical axis (we call them (A1V) and (A2V). 

\begin{center}
\begin{tabular}{l l l r}
\raisebox{0.5cm}{{\Huge $[$}}
\includegraphics[width=1cm,height=1.5cm]{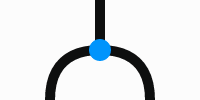} 
\raisebox{0.5cm}{$\Rightarrow$} 
\includegraphics[width=1cm,height=1.5cm]{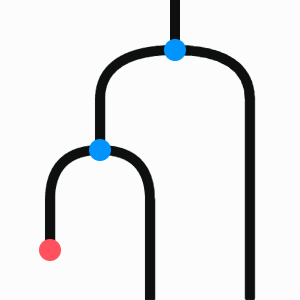} 
\raisebox{0.5cm}{$\Rightarrow$} 
\includegraphics[width=1cm,height=1.5cm]{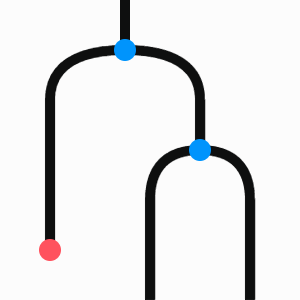}
\raisebox{0.5cm}{{\Huge $]$}} &
\raisebox{0.5cm}{{\Large $=$}} &
\raisebox{0.5cm}{{\Huge $[$}}
\includegraphics[width=1cm,height=1.5cm]{multlunitlassocstart}
\raisebox{0.5cm}{$\Rightarrow$} 
\includegraphics[width=1cm,height=1.5cm]{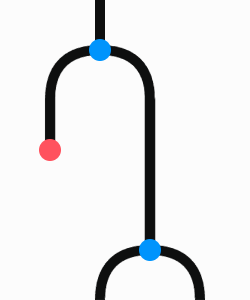}
\raisebox{0.5cm}{$\Rightarrow$} 
\includegraphics[width=1cm,height=1.5cm]{multlunitlassocend}
\raisebox{0.5cm}{{\Huge $]$}}
& \raisebox{0.5cm}{(A1)}
\\
\raisebox{0.5cm}{{\Huge $[$}}
\includegraphics[width=1cm,height=1.5cm]{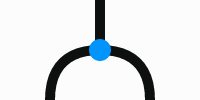} 
\raisebox{0.5cm}{$\Rightarrow$} 
\includegraphics[width=1cm,height=1.5cm]{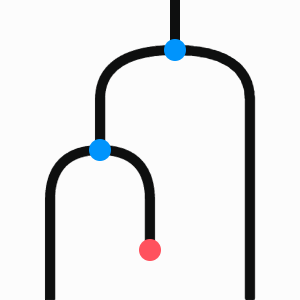} 
\raisebox{0.5cm}{$\Rightarrow$} 
\includegraphics[width=1cm,height=1.5cm]{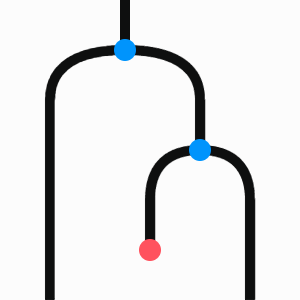}
\raisebox{0.5cm}{{\Huge $]$}} &
\raisebox{0.5cm}{{\Large $=$}} &
\raisebox{0.5cm}{{\Huge $[$}}
\includegraphics[width=1cm,height=1.5cm]{multrunitlassocstart}
\raisebox{0.5cm}{$\Rightarrow$} 
\includegraphics[width=1cm,height=1.5cm]{multrunitlassocend}
\raisebox{0.5cm}{{\Huge $]$}}
& \raisebox{0.5cm}{(A2)}
\end{tabular}
\end{center}
The full derivations of these rewrites can be found in the \emph{Globular} workspace as the 6-cells `Lemma \ref{bubblenormalformwithfixedmultnoderewrite} - Mult lower partner in associator Pf (A1)' and `Lemma \ref{bubblenormalformwithfixedmultnoderewrite} - Mult lower partner in associator Pf (A2)'. In the proof we use two intermediate lemmas, the 5-cells `Lemma~\ref{bubblenormalformwithfixedmultnoderewrite} - Associator Lemma Left' and `Lemma~\ref{bubblenormalformwithfixedmultnoderewrite} - Associator Lemma Right'. We give the proof for the left lemma  as the 6-cell `Lemma~\ref{bubblenormalformwithfixedmultnoderewrite} - Associator Lemma Left Pf'; the proof for the right lemma is similar.

\item \emph{The multiplication is the upper partner in an associator or inverse associator.} Here the unit creation operator occurs directly above a multiplication 1-cell, the unit interchanges downwards, and an associator or inverse associator is then performed. For a left unit, this will be an associator, and for a right unit it will be an inverse associator. We require two equalities, one of which is (A1) postcomposed on both sides with an inverse associator, and the other of which is (A1V) postcomposed on both sides of the equality with an associator.

\item \emph{The multiplication annihilates with another unit.} In this case, the creation operator occurs directly above a unit. The created unit then interchanges downwards and the multiplication annihilates with the other unit. To rewrite this movie we need one equality for a left unit creation, derivable from the triangle equality. The equality for a right unit creation is simply the flip of this one.
\begin{center}
\begin{tabular}{l l l}
\raisebox{0.5cm}{{\Huge $[$}}
\includegraphics[width=1cm,height=1.5cm]{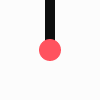} 
\raisebox{0.5cm}{$\Rightarrow$} 
\includegraphics[width=1cm,height=1.5cm]{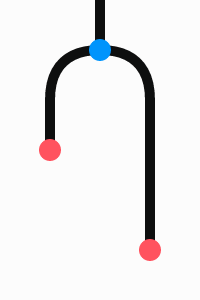} 
\raisebox{0.5cm}{$\Rightarrow$} 
\includegraphics[width=1cm,height=1.5cm]{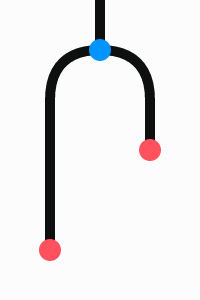} 
\raisebox{0.5cm}{$\Rightarrow$} 
\includegraphics[width=1cm,height=1.5cm]{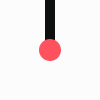}
\raisebox{0.5cm}{{\Huge $]$}} &
\raisebox{0.5cm}{{\Large $=$}} &
\raisebox{0.5cm}{{\Huge $[$}}
\includegraphics[width=1cm,height=1.5cm]{multunitorstart}
\raisebox{0.5cm}{{\Huge $]$}}
\end{tabular}
\end{center}
We provide the full derivation of this equality in the \emph{Globular} workspace as the 6-cell `Lemma~\ref{bubblenormalformwithfixedmultnoderewrite} - Left unit twist Pf'. For convenience, we include the 5-cell `Lemma~\ref{bubblenormalformwithfixedmultnoderewrite} - Right unit twist' separately.

\item \emph{The multiplication is acted on by a commutator or inverse commmutator}. For a commutator the unit is created, pulls through the other string, and a commutator occurs. For an inverse commutator, the unit is created and then an inverse commutator occurs. One equality for a commutator is shown below; the other equality is the same, but with all diagrams flipped in a vertical axis. The two equalities for inverse commutators follow from the equalities for commutators by flipping and then postcomposing on both sides with a unit creation and a pullthrough.
\begin{center}
\begin{tabular}{l l l }
\raisebox{0.5cm}{{\Huge $[$}}
\includegraphics[width=1cm,height=1.5cm]{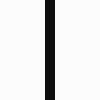} 
\raisebox{0.5cm}{$\Rightarrow$} 
\includegraphics[width=1cm,height=1.5cm]{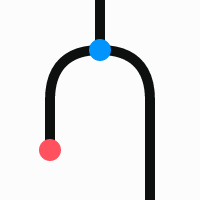} 
\raisebox{0.5cm}{$\Rightarrow$} 
\includegraphics[width=1cm,height=1.5cm]{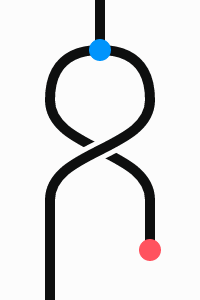}
\raisebox{0.5cm}{$\Rightarrow$} 
\includegraphics[width=1cm,height=1.5cm]{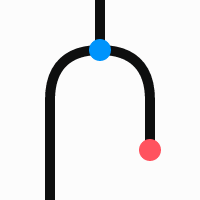}
\raisebox{0.5cm}{{\Huge $]$}} &
\raisebox{0.5cm}{{\Large $=$}} &
\raisebox{0.5cm}{{\Huge $[$}}
\includegraphics[width=1cm,height=1.5cm]{multinvcommrunitstart}
\raisebox{0.5cm}{$\Rightarrow$} 
\includegraphics[width=1cm,height=1.5cm]{multinvcommrunit1-1} 
\raisebox{0.5cm}{{\Huge $]$}} 
\end{tabular}
\end{center}
The full derivation of this rewrite is shown in the \emph{Globular} workspace as `Lemma~\ref{bubblenormalformwithfixedmultnoderewrite} - Commutator Lemma Pf'.
\end{itemize}
All these rewrites remove the first 2-cell on the multiplication 1-cell. The result therefore follows by iterating the procedure.
\end{proof}
\noindent
Using these results together, we may remove all unitors. 
\begin{proposition}
Any loop on a 1-morphism in the image of $F$ may be rewritten so that it contains no unit creation or destruction operators. 
\end{proposition} 
\begin{proof}
Use Lemma~\ref{bubblenormalformwithfixedmultnoderewrite} to rewrite the movie such that the last unit creation operator is unnested with fixed multiplication; then eliminate it using Lemma \ref{unbraidedbubbleelimination}. Repeat until all unit creation operators have been removed.
\end{proof}

\subsubsection{Fixing the trees}\label{sec:treefixing}

We now have a loop consisting only of associators, interchangers, commutators, pullthroughs, syllepses, braiding inverse-inserts and braiding cancellations. Recall that the source 1-cell diagram of the loop is a braiding followed by a series of left-bracketed multiplication trees with heights rising from left to right, where we consider a unit 1-cell to be a multiplication tree $m^0$. We will now provide a series of rewrites that will `fix the trees'. The intuitive meaning of this is shown in Figure \ref{fixedtreesfigure}.
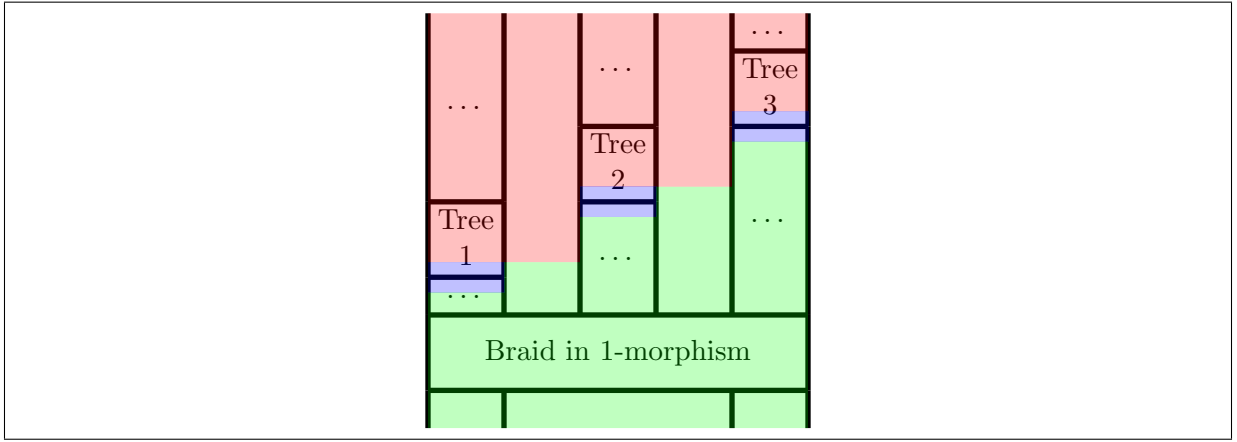
\begin{figure}
\centering
\begin{tikzpicture}
	\begin{pgfonlayer}{nodelayer}
		\node [style=none] (0) at (0, -0) {};
		\node [style=none] (1) at (1, -0) {};
		\node [style=none] (2) at (0, 0.5) {};
		\node [style=none] (3) at (0, 1.5) {};
		\node [style=none] (4) at (5, 1.5) {};
		\node [style=none] (5) at (5, 0.5) {};
		\node [style=none] (6) at (5, -0) {};
		\node [style=none] (7) at (4, -0) {};
		\node [style=none] (8) at (1, 0.5) {};
		\node [style=none] (9) at (4, 0.5) {};
		\node [style=none] (10) at (1, 1.5) {};
		\node [style=none] (11) at (4, 1.5) {};
		\node [style=none] (12) at (2, 1.5) {};
		\node [style=none] (13) at (3, 1.5) {};
		\node [style=none] (14) at (0, 2) {};
		\node [style=none] (15) at (0, 3) {};
		\node [style=none] (16) at (1, 3) {};
		\node [style=none] (17) at (1, 2) {};
		\node [style=none] (18) at (0.5, 2.5) {\makecell{Tree\\1}};
		\node [style=none] (19) at (2, 3) {};
		\node [style=none] (20) at (2, 4) {};
		\node [style=none] (21) at (3, 4) {};
		\node [style=none] (22) at (3, 3) {};
		\node [style=none] (23) at (2.5, 3.5) {\makecell{Tree\\2}};
		\node [style=none] (24) at (4, 4) {};
		\node [style=none] (25) at (4, 5) {};
		\node [style=none] (26) at (4.5, 4.5) {\makecell{Tree\\3}};
		\node [style=none] (27) at (5, 4) {};
		\node [style=none] (28) at (5, 5) {};
		\node [style=none] (29) at (2.5, 2.25) {$\dots$};
		\node [style=none] (30) at (4.5, 2.75) {$\dots$};
		\node [style=none] (31) at (0.5, 1.75) {$\dots$};
		\node [style=none] (32) at (5, 5.5) {};
		\node [style=none] (33) at (4, 5.5) {};
		\node [style=none] (34) at (3, 5.5) {};
		\node [style=none] (35) at (2, 5.5) {};
		\node [style=none] (36) at (0, 5.5) {};
		\node [style=none] (37) at (1, 5.5) {};
		\node [style=none] (38) at (0.5, 4.25) {$\dots$};
		\node [style=none] (39) at (2.5, 4.75) {$\dots$};
		\node [style=none] (40) at (4.5, 5.25) {$\dots$};
		\node [style=none] (41) at (2.5, 1) {\makecell{Braid in 1-morphism}};
	\end{pgfonlayer}
	\begin{pgfonlayer}{edgelayer}
		\draw [style=simple] (0.center) to (2.center);
		\draw [style=simple] (1.center) to (8.center);
		\draw [style=simple] (7.center) to (9.center);
		\draw [style=simple] (6.center) to (5.center);
		\draw [style=simple] (3.center) to (2.center);
		\draw [style=simple] (3.center) to (4.center);
		\draw [style=simple] (4.center) to (5.center);
		\draw [style=simple] (5.center) to (2.center);
		\draw [style=simple] (14.center) to (3.center);
		\draw [style=simple] (17.center) to (10.center);
		\draw [style=simple] (19.center) to (12.center);
		\draw [style=simple] (22.center) to (13.center);
		\draw [style=simple] (24.center) to (11.center);
		\draw [style=simple] (27.center) to (4.center);
		\draw [style=simple] (15.center) to (14.center);
		\draw [style=simple] (16.center) to (15.center);
		\draw [style=simple] (16.center) to (17.center);
		\draw [style=simple] (17.center) to (14.center);
		\draw [style=simple] (19.center) to (20.center);
		\draw [style=simple] (20.center) to (21.center);
		\draw [style=simple] (21.center) to (22.center);
		\draw [style=simple] (22.center) to (19.center);
		\draw [style=simple] (25.center) to (28.center);
		\draw [style=simple] (28.center) to (27.center);
		\draw [style=simple] (25.center) to (24.center);
		\draw [style=simple] (24.center) to (27.center);
		\draw [style=simple] (36.center) to (15.center);
		\draw [style=simple] (37.center) to (16.center);
		\draw [style=simple] (20.center) to (35.center);
		\draw [style=simple] (34.center) to (21.center);
		\draw [style=simple] (33.center) to (25.center);
		\draw [style=simple] (32.center) to (28.center);
	\end{pgfonlayer}
\fill[red,nearly transparent] (0,2.2) -- (2,2.2) -- (2,3.2) -- (4,3.2) -- (4,4.2) -- (5,4.2) -- (5,5.5) -- (0,5.5) -- cycle;
\fill[blue,nearly transparent] (0,1.8) -- (1,1.8) -- (1,2.2) -- (0,2.2) -- cycle;
\fill[blue,nearly transparent] (2,2.8) -- (3,2.8) -- (3,3.2) -- (2,3.2) -- cycle;
\fill[blue,nearly transparent] (4,3.8) -- (5,3.8) -- (5,4.2) -- (4,4.2) -- cycle;
\fill[green,nearly transparent] (0,0) -- (5,0) -- (5,3.8) -- (4,3.8) -- (4,3.2) -- (3,3.2) -- (3,2.8) -- (2,2.8) -- (2,2.2) -- (1,2.2) --(1,1.8) -- (0,1.8) -- cycle;
\end{tikzpicture}
\caption{This figure shows fixed trees. In the green region, only structural 2-cells may occur. Commutators occur only at the bottom of the trees, in the blue regions. In the red region, only associators may occur.}
\label{fixedtreesfigure}
\end{figure}
We may write the condition as follows.
\begin{lemma}
The loop can be rewritten so that non-structural 1-cells are in TSNF; there are no interchangers between non-structural 1-cells; and commutators only occur on multiplication 1-cells at the bottom of a tree.
\end{lemma}

\begin{proof}
Begin with the leftmost tree. Consider the source diagram. Call the top 1-cell in the tree $N_1$, the next 1-cell down $N_2$, etc. Put $N_1$  into TSNF. At the start and end of the loop use Type II rewrites so that $N_1$ moves to the top of the diagram and then back down again. Now we will rewrite the loop so that $N_1$ is at the top of the diagram when any non-braiding 2-cell occurs. Consider the first non-braiding 2-cell in the loop:
\begin{itemize}[leftmargin=*]
\item \emph{The first non-braiding 2-cell involves $N_1$}.
\begin{enumerate}[leftmargin=*]
\item Insert IPI immediately before the 2-cell so that $N_1$ and the other 1-cells acted on by the 2-cell rise together to the top of the diagram, then return, then the 2-cell occurs. There can be no obstructing 1-cells above $N_1$ since it was the highest 1-cell in its tree, so this is always possible.
\item Use Type III rewrites so that $N_1$ and the other 1-cells involved in the 2-cell rise together to the top of the diagram, the 2-cell occurs, then they return.
\end{enumerate}
\item \emph{The first non-braiding 2-cell does not involve $N_1$}. 
\begin{enumerate}[leftmargin=*]
\item Insert IPI immediately before the 2-cell so that $N_1$ rises to the top of the diagram, then returns down again, then the 2-cell occurs. There can be no obstructing 1-cells for the reason already stated. 
\item Use Type I rewrites so that $N_1$ rises to the top of the diagram, returns down to just below the region on which the 2-cell acts, then the 2-cell occurs, then $N_1$ returns to its original position. 
\item Use a Type III rewrite followed by Type I rewrites so that $N_1$ rises to the top of the diagram, the 2-cell occurs, then $N_1$ returns to its original position. 
\end{enumerate}
\end{itemize}
Repeat this process for all non-braiding 2-cells occuring in the loop. In between the two points at the start and end of the loop where $N_1$ is moved to top of the diagram, we now have a loop where $N_1$ is at the top of the diagram during all non-braiding 2-cells. Consider the clips between the applications of the non-braiding 2-cells. These contain only structural 2-cells, and begin and end with $N_1$ at the top of the diagram. Use Theorem \ref{gurskiosornocoherencethmextended} to rewrite these so that $N_1$ remains at the top of the diagram throughout the whole loop. We now have a loop where $N_1$ rises to the top of the diagram in the beginning, remains there throughout, then returns to its original position.

Now we show that, having fixed $N_i$, we can fix $N_{i+1}$; that is, we can rewrite so that $N_{i+1}$ remains in TSNF directly beneath $N_i$ throughout the loop. First put $N_{i+1}$ in TSNF; this will mean that we can always pull it up to directly beneath $N_i$. At the beginning and end of the loop, after $N_i$ rises to the top, use Type II rewrites so that $N_{i+1}$ rises directly beneath $N_i$ and then returns to its starting position. Now we rewrite so that $N_{i+1}$ is directly beneath $N_i$ whenever a non-braiding 2-cell is performed. Consider the first 2-cell in the movie; there are three possibilities.

\begin{itemize}[leftmargin=*]
\item \emph{The 2-cell acts on $N_i$ and on $N_{i+1}$.} Here the 2-cell must be an associator. In this case, since $N_{i+1}$ must be directly beneath $N_i$ for the performance of the 2-cell, no rewrite is necessary.
\item \emph{The 2-cell acts on $N_i$ but not on $N_{i+1}$.} In this case, the 2-cell will be either an associator or a commutator on $N_i$.
	\begin{itemize}[leftmargin=*]
	\item \emph{The 2-cell is an associator or inverse associator}. Insert IPI so that $N_{i+1}$ pulls up directly beneath $N_i$, returns to its starting position, then the associator is performed. Use Type I rewrites so that $N_{i+1}$ pulls up directly beneath $N_i$, then interchanges downwards once, then the associator is performed, then $N_{i+1}$ returns to its starting position. Then use the pentagon equality so that $N_{i+1}$ moves directly beneath $N_i$, a series of associators are performed, then the other multiplication 1-cell moves back to the starting position of $N_i$:
	\begin{center}
\begin{tabular}{l l}
\raisebox{0.5cm}{{\Huge $[$}}
&
\includegraphics[width=1cm,height=1.5cm]{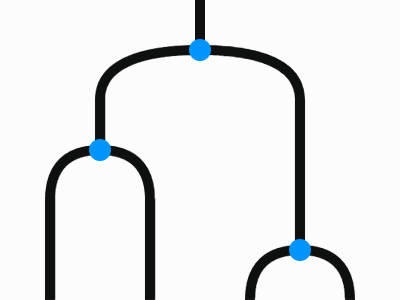} 
\raisebox{0.5cm}{$\Rightarrow$} 
\includegraphics[width=1cm,height=1.5cm]{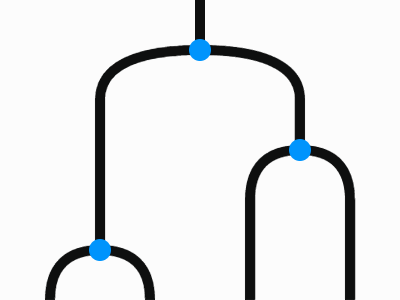} 
\raisebox{0.5cm}{$\Rightarrow$} 
\includegraphics[width=1cm,height=1.5cm]{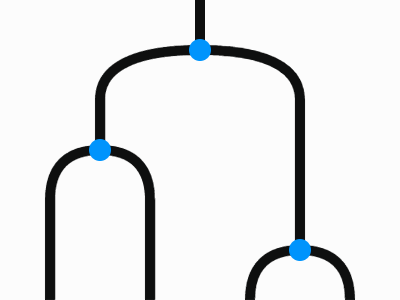}
\raisebox{0.5cm}{$\Rightarrow$} 
\includegraphics[width=1cm,height=1.5cm]{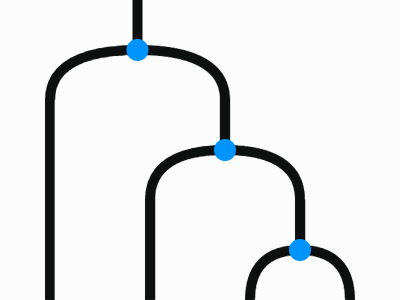}
\raisebox{0.5cm}{{\Huge $]$}} \\
\raisebox{0.5cm}{{\Large $=$}} &
\raisebox{0.5cm}{{\Huge $[$}}
\includegraphics[width=1cm,height=1.5cm]{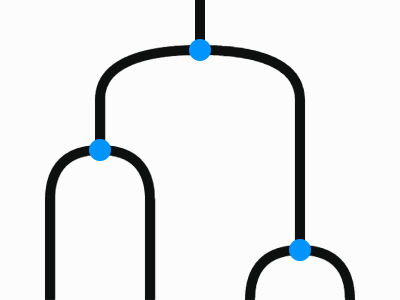}
\raisebox{0.5cm}{$\Rightarrow$} 
\includegraphics[width=1cm,height=1.5cm]{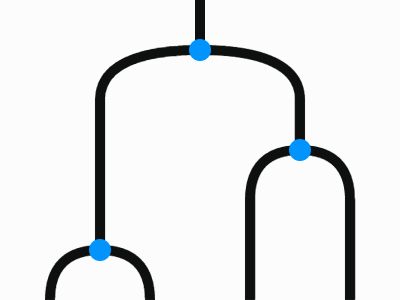}
\raisebox{0.5cm}{$\Rightarrow$} 
\includegraphics[width=1cm,height=1.5cm]{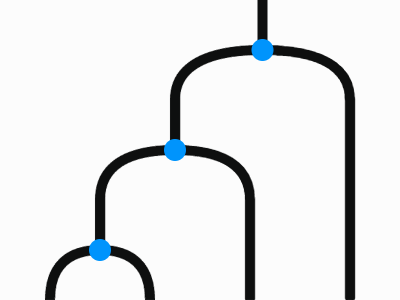}
\raisebox{0.5cm}{$\Rightarrow$} 
\includegraphics[width=1cm,height=1.5cm]{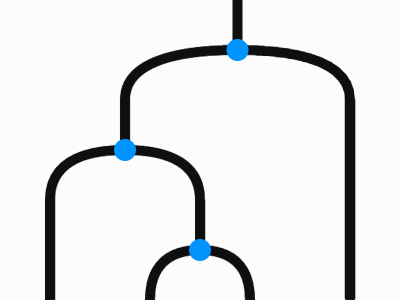}
\raisebox{0.5cm}{$\Rightarrow$} 
\includegraphics[width=1cm,height=1.5cm]{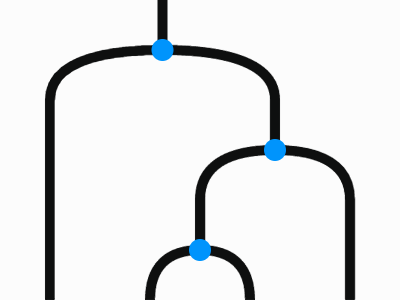}
\raisebox{0.5cm}{$\Rightarrow$} 
\includegraphics[width=1cm,height=1.5cm]{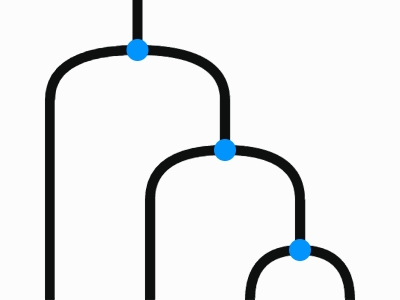}
\raisebox{0.5cm}{{\Huge $]$}}
\end{tabular}
\end{center}
	
	\item \emph{The 2-cell is a commutator or inverse commutator}. Insert IPI so that $N_{i+1}$ pulls up directly beneath $N_i$, then returns to its starting position, then the commutator is performed. Use Type I rewrites so that $N_{i+1}$ is pulled up beneath $N_i$, pulls through downwards, and then the commutator is performed. Then use one of the two hexagons to rewrite the movie so that $N_{i+1}$ pulls up directly beneath $N_i$, then a series of associators and commutators are performed, then $N_{i+1}$ returns to its starting position:
\begin{center}
\begin{tabular}{l l }
\raisebox{0.5cm}{{\Huge $[$}} &
\includegraphics[width=1cm,height=1.5cm]{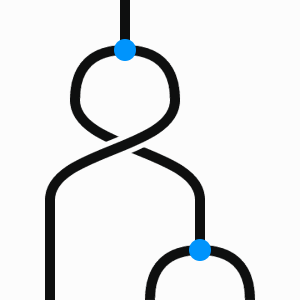} 
\raisebox{0.5cm}{$\Rightarrow$} 
\includegraphics[width=1cm,height=1.5cm]{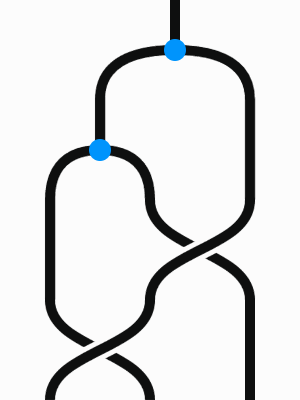} 
\raisebox{0.5cm}{$\Rightarrow$} 
\includegraphics[width=1cm,height=1.5cm]{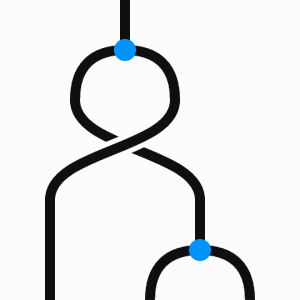}
\raisebox{0.5cm}{$\Rightarrow$} 
\includegraphics[width=1cm,height=1.5cm]{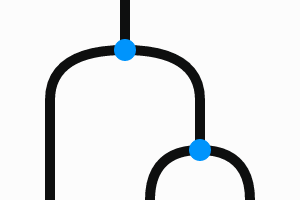}
\raisebox{0.5cm}{{\Huge $]$}} \\
\raisebox{0.5cm}{{\Large $=$}} &
\raisebox{0.5cm}{{\Huge $[$}}
\includegraphics[width=1cm,height=1.5cm]{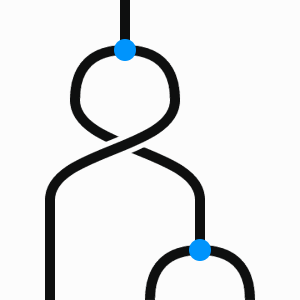}
\raisebox{0.5cm}{$\Rightarrow$} 
\includegraphics[width=1cm,height=1.5cm]{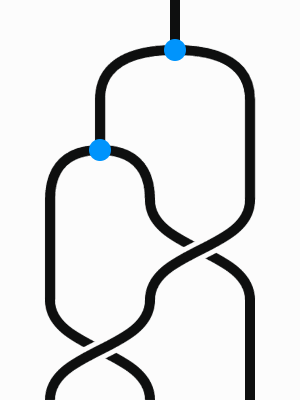}
\raisebox{0.5cm}{$\Rightarrow$} 
\includegraphics[width=1cm,height=1.5cm]{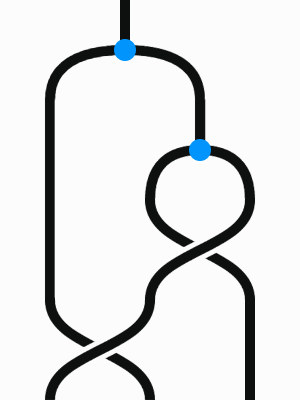}
\raisebox{0.5cm}{$\Rightarrow$} 
\includegraphics[width=1cm,height=1.5cm]{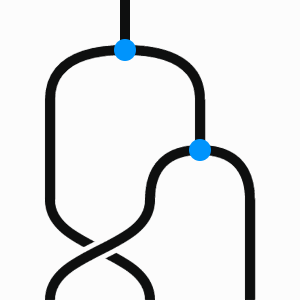}
\raisebox{0.5cm}{$\Rightarrow$} 
\includegraphics[width=1cm,height=1.5cm]{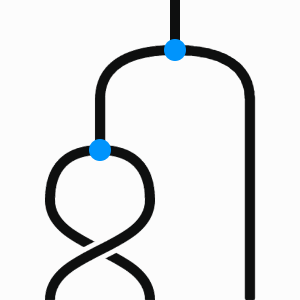}
\raisebox{0.5cm}{$\Rightarrow$} 
\includegraphics[width=1cm,height=1.5cm]{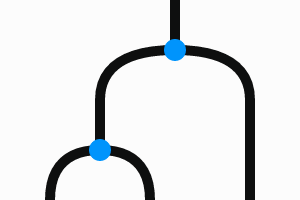}
\raisebox{0.5cm}{$\Rightarrow$} 
\includegraphics[width=1cm,height=1.5cm]{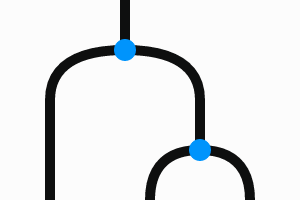}
\raisebox{0.5cm}{{\Huge $]$}}
\end{tabular}
\end{center}
\end{itemize}
\item \emph{The 2-cell acts on $N_{i+1}$ but not on $N_i$.} In this case, the 2-cell must be an associator or commutator involving $N_{i+1}$ and a multiplication 1-cell or braiding beneath it. Use IPI so that both affected 1-cells rise to beneath $N_i$, then use Type III rewrites so that the 1-cells move up, the 2-cell occurs, and then they move back down.
\item \emph{The 2-cell acts neither on $N_i$ nor on $N_{i+1}$.} In this case, insert IPI and use Type I rewrites and possibly a Type III rewrite so that $N_{i+1}$ rises up, the 2-cell occurs, and then $N_{i}$ moves back down again. 
\end{itemize}
By induction, we obtain a loop where each node in the leftmost tree rises to the top of the diagram and remains there throughout, with commutators occuring only on the bottom multiplication 1-cell, before descending again. Repeat for all trees, from left to right; then remove the loop of interchangers at the beginning and end of the movie using Theorem \ref{gurskiosornocoherencethmextended}. The resulting loop will be of the desired form.
\end{proof}

\subsubsection{Associators, commutators and braidings}\label{sec:assoccommsandbraidingssection}

We now finish rewriting the loop into normal form $N$. First we deal with the associators; we ensure that the trees remain left bracketed until a commutator or inverse commutator is about to occur, at which point the affected tree will associate in the manner of Definition \ref{braidedbiequivalencedefinition}, then return to the left bracketing using the inverse sequence of associators when the commutator or inverse commutator is complete.  For this, we use the following lemma.

\begin{lemma}\label{assoccoherence}
Any two unbroken sequences of associators (i.e. without interchangers) between two bracketings of a tree are equal.
\end{lemma}
\begin{proof}
We prove this by induction on the size of the tree $m^{n}$. It is clearly true for $m^1$. Now consider $m^{n+1}$. Follow the progress of the lowest 1-cell throughout the chain of associators. Take the first turning point, where an associator on the lowest 1-cell is followed by an inverse associator on it. In between the associator and inverse associator we have a series of associators of the tree $m^{n}$ above it such the lowest 1-cell in that tree starts and ends in the same position. By the induction hypothesis this may be rewritten using associator cancellations so that the lowest 1-cell in that tree does not move at all. We may then use Type I rewrites to pull the associator up to the inverse associator and cancel the two. Repeating this, we eliminate all movement of the bottom 1-cell in the loop. The result follows. 
\end{proof}\noindent
We may then rewrite the associators.
\begin{lemma}
The loop may be rewritten so that all trees are left bracketed until a commutator or inverse commutator, at which point the affected tree associates in the manner prescribed by Definition \ref{braidedbiequivalencedefinition}, the commutator occurs,  and then the tree associates back to the left bracketing in the manner prescribed by Definition \ref{braidedbiequivalencedefinition}.
\end{lemma}
\begin{proof}
Immediately before every commutator or inverse commutator, insert associators and their inverses so that the tree is rewritten into the bracketing prescribed by Definition \ref{braidedbiequivalencedefinition}, then returns to the original bracketing. Use Type I rewrites so that the commutators and inverse commutators occur while the bracketing is as prescribed by Definition \ref{braidedbiequivalencedefinition}. Finally, insert associators and their inverses to left bracket the tree immediately before and after each commutator. By Lemma \ref{assoccoherence} and Type I rewrites we may now rewrite so that the tree remains left bracketed in between commutators.
\end{proof}
\begin{lemma}
The loop may be rewritten so that all commutators and inverse commutators occur at the very end of the loop, and the commutators occur on each tree in turn, from left to right. 
\end{lemma}
\begin{proof}
We first rewrite so that all inverse commutators are absorptions of negative braidings, in the sense that the emission of a positive braiding is immediately followed by its cancellation with a negative braiding:
\begin{center}
\raisebox{0.5cm}{{\Huge $[$}}
\includegraphics[width=1cm,height=1.5cm]{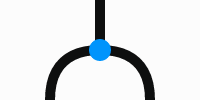}
\raisebox{0.5cm}{$\Rightarrow$} 
\includegraphics[width=1cm,height=1.5cm]{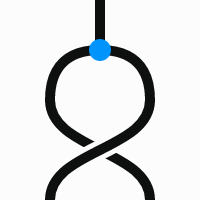} 
\raisebox{0.5cm}{{\Huge $]$}} 
\raisebox{0.5cm}{{\Large $=$}}
\raisebox{0.5cm}{{\Huge $[$}}
\includegraphics[width=1cm,height=1.5cm]{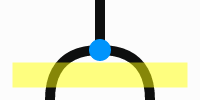} 
\raisebox{0.5cm}{$\Rightarrow$} 
\includegraphics[width=1cm,height=1.5cm]{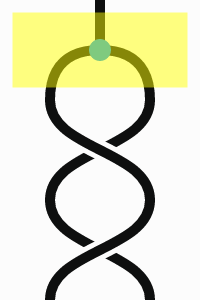} 
\raisebox{0.5cm}{$\Rightarrow$} 
\includegraphics[width=1cm,height=1.5cm]{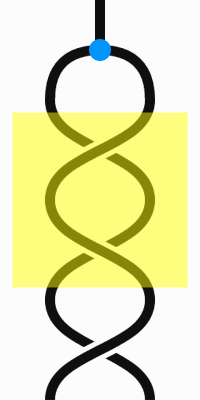}
\raisebox{0.5cm}{$\Rightarrow$} 
\includegraphics[width=1cm,height=1.5cm]{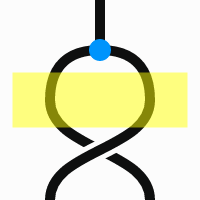}
\raisebox{0.5cm}{$\Rightarrow$} 
\includegraphics[width=1cm,height=1.5cm]{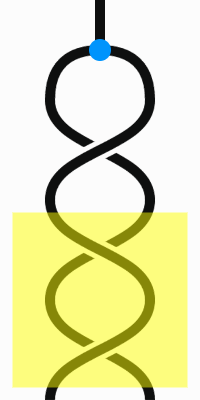} 
 \raisebox{0.5cm}{$\Rightarrow$} 
\includegraphics[width=1cm,height=1.5cm]{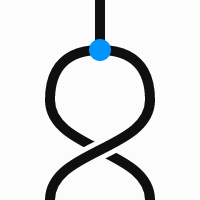} 
\raisebox{0.5cm}{{\Huge $]$}}
\end{center}
From now on we consider the inverse commutator as being the pair of emission and cancellation, which together form an absorption. We then use Type II rewrites to move the structural 2-cells out of the sequence of associators preceding and following the inverse commutator, so that the tree once again remains left bracketed until immediately before the inverse commutator, and then returns to the left bracketing immediately afterwards.

Now we may simply move all commutators and all inverse commutators to the end of the loop using Type I rewrites, beginning with the last, since there can be no obstruction. 
Finally, since the sets of 1-cells involved in commutators in different trees are always disjoint, we may now use Type I rewrites to ensure that the commutators occur on each tree in turn, from left to right. 
\end{proof}

The loop is now of the following form: some braids are created directly beneath the trees, and are then entirely absorbed at the end of the loop. We use Theorem \ref{gurskiosornocoherencethmextended} to rewrite the loop so that the braids are created underneath each tree in turn from left to right. Finally, we use Type I rewrites so that the braid beneath the leftmost tree is created, then absorbed, then the same happens for each tree in turn, from left to right. The loop is now in normal form. 

\subsection{Showing fullness from normal form}
Now that the loop is in normal form, we analyse the various cases in Table~\ref{tbl:pseudomonoidsintro} separately.

\begin{proposition}
The map from ${\bf \Delta}$ to the naked Gray monoid on the pseudomonoid computad defined in Section \ref{sec:biequivalences} is full on 2-morphisms.
\end{proposition}

\begin{proof} 
There is no braided structure in the naked Gray monoid on the pseudomonoid computad, so the braid created in the normal form is therefore trivial. All loops are therefore equal to the identity.
\end{proof}

\begin{proposition}\label{braidednoncommfullnessprop}
The map from ${\bf B\Delta}$ to the braided Gray monoid on the pseudomonoid computad defined in Section \ref{sec:biequivalences} is full on 2-morphisms.
\end{proposition}

\begin{proof}
There is no commutator in the signature; the braid created in the normal form must therefore be trivial, as there is no way for it to be absorbed. All loops are therefore equal to the identity.
\end{proof}

\begin{proposition}
The map from ${\bf S\Delta }$ to the symmetric Gray monoid on the pseudomonoid computad defined in Section \ref{sec:biequivalences} is full on 2-morphisms.
\end{proposition}

\begin{proof}
As for Proposition \ref{braidednoncommfullnessprop}.
\end{proof}
\noindent
For braided and symmetric pseudomonoids, the braid created in the normal form can be nontrivial. A priori, any braid can be created and absorbed in the normal form. However, we now show that two normal form loops where isotopic braids are created and absorbed are equal.

\begin{proposition}\label{braidsistopyinmoviesprop}
If the list of isotopy classes of braids absorbed by each tree in two normal form loops on a given 1-cell are the same, then the loops are equal.
\end{proposition}

\begin{proof}
We need to show that the group properties and braid equations are satisfied. That is, on each tree: 

\begin{enumerate}
\item Associativity: A loop where $(\sigma_i \sigma_j) \sigma_k$ is absorbed may be rewritten to a loop where $\sigma_i (\sigma_j \sigma_k)$ is absorbed.
\item Inverses: A loop where $\sigma_i \sigma_i^{-1}$ is absorbed may be rewritten to a loop where nothing is absorbed.
\item A loop where $\sigma_i \sigma_{i+1} \sigma_i$ is absorbed may be rewritten to a loop where $\sigma_{i+1} \sigma_{i} \sigma_{i+1}$ is absorbed.
\item A loop where $\sigma_i \sigma_j$, for $|i-j|>1$, is absorbed may be rewritten to a loop where $\sigma_j \sigma_i$ is absorbed.
\end{enumerate}

We now show each of these in turn. 

\begin{enumerate}
\item Associativity: This is trivially satisfied.
\item Inverses: One may perform a rewrite which exchanges $\sigma_i \sigma_i^{-1}$ for a cancellation of the two braids by simply cancelling a commutator and its inverse, as follows. (Recall that $\sigma_i^{-1}$ is a pair of an emission and a cancellation.)
\begin{center}
\raisebox{0.5cm}{{\Huge $[$}}
\includegraphics[width=1cm,height=1.5cm]{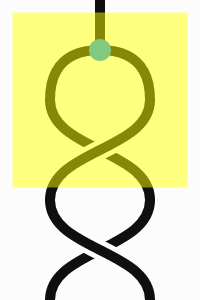}
\raisebox{0.5cm}{$\Rightarrow$} 
\includegraphics[width=1cm,height=1.5cm]{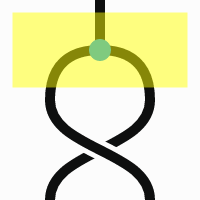} 
\raisebox{0.5cm}{$\Rightarrow$} 
\includegraphics[width=1cm,height=1.5cm]{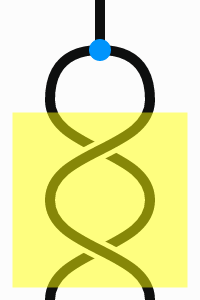}
\raisebox{0.5cm}{$\Rightarrow$} 
\includegraphics[width=1cm,height=1.5cm]{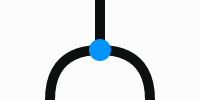} 
\raisebox{0.5cm}{{\Huge $]$}} 
\raisebox{0.5cm}{{\Large $=$}}
\raisebox{0.5cm}{{\Huge $[$}}
\includegraphics[width=1cm,height=1.5cm]{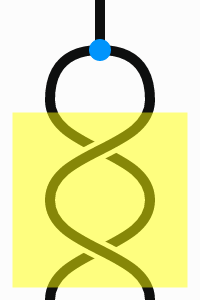} 
\raisebox{0.5cm}{$\Rightarrow$} 
\includegraphics[width=1cm,height=1.5cm]{inversesobeyedend} 
\raisebox{0.5cm}{{\Huge $]$}}
\end{center}
We then use Type I rewrites to move the cancellation before the commutators, restoring the loop to normal form $N$.
\item $\sigma_i \sigma_{i+1} \sigma_i = \sigma_{i+1} \sigma_{i} \sigma_{i+1}$: We perform a sequence of rewrites with the following effect:
\begin{center}
\raisebox{0.5cm}{{\Huge $[$}}
\includegraphics[width=1cm,height=1.5cm]{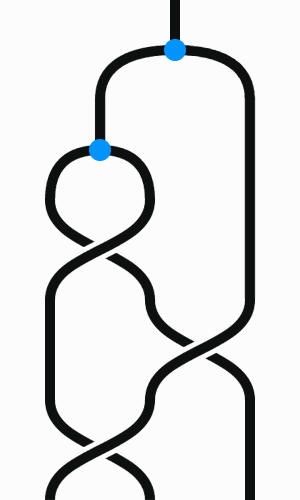}
\raisebox{0.5cm}{$\Rightarrow$} 
\includegraphics[width=1cm,height=1.5cm]{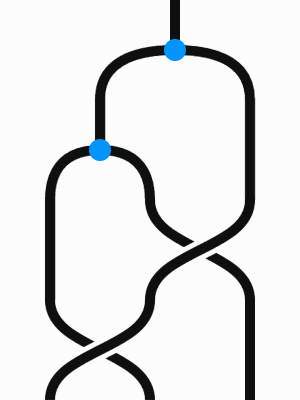} 
\raisebox{0.5cm}{$\Rightarrow$} 
\includegraphics[width=1cm,height=1.5cm]{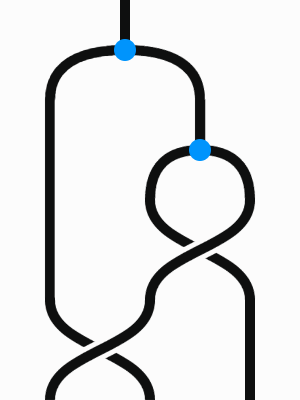}
\raisebox{0.5cm}{$\Rightarrow$} 
\includegraphics[width=1cm,height=1.5cm]{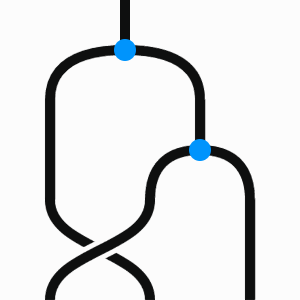}
\raisebox{0.5cm}{$\Rightarrow$} 
\includegraphics[width=1cm,height=1.5cm]{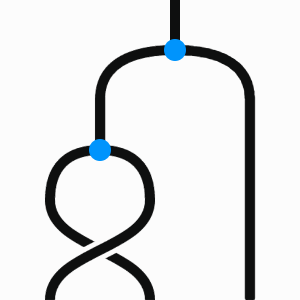}
\raisebox{0.5cm}{$\Rightarrow$} 
\includegraphics[width=1cm,height=1.5cm]{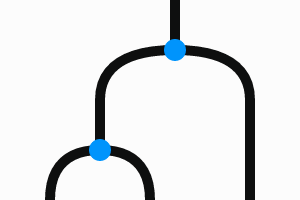} 
\raisebox{0.5cm}{{\Huge $]$}}
\raisebox{0.5cm}{{\Large $=$}}
\raisebox{0.5cm}{{\Huge $[$}}
\includegraphics[width=1cm,height=1.5cm]{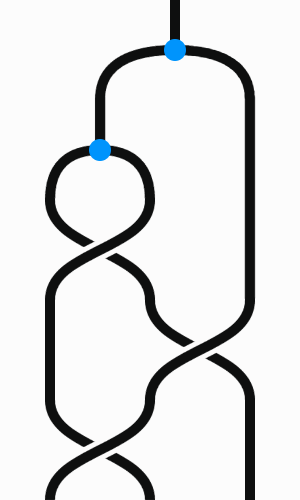}
\raisebox{0.5cm}{$\Rightarrow$} 
\includegraphics[width=1cm,height=1.5cm]{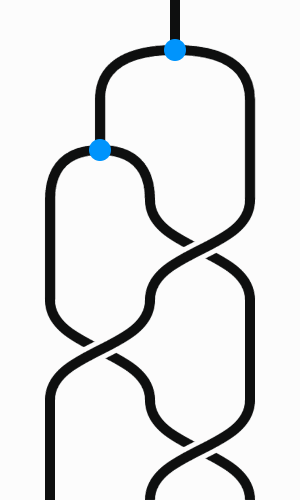} 
\raisebox{0.5cm}{$\Rightarrow$} 
\includegraphics[width=1cm,height=1.5cm]{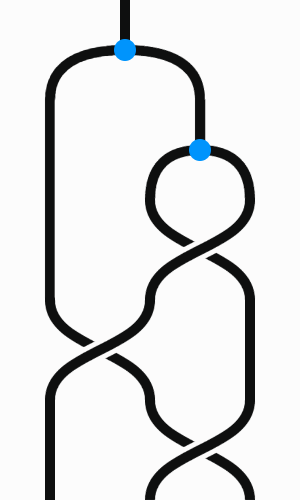}
\raisebox{0.5cm}{$\Rightarrow$} 
\includegraphics[width=1cm,height=1.5cm]{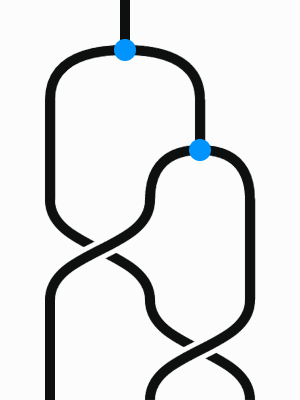}
\raisebox{0.5cm}{$\Rightarrow$} 
\includegraphics[width=1cm,height=1.5cm]{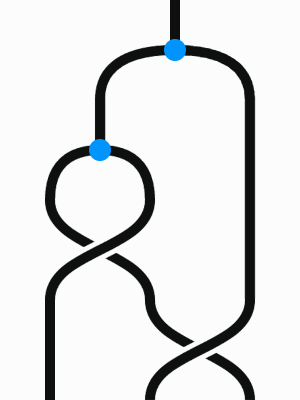}
\raisebox{0.5cm}{$\Rightarrow$} 
\includegraphics[width=1cm,height=1.5cm]{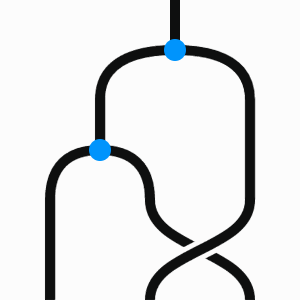} 
\raisebox{0.5cm}{$\Rightarrow$} 
\includegraphics[width=1cm,height=1.5cm]{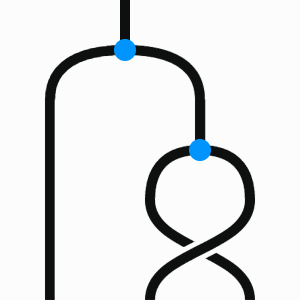}
\raisebox{0.5cm}{$\Rightarrow$} 
\includegraphics[width=1cm,height=1.5cm]{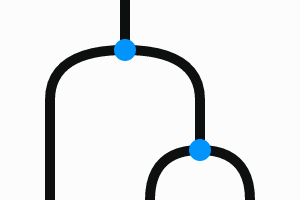}
\raisebox{0.5cm}{$\Rightarrow$} 
\includegraphics[width=1cm,height=1.5cm]{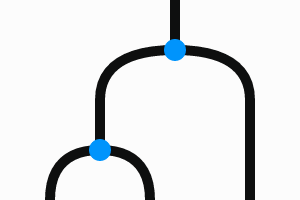} 
\raisebox{0.5cm}{{\Huge $]$}} 
\end{center}

The full rewrite sequence contains two uses of the hexagon equalities and is in the \emph{Globular} workspace as the 6-cell `Proposition~\ref{braidsistopyinmoviesprop} - Yang-Baxter Pf'. We then use Type I rewrites to move the pullthrough before all the commutators of that tree, restoring the loop to normal form $N$.

\item $\sigma_i \sigma_j = \sigma_j \sigma_i$ for $|i-j|>1$: Firstly, we use the unique series of associators with no interchangers between multiplication 1-cells (Lemma \ref{assoccoherence}) to rewrite the movie between the commutators so that it is of the form shown in Figure \ref{sigmaij=sigmajifig}. We then perform a sequence of rewrites with the following effect:
\begin{figure}
\centering
\begin{tikzpicture}[scale=0.4]
	\begin{pgfonlayer}{nodelayer}
		\node [style=none] (0) at (0, 3.5) {{ $\dots$}};
		\node [style=none] (1) at (1.75, 1.75) {};
		\node [style=none] (2) at (3, 3.5) {{ $\dots$}};
		\node [style=none] (3) at (4, 2) {};
		\node [style=none] (4) at (5, 2) {};
		\node [style=gn] (5) at (1.5, 3) {};
		\node [style=gn] (6) at (4.5, 4) {};
		\node [style=none] (7) at (1.5, 5) {};
		\node [style=none] (8) at (4.5, 5) {};
		\node [style=none] (9) at (0, 5) {};
		\node [style=none] (10) at (6, 5) {};
		\node [style=none] (11) at (6, 6) {};
		\node [style=none] (12) at (0, 6) {};
		\node [style=none] (13) at (6, 3.5) {{$\dots$}};
		\node [style=none] (14) at (3, 5.5) {Rest of tree};
		\node [style=none] (15) at (1, 1) {};
		\node [style=none] (16) at (1.5, 1.25) {};
		\node [style=none] (17) at (1.25, 1.5) {};
		\node [style=none] (18) at (2, -0) {};
		\node [style=none] (19) at (1, -0) {};
		\node [style=none] (20) at (4.75, 1) {};
		\node [style=none] (21) at (4, -0) {};
		\node [style=none] (22) at (4.25, 0.75) {};
		\node [style=none] (23) at (4.5, 0.5) {};
		\node [style=none] (24) at (5, -0) {};
	\end{pgfonlayer}
	\begin{pgfonlayer}{edgelayer}
		\draw [style=simple] (9.center) to (10.center);
		\draw [style=simple] (10.center) to (11.center);
		\draw [style=simple] (11.center) to (12.center);
		\draw [style=simple] (12.center) to (9.center);
		\draw [style=simple, bend right=45, looseness=1.00] (1.center) to (5);
		\draw [style=simple] (5) to (7.center);
		\draw [style=simple, bend left=15, looseness=1.00] (3.center) to (6);
		\draw [style=simple, in=-60, out=90, looseness=1.00] (4.center) to (6);
		\draw [style=simple] (6) to (8.center);
		\draw [style=simple, bend right, looseness=0.50] (1.center) to (15.center);
		\draw [style=simple, bend right=15, looseness=1.00] (15.center) to (19.center);
		\draw [style=simple, bend left=15, looseness=1.00] (4.center) to (20.center);
		\draw [style=simple] (20.center) to (21.center);
		\draw [style=simple] (23.center) to (24.center);
		\draw [style=simple, in=120, out=-135, looseness=1.00] (5) to (17.center);
		\draw [style=simple, in=90, out=-60, looseness=1.00] (16.center) to (18.center);
		\draw [style=simple, bend left=15, looseness=1.00] (22.center) to (3.center);
	\end{pgfonlayer}
\end{tikzpicture}
\raisebox{1cm}{$\Rightarrow$}
\begin{tikzpicture}[scale=0.4]
	\begin{pgfonlayer}{nodelayer}
		\node [style=none] (0) at (0, 3.5) {{ $\dots$}};
		\node [style=none] (1) at (3, 3.5) {{ $\dots$}};
		\node [style=none] (2) at (4, 2) {};
		\node [style=none] (3) at (5, 2) {};
		\node [style=gn] (4) at (1.5, 3) {};
		\node [style=gn] (5) at (4.5, 4) {};
		\node [style=none] (6) at (1.5, 5) {};
		\node [style=none] (7) at (4.5, 5) {};
		\node [style=none] (8) at (0, 5) {};
		\node [style=none] (9) at (6, 5) {};
		\node [style=none] (10) at (6, 6) {};
		\node [style=none] (11) at (0, 6) {};
		\node [style=none] (12) at (6, 3.5) {{ $\dots$}};
		\node [style=none] (13) at (3, 5.5) {Rest of tree};
		\node [style=none] (14) at (2, -0) {};
		\node [style=none] (15) at (1, -0) {};
		\node [style=none] (16) at (4.75, 1) {};
		\node [style=none] (17) at (4, -0) {};
		\node [style=none] (18) at (4.25, 0.75) {};
		\node [style=none] (19) at (4.5, 0.5) {};
		\node [style=none] (20) at (5, -0) {};
	\end{pgfonlayer}
	\begin{pgfonlayer}{edgelayer}
		\draw [style=simple] (8.center) to (9.center);
		\draw [style=simple] (9.center) to (10.center);
		\draw [style=simple] (10.center) to (11.center);
		\draw [style=simple] (11.center) to (8.center);
		\draw [style=simple] (4) to (6.center);
		\draw [style=simple, bend left=15, looseness=1.00] (2.center) to (5);
		\draw [style=simple, in=-60, out=90, looseness=1.00] (3.center) to (5);
		\draw [style=simple] (5) to (7.center);
		\draw [style=simple, bend left=15, looseness=1.00] (3.center) to (16.center);
		\draw [style=simple] (16.center) to (17.center);
		\draw [style=simple] (19.center) to (20.center);
		\draw [style=simple, bend left=15, looseness=1.00] (18.center) to (2.center);
		\draw [style=simple, bend left=15, looseness=1.00] (15.center) to (4);
		\draw [style=simple, bend right=15, looseness=0.75] (14.center) to (4);
	\end{pgfonlayer}
\end{tikzpicture}
\raisebox{1cm}{$\overset{\textrm{Associators}}{\Rightarrow}$}
\begin{tikzpicture}[scale=0.4]
	\begin{pgfonlayer}{nodelayer}
		\node [style=none] (0) at (0, 3.5) {{ $\dots$}};
		\node [style=none] (1) at (3, 3.5) {{ $\dots$}};
		\node [style=gn] (2) at (1.5, 4) {};
		\node [style=gn] (3) at (4.5, 3) {};
		\node [style=none] (4) at (1.5, 5) {};
		\node [style=none] (5) at (4.5, 5) {};
		\node [style=none] (6) at (0, 5) {};
		\node [style=none] (7) at (6, 5) {};
		\node [style=none] (8) at (6, 6) {};
		\node [style=none] (9) at (0, 6) {};
		\node [style=none] (10) at (6, 3.5) {{ $\dots$}};
		\node [style=none] (11) at (3, 5.5) {Rest of tree};
		\node [style=none] (12) at (2, -0) {};
		\node [style=none] (13) at (1, -0) {};
		\node [style=none] (14) at (4.75, 1) {};
		\node [style=none] (15) at (4, -0) {};
		\node [style=none] (16) at (4.25, 0.75) {};
		\node [style=none] (17) at (4.5, 0.5) {};
		\node [style=none] (18) at (5, -0) {};
	\end{pgfonlayer}
	\begin{pgfonlayer}{edgelayer}
		\draw [style=simple] (6.center) to (7.center);
		\draw [style=simple] (7.center) to (8.center);
		\draw [style=simple] (8.center) to (9.center);
		\draw [style=simple] (9.center) to (6.center);
		\draw [style=simple] (2) to (4.center);
		\draw [style=simple] (3) to (5.center);
		\draw [style=simple] (14.center) to (15.center);
		\draw [style=simple] (17.center) to (18.center);
		\draw [style=simple, bend left=15, looseness=1.00] (13.center) to (2);
		\draw [style=simple, bend right=15, looseness=0.75] (12.center) to (2);
		\draw [style=simple, bend right, looseness=0.75] (3) to (16.center);
		\draw [style=simple, bend left=60, looseness=0.75] (3) to (14.center);
	\end{pgfonlayer}
\end{tikzpicture}
\raisebox{1cm}{$\Rightarrow$}
\begin{tikzpicture}[scale=0.4]
	\begin{pgfonlayer}{nodelayer}
		\node [style=none] (0) at (0, 3.5) {{ $\dots$}};
		\node [style=none] (1) at (3, 3.5) {{ $\dots$}};
		\node [style=gn] (2) at (1.5, 4) {};
		\node [style=gn] (3) at (4.5, 3) {};
		\node [style=none] (4) at (1.5, 5) {};
		\node [style=none] (5) at (4.5, 5) {};
		\node [style=none] (6) at (0, 5) {};
		\node [style=none] (7) at (6, 5) {};
		\node [style=none] (8) at (6,6) {};
		\node [style=none] (9) at (0, 6) {};
		\node [style=none] (10) at (6, 3.5) {{ $\dots$}};
		\node [style=none] (11) at (3, 5.5) {Rest of tree};
		\node [style=none] (12) at (2, -0) {};
		\node [style=none] (13) at (1, -0) {};
		\node [style=none] (14) at (4, -0) {};
		\node [style=none] (15) at (5, -0) {};
	\end{pgfonlayer}
	\begin{pgfonlayer}{edgelayer}
		\draw [style=simple] (6.center) to (7.center);
		\draw [style=simple] (7.center) to (8.center);
		\draw [style=simple] (8.center) to (9.center);
		\draw [style=simple] (9.center) to (6.center);
		\draw [style=simple] (2) to (4.center);
		\draw [style=simple] (3) to (5.center);
		\draw [style=simple, bend left=15, looseness=1.00] (13.center) to (2);
		\draw [style=simple, bend right=15, looseness=0.75] (12.center) to (2);
		\draw [style=simple, bend right=15, looseness=0.75] (3) to (14.center);
		\draw [style=simple, bend right=15, looseness=0.75] (15.center) to (3);
	\end{pgfonlayer}
\end{tikzpicture}
\caption{}
\label{sigmaij=sigmajifig}
\end{figure}
\begin{center}
\raisebox{0.5cm}{{\Huge $[$}}
\includegraphics[width=.9cm,height=1.5cm]{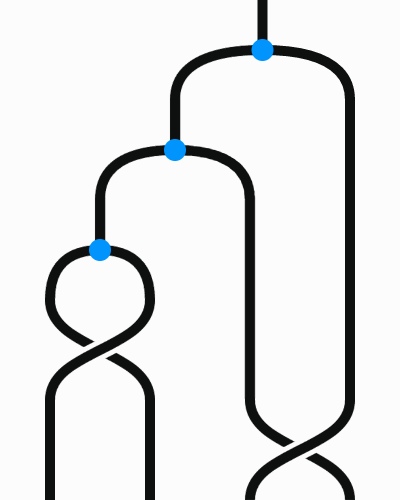}
\raisebox{0.5cm}{$\Rightarrow$} 
\includegraphics[width=.9cm,height=1.5cm]{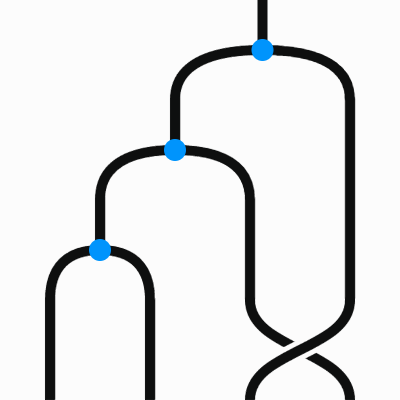} 
\raisebox{0.5cm}{$\Rightarrow$} 
\includegraphics[width=.9cm,height=1.5cm]{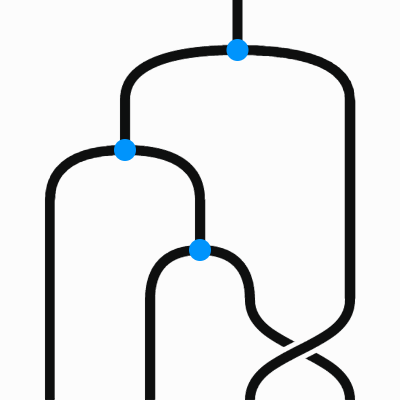}
\raisebox{0.5cm}{$\Rightarrow$} 
\includegraphics[width=.9cm,height=1.5cm]{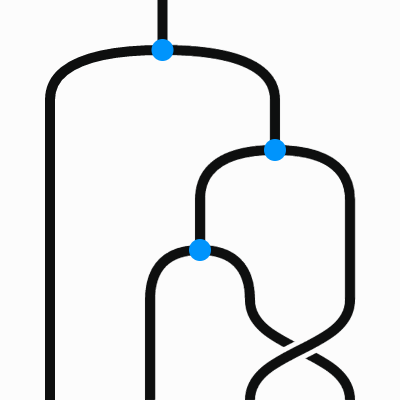}
\raisebox{0.5cm}{$\Rightarrow$} 
\includegraphics[width=.9cm,height=1.5cm]{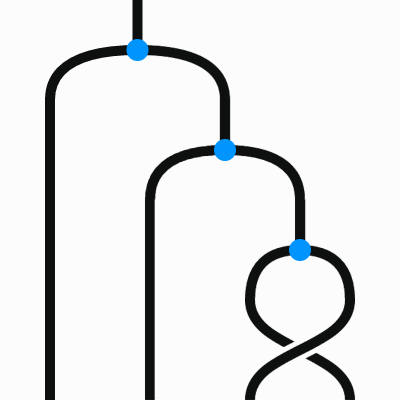}
\raisebox{0.5cm}{$\Rightarrow$} 
\includegraphics[width=.9cm,height=1.5cm]{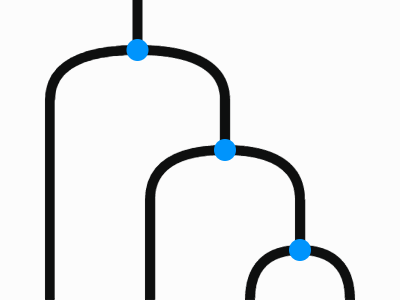}
\raisebox{0.5cm}{$\Rightarrow$} 
\includegraphics[width=.9cm,height=1.5cm]{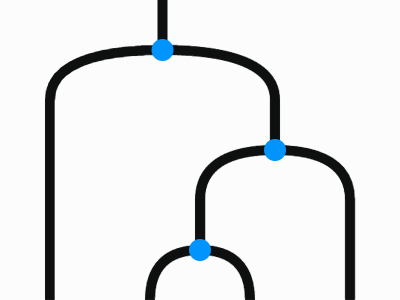}
\raisebox{0.5cm}{$\Rightarrow$} 
\includegraphics[width=.9cm,height=1.5cm]{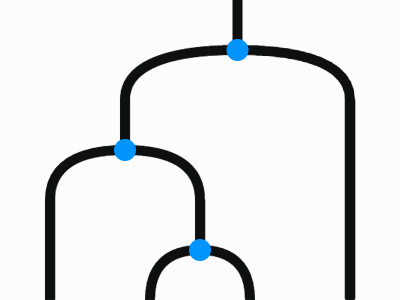}
\raisebox{0.5cm}{$\Rightarrow$} 
\includegraphics[width=.9cm,height=1.5cm]{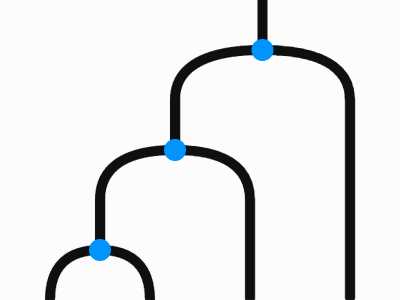} 
\raisebox{0.5cm}{{\Huge $]$}}\\
\raisebox{0.5cm}{{\Large $=$}}
\raisebox{0.5cm}{{\Huge $[$}}
\includegraphics[width=.9cm,height=1.5cm]{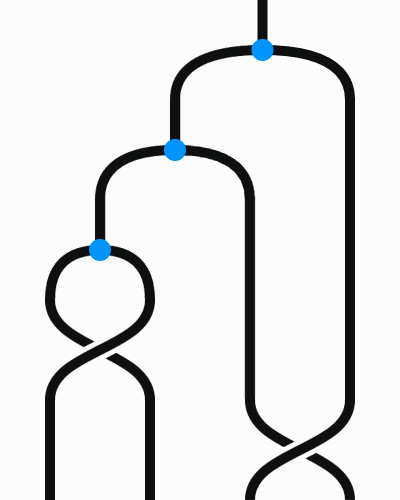}
\raisebox{0.5cm}{$\Rightarrow$} 
\includegraphics[width=.9cm,height=1.5cm]{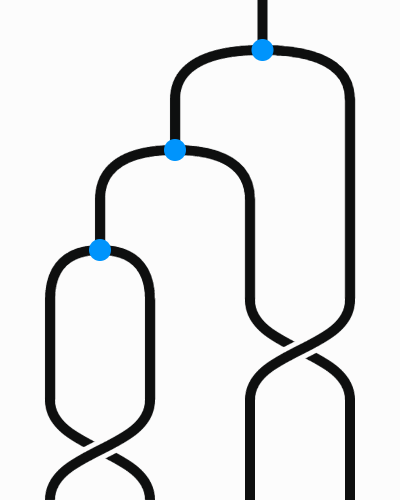} 
\raisebox{0.5cm}{$\Rightarrow$} 
\includegraphics[width=.9cm,height=1.5cm]{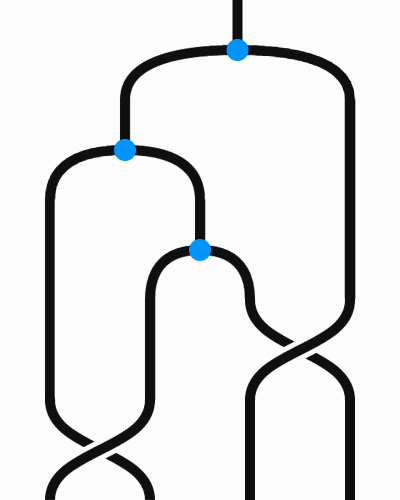}
\raisebox{0.5cm}{$\Rightarrow$} 
\includegraphics[width=.9cm,height=1.5cm]{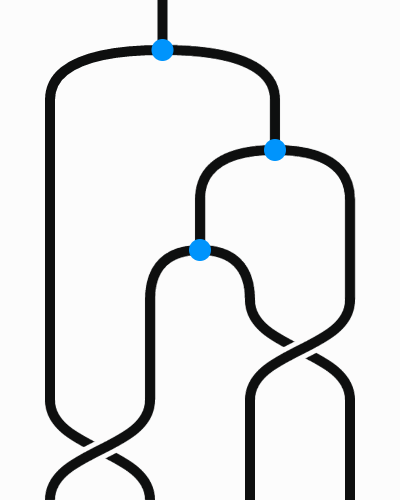}
\raisebox{0.5cm}{$\Rightarrow$} 
\includegraphics[width=.9cm,height=1.5cm]{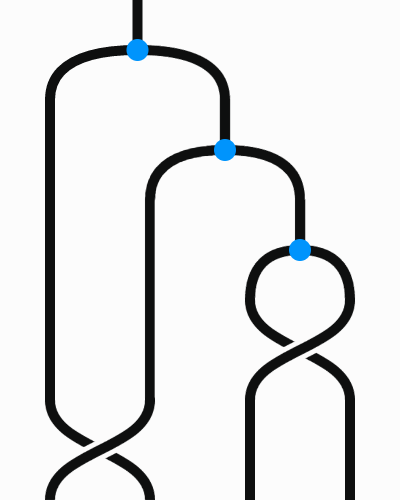}
\raisebox{0.5cm}{$\Rightarrow$} 
\includegraphics[width=.9cm,height=1.5cm]{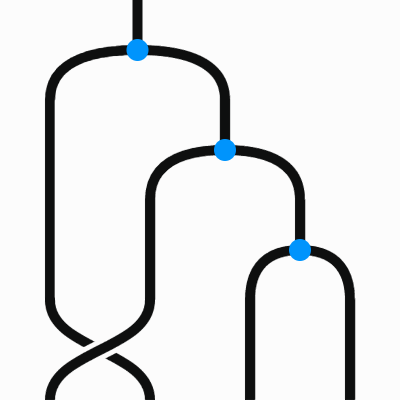}
\raisebox{0.5cm}{$\Rightarrow$} 
\includegraphics[width=.9cm,height=1.5cm]{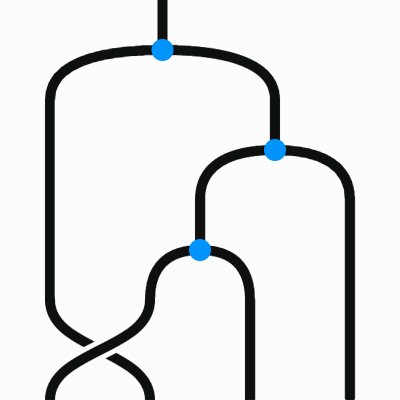}
\raisebox{0.5cm}{$\Rightarrow$} 
\includegraphics[width=.9cm,height=1.5cm]{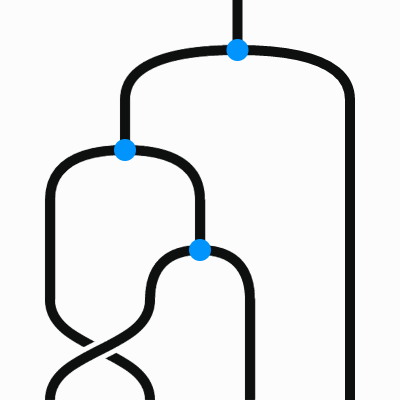}
\raisebox{0.5cm}{$\Rightarrow$} 
\includegraphics[width=.9cm,height=1.5cm]{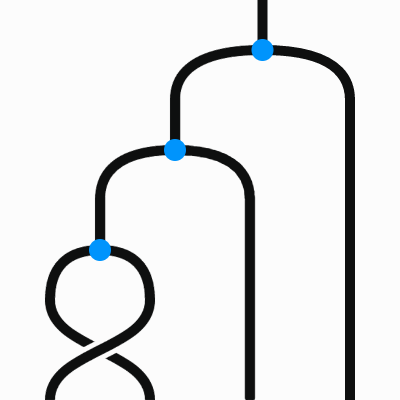} 
\raisebox{0.5cm}{$\Rightarrow$} 
\includegraphics[width=.9cm,height=1.5cm]{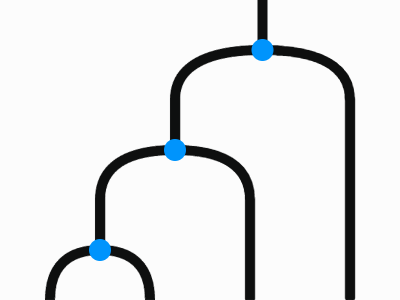}
\raisebox{0.5cm}{{\Huge $]$}}
\end{center}
The full sequence of rewrites is shown in the \emph{Globular} workspace as the 6-cell 'Proposition~\ref{braidsistopyinmoviesprop} - Swap relation Pf'. We then use Type I rewrites to move the interchanger before the commutators, restoring the loop to normal form $N$.
\end{enumerate}
\end{proof}

\begin{proposition}
The map from ${\bf B\Delta/{\raise.17ex\hbox{$\scriptstyle{\sim}$}}}$ to the braided Gray monoid on a braided pseudomonoid defined in Section \ref{sec:biequivalences} is full on 2-morphisms.
\end{proposition}

\begin{proof}
Only trivial braids can be created in a braided non-symmetric Gray monoid; Proposition~\ref{braidsistopyinmoviesprop} therefore implies that all the loops are contractible to the identity.
\end{proof}

\begin{proposition}
The map from ${\bf FS}^{br}$ to the symmetric Gray monoid on a braided pseudomonoid defined in Section \ref{sec:biequivalences} is full on 2-morphisms. 
\end{proposition}

\begin{proof}
The only braids which it is possible to create in a symmetric Gray monoid are pure, by Theorem~\ref{gurskiosornocoherencethm}. We can then use Proposition~\ref{braidsistopyinmoviesprop} to put the absorbed pure braids into Artin normal form, since all the braid equations are satisfied.
\end{proof}

\begin{proposition}
The map from ${\bf FS}$ to the symmetric  Gray monoid on a symmetric pseudomonoid defined in Section~\ref{sec:biequivalences} is full on 2-morphisms. 
\end{proposition}

\begin{proof}
We now have the additional relation $\sigma_i^2=\id$ in the computad for a symmetric pseudomonoid. With this additional relation, all pure braids are trivialised, and the loop can therefore be contracted.
\end{proof}

\section{Proof of functoriality}\label{sec:functorialityproof}

We use the definitions of homomorphisms of unbraided, braided and symmetric Gray monoids given in~\cite[Definitions 2, 14 and 16]{Day1997}. 

We first show that the maps we defined in Section~\ref{sec:biequivalences} are homomorphisms of Gray monoids. Throughout this section we use the letter $F$ for such a map.

\begin{proposition}
All maps defined in Section~\ref{sec:biequivalences} are homomorphisms of Gray monoids. 
\end{proposition}

\begin{proof}
Let $\Hom(A,B)_{S}$ and $\Hom(A,B)_{T}$ be the Hom-categories between objects $A$ and $B$ in the source and target Gray monoids respectively; $c^{A,B,C}_{S}$ and $c^{A,B,C}_{T}$ the composition bifunctors $\Hom(A,B) \times \Hom(B,C) \to \Hom(A,C)$ in each Gray monoid; and $I^A_{S}$ and $I^A_{T}$ the identity functors ${\bf 1} \to \Hom(A,A)$ in each Gray monoid, where ${\bf 1}$ is the trivial one-object category. 

First note that $F \circ I^A_{S} = I^A_{T}$, so the 2-cell relating the identity functors is trivial. We therefore only require a natural transformation $m$ of the following type: 

\begin{diagram}
\Hom(A,B)_{S} \times \Hom(B,C)_{S} & \rTo^{c^{A,B,C}_{S}} & \Hom(A,C)_{S} \\
\dTo^{F \times F} & \ruImplies^{m} & \dTo_{F} \\
\Hom(A,B)_{T} \times \Hom(B,C)_{T} & \rTo^{c^{A,B,C}_{T}} & \Hom(A,C)_{T} \\
\end{diagram}

The 2-cell $m_{f,g}$ relates the composition of diagrams in the combinatorial category to composition in the higher PRO, and can be defined as followed. First pull all attached unit nodes up to their connected multiplication node and remove them with a unit destruction operator. Now pull all trees in the bottom diagram, starting with the rightmost tree, up through any braids so they are directly beneath the trees in the top diagram. Finally, use associators to left bracket the combined trees, and put the braid in Artin normal form using the standard algorithm.

We now prove naturality of $m$. Let $f: \underline{m} \to \underline{n}$ and $g: \underline{n} \to \underline{o}$ be some 1-morphisms in the source. For every $\alpha:f \Rightarrow f, \beta: g \Rightarrow g$, the following diagram must commute:
\begin{diagram}
F(g \circ f) & \rTo^{F(\beta \circ_H \alpha)} & F(g \circ f) \\
\uTo^{m_{f,g}} & & \uTo_{m_{f,g}} \\
F(g) \circ F(f) & \rTo_{F(\beta) \circ_H F(\alpha)} &
F(g) \circ F(f)
\end{diagram}
Commutativity of this diagram follows from flipping the right and bottom arrows; we get a loop on a 1-morphism in the image of $F$ such that the isotopy class of the braid absorbed is trivial, which must be the identity by the results of Section \ref{sec:fullnessproof}.

We need one further diagram to commute for coherence of $m$. For all $f: \underline{m} \to \underline{n}$, $g: \underline{n} \to \underline{o}$ and $h:\underline{o} \to \underline{p}$:
\begin{diagram}
F(h) \circ F(g) \circ F(f) & \rTo^{\Id \circ_{H} m_{f,g}} & F(h) \circ F(g \circ f) \\
\dTo^{m_{g,h} \circ_{H} \Id} & & \dTo_{m_{g \circ f,h}} \\
F(h \circ g) \circ F(f) & \rTo_{m_{f,h \circ g}} & F(h \circ g \circ f)
\end{diagram}
Again, flipping the bottom and left arrows we get a loop on a 1-morphism in the image of $F$ involving no commutators, which is the identity by fullness and faithfulness of $F$; commutativity follows. 
\end{proof}
\noindent
We now show that the map is monoidal.

\begin{proposition}
All maps defined in Section \ref{sec:biequivalences} are strictly monoidal. 
\end{proposition}

\begin{proof}
Clear from the the definitions of Table~\ref{compositioningeneratedgraymonoid} and Section~\ref{sec:biequivalences}.
\end{proof}

\begin{proposition}
All maps from braided Gray monoids defined in Section~\ref{sec:biequivalences} are strictly braided. 
\end{proposition}
\begin{proof}
By the definitions of Section \ref{sec:biequivalences}, we have that $R_{F(X),F(Y)} = F(R_{X,Y})$.
\end{proof}
\begin{proposition}
All maps from symmetric Gray monoids defined in Section~\ref{sec:coherence} are strictly symmetric. 
\end{proposition}
\begin{proof}
By the definitions of Section~\ref{sec:biequivalences}, we have that $\sigma_{F(X),F(Y)} = F(\sigma_{X,Y})$.
\end{proof} 

\bibliographystyle{plainurl}
\bibliography{pseudomonoids}
\appendix
\appendixpage

\section{Semistrictness for Bar-Vicary braided monoidal bicategories}\label{app:semistrictnesshexagonators}

Here we imitate Schommer-Pries' semistrictness proof for quasistrict symmetric monoidal bicategories~\cite{Schommer-Pries2009}; the proofs are almost identical, so we do not repeat them here. First, since the braided monoidal bicategory monad is finitary and monotone, every braided monoidal bicategory is strictly biequivalent to a computadic one~\cite[Lemma 2.66, Corollary 2.67]{Schommer-Pries2009}. The coherence result of Gurski~\cite[Theorem 2.26]{Gurski2011} can be used to show that every computadic weak braided monoidal bicategory is biequivalent to the Crans braided monoidal bicategory on the corresponding Crans computad~\cite[Proposition 2.91]{Schommer-Pries2009}. We therefore need only prove that a computadic Crans braided monoidal bicategory is equivalent to the same category with trivial hexagonators. We use the following theorem.

\begin{theorem}[{Whitehead's theorem for braided monoidal bicategories~\cite[Theorem 2.25]{Schommer-Pries2009}}]\label{thm:whiteheadbraidedmon}
A braided monoidal homomorphism between braided monoidal bicategories is a braided monoidal biequivalence if and only if it is an equivalence of underlying bicategories.
\end{theorem}
\noindent
The semistrictness result can now be straightforwardly proven.
\begin{theorem}[Semistrictness for Bar-Vicary braided monoidal bicategories]
For any computadic Crans braided monoidal bicategory, the quotient homomorphism $\phi$ which identifies all braiding 1-cells with the corresponding `expanded' composite of braidings of generating 1-cells~\eqref{eq:trivialhexagonatorbraidings}, and sends all hexagonators to the identity, is a braided monoidal biequivalence. 
\end{theorem}
\begin{proof}

By Theorem~\ref{thm:whiteheadbraidedmon} we need only show that it is an equivalence of underlying bicategories. To do this, we construct an inverse homomorphism $H$ which is a section in the sense that $\phi \circ H = id$, and such that $H \circ \phi$ is naturally isomorphic to the identity.

Since $\phi$ is the identity on objects, we define $H$ to be the identity on objects. On 1-cells $f$, we define $H(f)$ to be the identical 1-cell with all braidings expanded.

Before defining $H$ on 2-cells, we specify the components of the invertible natural transformation $\nu: \Id \to H \circ \phi $. Since $H \circ \phi$ is the identity, we choose the 1-cell components of $\nu$ to be the identity 1-cells. We need therefore only define the 2-cell components. We note that any 1-cell differs only from its image under $H \circ \phi$ by repeated directed application of hexagonators. However, any order of directed application of hexagonators gives an equal 2-morphism, as can be shown using the polyhedra (2.4-2.7) in~\cite{Crans1998}. It follows that there is a single canonical isomorphism $f \to H ( \phi(f))$; all these isomorphisms together define the invertible natural transformation $\nu$. It is clear that this is compatible with composition using the structural equalities of a Gray monoid.

On 2-cells $\mu: f \to g$, we define $H$ by picking any 2-cell in the fibre $\phi^{-1}(\mu)$, and conjugating it by the 2-cell components of $\nu$, so that the source and the target 1-cells are expanded. Since there is only one composition of hexagonators taking the target of $f$ to the source of $g$, strict compositionality follows; we also have strict unitality.

Finally, we need to show pseudonaturality of $\nu: \Id \to H \circ \phi$; this follows immediately from  the fact that all parallel sequences of hexagonators are equal. 
\end{proof}

\section{Proofs from Section 2.4}\label{sec:proofsforgraymoncohappendix}

\begin{theorem}[Putting a 1-cell in TSNF]
Let $M$ be a clip in a computadic braided Gray monoid. Let $N$  be a non-structural generating 1-cell whose output is a single generating 0-cell. If no non-structural 2-cells occur on a rectangular subregion containing the output string of $N$ during $M$, then there exists a series of rewrites to put $N$ in TSNF.
\end{theorem}
\begin{proof}
Consider the first 2-cell in $M$ which acts on the output string of $N$. We rewrite $M$ to remove this 2-cell from the output string. We will now detail the rewrite case-by-case:
\begin{itemize}[leftmargin=*]
\item \emph{A braiding inverse-insert}. 
\begin{enumerate}[leftmargin=*]
\item Insert IPI immediately following the braiding inverse-insert so that the braiding inverse insert occurs, $N$ moves up through both created braidings, and then back down again.
\item Use Type I rewrites so that $N$ moves up to immediately beneath the site of the braiding inverse-insert, the braiding inverse-insert occurs, $N$ moves up through the created braidings, and then back down again.
\item Use PT-B so that $N$ moves up, the braiding occurs underneath $N$, and $N$ then moves back down again.
\end{enumerate}
See the 6-cell `Theorem A1 - Braiding inverse-insert Pf' in the \emph{Globular} workspace.

\item \emph{A braiding cancellation}. 
\begin{enumerate}[leftmargin=*]
\item Insert IPI immediately before the cancellation so that $N$ moves up through the cancelled braidings, moves back down again, and then the braidings are cancelled.
\item Use Type I rewrites so that $N$ moves up to and through the braidings, down again through both braidings, the braidings are cancelled, and then $N$ returns. 
\item Use the flip of PT-B so that $N$ moves up to and through the braidings, the braidings are cancelled, and then $N$ returns. 
\end{enumerate}
See the 6-cell `Theorem A1 - Braiding cancellation Pf' in the \emph{Globular} workspace.

\item \emph{An upwards pullthrough}. 
\begin{enumerate}[leftmargin=*]
\item Insert IPI immediately following the pullthrough so that the 1-cell pulls through, then $N$ then travels up far enough to interchange with the 1-cell, then $N$ returns. 
\item Use $(\rightarrow \otimes \rightarrow)$ so that $N$ travels up and pulls over the output string of the other 1-cell, then the 1-cell moves up just below $N$, then $N$ returns. 
\end{enumerate}
See the 6-cell `Theorem A1 - Upwards pullthrough Pf' in the \emph{Globular} workspace.

\item \emph{A downwards pullthrough}. 
\begin{enumerate}[leftmargin=*]
\item Insert IPI immediately prior to the pullthrough so that $N$ moves up to and through the braidings, then returns, then the other 1-cell pulls downwards through the braidings.
\item Use Type I rewrites so that $N$ moves up to and through the braidings, then $N$ moves back down through the braidings, then the other 1-cell pulls through the braidings, then $N$ returns.
\item Use the flip of $(\rightarrow \otimes \rightarrow)$ and Type II rewrites so that $N$ moves up and through the braidings, then the other 1-cell moves down through the braidings, then $N$ returns.
\end{enumerate}
See the 6-cell `Theorem A1 - Downwards pullthrough Pf' in the \emph{Globular} workspace.

\item \emph{An upwards interchanger}. 
\begin{enumerate}[leftmargin=*]
\item Insert IPI immediately following the interchanger so that  then $N$ pulls up and through the interchanged braiding, then returns.
\item Use Type I rewrites so that $N$ pulls up beneath the braiding, the braiding interchanges upwards, then $N$ interchanges upwards and pulls through upwards, then returns.
\item Use a Type III interchanger so that $N$ pulls through the braiding, then interchanges upwards, then the braiding interchanges upwards, then $N$ returns.
\end{enumerate}
See the 6-cell `Theorem A1 - Upwards interchanger Pf' in the \emph{Globular} workspace.

\item \emph{A downwards interchanger}.
\begin{enumerate}[leftmargin=*]
\item Insert IPI immediately prior to the downwards interchanger so that $N$ pulls up to and through the braiding, then $N$ returns, then the braiding interchanges downwards. 
\item Use Type I rewrites so that $N$ pulls up to and through the braiding, then back down through the braiding, then interchanges downwards, then the braiding interchanges downwards, then $N$ returns. 
\item Use a Type III rewrite so that $N$ pulls up to and through the braiding, the braiding interchanges downwards, then $N$ returns.
\end{enumerate}
See the 6-cell `Theorem A1 - Downwards interchanger Pf' in the \emph{Globular} workspace.
\end{itemize}
For a symmetric Gray monoid, there are two more possibilities:

\begin{itemize}[leftmargin=*]
\item \emph{A syllepsis}. 
\begin{enumerate}[leftmargin=*]
\item Insert IPI immediately before the syllepsis so that $N$ moves up through the braidings, then back down again, then the syllepsis occurs.
\item Use Type I rewrites so that $N$ moves up through the braidings, then pulls back through the braidings, the syllepsis occurs, then $N$ returns.
\item Use PT-SYL and Type II rewrites so that $N$ moves up through the braidings, the syllepsis occurs and then $N$ returns.
\end{enumerate}
See the 6-cell `Theorem A1 - Syllepsis Pf' in the \emph{Globular} workspace.

\item \emph{An inverse syllepsis}.
\begin{enumerate}[leftmargin=*]
\item Insert IPI immediately following the inverse syllepsis so that the inverse syllepsis occurs, $N$ moves up through the created braidings, and then returns.
\item Use Type I rewrites so that $N$ moves up immediately beneath the site of the inverse syllepsis, the inverse syllepsis occurs, $N$ moves up through the created braidings, and then returns.
\item Use the flip of PT-SYL and Type II rewrites so that $N$ moves up, the inverse syllepsis occurs beneath $N$, and then $N$ returns.
\end{enumerate}
See the 6-cell `Theorem A1 - Inverse syllepsis Pf' in the \emph{Globular} workspace.
\end{itemize}
Repeat until all 2-cells acting on the output string of $N$ have been removed.
\end{proof}
We now extend Theorems \ref{gurskicoherencethmbraids} and \ref{gurskiosornocoherencethm} using the TSNF procedure we just introduced.

\begin{theorem}[Extended coherence for computadic braided and symmetric Gray monoids]
Let $C$ be a computad for a braided or symmetric Gray monoid with no nonstructural generating 2-cells, whose generating 1-cells all have a single generating 0-cell as output. Then  all parallel 2-cells in the braided or symmetric Gray monoid generated from $C$ are equal.
\end{theorem}
\begin{proof}
Since every structural generating 2-cell is an isomorphism we need only show that any loop is contractible. Consider the source diagram. Order the non-braiding 1-cells in the source by height; call the highest $N_1$, the next highest $N_2$, etc.; up to $N_d$, where $d$ is the number of non-braiding 1-cells in the source. The target of our rewrites will be the following movie: $N_1$ interchanges and pulls through directly upwards to the top of the diagram.  $N_2$ then does the same, to the level just below $N_1$. $N_3$ then does the same; this continues until all $N_i$ are at the top of the diagram in the same height order as they were originally. Now, a loop of 2-cells occurs beneath all the $N_i$. Finally, $N_d$ interchanges and pulls through directly downwards back to its initial position, then $N_{d-1}$ does the same, etc.; this continues until  all nodes have returned to their original position. We may then use Theorem \ref{gurskiosornocoherencethm} to remove the loop beneath the $N_i$, since it features no 1-cells; then we cancel the upwards movement of the $N_i$ with the downwards movement, since they are inverse. 

Our series of rewrites will be inductive in the following sense: we will first rewrite the movie so that $N_1$ moves to the top, remains there throughout and then returns to the bottom again, while the other 1-cells interact beneath it in between its ascent and and its descent. We will then do exactly the same thing for the movie of interactions beneath $N_1$, and then for the movie of interactions beneath $N_2$, etc. It is clear that this approach will produce a movie in the desired form. Therefore, all we need to show is that we can rewrite the movie to one where $N_1$ moves to the top, remains there throughout and then returns to the bottom again. The procedure is as follows.

\begin{enumerate}[leftmargin=*]
\item Put $N_1$ in TSNF using Procedure \ref{tsnfprocedure}. 
\item Insert IPI so that the movie begins and ends with $N_1$ travelling straight up to the top of the diagram, then returning. This is possible since $N_1$ was the highest non-braiding 1-cell.
\item We now ensure that every 2-cell not involving $N_1$ occurs below $N_1$. Take the first 2-cell in the movie not involving $N_1$ and occuring above it. We have the following cases.
	\begin{itemize}[leftmargin=*]
	\item \emph{The 2-cell is a braiding inverse-insert or inverse syllepsis}. 
		\begin{enumerate}[leftmargin=*]
		\item Insert IPI so that $N_1$ rises to immediately beneath the rectangular subregion on which 2-cell occurs then returns. This is possible because $N_1$ is the highest non-braiding 1-cell. 
		\item Use Type I rewrites so that $N_1$ rises immediately beneath the rectangular subregion on which the 2-cell occurs, the 2-cell occurs, then $N_1$ returns to its starting position.
		\item Insert two interchangers and their inverses immediately following the 2-cell so that $N_1$ moves up past the created braidings and then returns. This is possible as $N_1$ is in TSNF, so the 2-cell cannot involve its output string. 
		\item Use a Type III rewrite to rewrite the movie to one where $N_1$ is pulled upwards, the 2-cell occurs directly beneath it, and then $N_1$ returns to its starting position. This is possible because $N_1$ is in TSNF, so the 2-cell acts on one side of the output string.
		\end{enumerate}
	\item \emph{Any other 2-cell}. 
		\begin{enumerate}[leftmargin=*]
		\item Insert IPI immediately before the 2-cell so that $N_1$ moves directly above the region acted on by the 2-cell, then returns, then the 2-cell occurs. This is possible because $N_1$ is the highest non-braiding 1-cell. 
		\item Use Type I rewrites so that $N_1$ moves directly above the region acted on by the 2-cell, returns to just below the region acted on by the 2-cell, then the 2-cell occurs. 
		\item Use a Type III rewrite so that $N_1$ is pulled upwards, the 2-cell occurs directly beneath it, then $N_1$ returns to its starting position. This is possible because $N_1$ is in TSNF, so the 2-cell acts on one side of the output string or is an interchanger.
		\end{enumerate}
\end{itemize}
Repeat until all 2-cells not involving $N_1$ occur below $N_1$.
\item We now remove all pullthroughs on $N_1$. Since $N_1$ is in TSNF, the number $b$ of braidings on the output string can only be changed by downwards and upwards pullthroughs of $N_1$. At the beginning and end of the clip $b=0$, so there are an equal number of upwards and downwards pullthroughs. Go to the first downwards pullthrough at which $b$ is maximised. Try to move this later in the movie using Type I rewrites. At some point this will be impossible. Since $N_1$ is in TSNF, no 2-cell can affect the braiding created by the pullthrough; there are therefore two possibilities for the obstruction.
\begin{itemize}[leftmargin=*]
\item \emph{The downwards pullthrough is immediately followed by an upwards pullthrough.} 
\begin{enumerate}[leftmargin=*]
\item Cancel the two pullthroughs, reducing the total number of pullthroughs on $N_1$ during the clip by $2$. 
\end{enumerate}
\item \emph{There is a chain of downward interchangers of $N_1$ immediately following the downwards pullthrough}. 
\begin{enumerate}[leftmargin=*]
\item  Go to the last downwards interchanger in this chain. Try to push it later in the movie using a Type II rewrite. If this succeeds, return to 4. If it fails, the obstruction cannot be a downwards pullthrough, since $b$ was maximal; nor can it be a 2-cell above $N_1$ not involving $N_1$, since we removed all these. It can therefore only be an upwards interchanger. Therefore, cancel both interchangers using a Type II rewrite. 
\end{enumerate}
Repeat this procedure to remove all pairs of downwards and upwards pullthroughs from the movie.
\end{itemize} 
\item There are now only interchangers on $N_1$. Go to the first interchanger of $N_1$; it will be the first of a chain of downwards interchangers. Go to the last downward interchanger in this chain and try to move it backwards in the movie using a Type I interchanger. If this is impossible, the only possible obstruction is an upwards interchanger, since there are no more pullthroughs on $N_1$ and no 2-cells above it not involving it; we may therefore cancel the two interchangers using a Type II rewrite. Repeat this procedure to remove all interchangers.
\end{enumerate}
The loop is now in the desired form and may be contracted.
\end{proof}

\end{document}